\documentclass[14pt,reqno,a4]{amsart}

\numberwithin{equation}{section}

\usepackage[latin1]{inputenc}
\usepackage[english]{babel}

\usepackage{amsmath,amsthm,amsfonts,latexsym,amssymb}
\usepackage[colorlinks]{hyperref}
\hypersetup{linkcolor=blue,citecolor=blue,filecolor=black,urlcolor=blue}
\usepackage{comment}
\usepackage{soul}

\usepackage[
  hmarginratio={1:1},     
  vmarginratio={1:1},     
  textwidth=15cm,        
  textheight=22cm,  
  heightrounded,          
]{geometry}

{ \theoremstyle{plain}
\newtheorem{theorem}{Theorem}[section]
\newtheorem{proposition}[theorem]{Proposition}
\newtheorem{lemma}[theorem]{Lemma}
\newtheorem{corollary}[theorem]{Corollary}
  \theoremstyle{remark}
\newtheorem{remark}[theorem]{Remark}
  \theoremstyle{definition}
\newtheorem{definition}[theorem]{Definition}

}

 \newcommand{\<}{\left\langle}
\renewcommand{\>}{\right\rangle}
 \renewcommand{\(}{\left(}
\renewcommand{\)}{\right)}
\renewcommand{\[}{\left[}
\renewcommand{\]}{\right]}
\newcommand{\eps}{\varepsilon}
\newcommand{\e}{\varepsilon}
\newcommand{\bm}{{\beta_-}}
\newcommand{\bp}{{\beta_+}}
\begin{document}

\title[Sign-changing solutions for critical equations with Hardy potential]{Sign-changing solutions for critical equations with Hardy potential}

\author{Pierpaolo Esposito}
\address{Pierpaolo Esposito, Universit\'a degli Studi Roma Tre,\newline \indent  Dipartimento di Matematica e Fisica,\newline \indent  L.go S. Leonardo Murialdo  1, 00146 Roma, Italy}
\email{esposito@mat.uniroma3.it}

\author{Nassif Ghoussoub}
\address{Nassif Ghoussoub, University of British Columbia,\newline \indent 
Department of Mathematics
4176-2207 Main Mall,
\newline \indent Vancouver BC V6T 1Z4,
Canada}
\email{nassif@math.ubc.ca}

\author[A. Pistoia]{Angela Pistoia}
\address{Angela Pistoia \newline \indent Universit\`a di Roma ``La Sapienza'' \newline \indent
Dipartimento di Metodi e Modelli Matematici,  via Antonio Scarpa 16, \newline 
\indent 00161 Roma, Italy}
\email{angela.pistoia@uniroma1.it}

\author{Giusi Vaira}
\address{Giusi Vaira, Sapienza Universit\`a degli Studi della Campania ``Luigi Vanvitelli'',\newline \indent  Dipartimento di Matematica e Fisica, 
viale Lincoln 5, \newline \indent 81100 Caserta, Italy }
\email{giusi.vaira@unicampania.it}

\date{\today}
\subjclass[2010]{35A15; 35J20; 35J47}

\keywords{critical  problem, Hardy potential, linear perturbation, blow-up point.}

 \begin{abstract}
We consider the following perturbed critical Dirichlet problem involving the Hardy-Schr\"odinger  operator 
 on a smooth bounded domain $\Omega \subset \mathbb{R}^N$, $N\geq 3$, with $0 \in \Omega$:
$$ 
\left\{ \begin{array}{ll}-\Delta u-\gamma \frac{u}{|x|^2}-\epsilon u=|u|^{\frac{4}{N-2}}u &\hbox{in }\Omega\\
u=0 & \hbox{on }\partial \Omega, \end{array}\right. $$
when $\epsilon>0$ is small and $\gamma< {(N-2)^2\over4}$. Setting
$ \gamma_j=  \frac{(N-2)^2}{4}\left(1-\frac{j(N-2+j)}{N-1}\right)\in(-\infty,0]$ for $j \in \mathbb{N},$
we show that 
if $\gamma\leq \frac{(N-2)^2}{4}-1$ and $\gamma \neq \gamma_j$ for any $j$, then  for small $\epsilon$, the above equation has a positive --non variational-- solution that develops a bubble at the origin. 
If moreover $\gamma<\frac{(N-2)^2}{4}-4,$ then for any integer $k \geq 2$, the equation has for small enough $\epsilon$, a sign-changing solution that develops into a superposition of  $k$  bubbles with alternating sign centered at the origin. The above result is optimal in the radial case, where the condition that $\gamma\neq \gamma_j$ is not necessary. Indeed, it is known that, if $\gamma > \frac{(N-2)^2}{4}-1$  and $\Omega$ is a ball $B$, then there is no radial positive solution for $\eps>0$ small. We complete the picture here by showing that, if $\gamma\geq \frac{(N-2)^2}{4}-4$, then the above  problem has no radial sign-changing solutions for $\eps>0$ small. These results  recover and improve what is known in the non-singular case, i.e., when $\gamma=0$. 

\end{abstract}

\maketitle

\maketitle

\tableofcontents
\section{Introduction}
We consider existence issues for the following Dirichlet problem:
\begin{equation} \label{1422}
\left\{ \begin{array}{ll}-\Delta u-\gamma \frac{u}{|x|^2}-\lambda u=|u|^{\frac{4}{N-2}}u &\hbox{in }\Omega\\
u=0 & \hbox{on }\partial \Omega, \end{array}\right. 
\end{equation}
where $\Omega \subset \mathbb{R}^N$, $N\geq 3$, is a smooth bounded domain with $0 \in \Omega$, $\gamma< {(N-2)^2\over4}$ and $\lambda \in \mathbb{R}$. Problem \eqref{1422} is the Euler-Lagrange equation of the following action functional
$$J_{\lambda}(u)=\frac{1}{2}\int_\Omega |\nabla u|^2-\frac{\gamma}{2} \int_\Omega \frac{u^2}{|x|^2}-\frac{\lambda}{2}  \int_\Omega u^2-\frac{N-2}{2N} \int_\Omega |u|^{\frac{2N}{N-2}},\quad u \in H_0^1(\Omega).$$
Since $\frac{(N-2)^2}{4}$ is the best constant in the classical Hardy inequality:
$$\frac{(N-2)^2}{4}= \inf \left\{ \int_\Omega |\nabla u|^2:\ u \in H_0^1(\Omega)\hbox{ s.t. } \int_\Omega \frac{u^2}{|x|^2}=1\right\}$$
see \cite{HLP}, we have that
\begin{equation}\label{1448}
\int_\Omega |\nabla u|^2-\gamma \int_\Omega \frac{u^2}{|x|^2}\geq \(1-\frac{4\gamma}{(N-2)^2}\) \int_\Omega |\nabla u|^2 \qquad \forall \ u \in H_0^1(\Omega).
\end{equation}
It is then useful to equip the Hilbert space $H^1_0(\Omega)$ with the inner product
$$ \langle u,v \rangle  =\int_\Omega\nabla u\nabla v-\gamma\int_\Omega \frac{uv}{|x|^2},$$
and the assumption $\gamma<\frac{(N-2)^2}{4}$ guarantees that the induced norm $\|\cdot\|$ is equivalent to the usual one in view of \eqref{1448}. 
Letting $L_\gamma=-\Delta-\frac{\gamma}{|x|^2}$  be the Hardy operator, let us denote by $0<\lambda_1\leq \lambda_2\leq \dots $ the eigenvalues of $L_\gamma$.

\medskip \noindent For $\lambda < \lambda_1$ positive solutions of \eqref{1422} can be found through the minimization problem:
$$S_{\gamma,\lambda}(\Omega)=\inf \left\{ \|u\|^2-\lambda \int_\Omega u^2: \ u \in H_0^1(\Omega) \hbox{ s.t. }\int_\Omega |u|^{\frac{2N}{N-2}}=1\right\}.$$
When $\lambda=0$, it is classical to see that $S_{\gamma,0}(\Omega)=S_{\gamma,0}(\mathbb{R}^N)$ and is never attained, the difficulty being here that \eqref{1422}  is doubly critical 
for the presence of the Hardy potential $\frac{1}{|x|^2}$ and the nonlinearity $|u|^{\frac{4}{N-2}}u$. Extremals for $S_{\gamma,0}(\mathbb{R}^N)$ exist for $\gamma \geq 0$ and have the form (up to a multiplicative constant)
\begin{equation}\label{bubblebis}
U_\mu(x)=\mu^{-\frac{N-2}{2}} U\(\frac{x}{\mu}\)=\frac{\alpha_N \mu^{\Gamma  }}{|x|^{\beta_-}(\mu^{\frac{4\Gamma}{N-2}}+|x|^{\frac{4\Gamma}{N-2}})^{\frac{N-2}{2}}},\quad  \mu>0, \end{equation}
where
\begin{equation}\label{bubble0bis} U (x)=\frac{\alpha_N}{|x|^{\beta^-}(1+|x|^{\frac{4\Gamma}{N-2}})^{\frac{N-2}{2}}}=
\frac{\alpha_N}{\(|x|^{\frac2{N-2}\beta^-}+|x|^{\frac2{N-2}\beta^+}\)^{\frac{N-2}2}}
\end{equation} 
with 
\begin{equation}\label{betabis}
\Gamma=\sqrt{\frac{(N-2)^2}{4}-\gamma},\quad \beta_\pm=\frac{N-2}{2}\pm \Gamma, \quad \alpha_N=\[\frac{4 \Gamma^2 N}{N-2}\]^{\frac{N-2}{4}},
\end{equation}
see \cite{CaWa,ChCh,Ter}. For $\gamma<0$ the problem is even more difficult since $S_{\gamma,0}(\mathbb{R}^N)=S_{0,0}(\mathbb{R}^N)$ is not attained, even though \eqref{bubblebis} is still a family of positive solutions to  
\begin{equation} \label{1635bis}
-\Delta U-\gamma \frac{U}{|x|^2}=U^{\frac{N+2}{N-2}} \hbox{ in }\mathbb{R}^N\setminus \left\{0\right\}.
\end{equation}

\medskip \noindent As in the classical Br\'ezis-Nirenberg problem \cite{BrNi}, on a bounded domain $\Omega$ the presence of a linear perturbation with $0<\lambda<\lambda_1$ results in a symmetry breaking which is 
responsible for the existence of minimizers for $S_{\gamma,\lambda}(\Omega)$ \cite{GhRo,Jan,RuWi}. More precisely, a positive ground-state solution for \eqref{1422} is found when
\begin{itemize}
\item $\gamma\leq 0$ and either 
$$N=3 \hbox{ and the ``Robin'' function }R_{\gamma,\lambda}>0 \hbox{ somewhere }$$
or 
$$N\geq 4,\quad \lambda>|\gamma| \inf\left\{\frac{1}{|x|^2}: \ x \in \Omega \right\}$$
\item $0<\gamma\leq \frac{(N-2)^2}{4}-1$
\item  $\max\left\{0,\frac{(N-2)^2}{4}-1\right\}<\gamma <\frac{(N-2)^2}{4}$ and ``mass'' $m_{\gamma,\lambda}>0$.
\end{itemize}
The question has been completely settled in \cite{GhRo}, which we refer to for a precise definition of $R_{\gamma,\lambda}$ and $m_{\gamma,\lambda}$, and the ranges displayed above are essentially optimal for the attainability of $S_{\gamma,\lambda}(\Omega)$, see also the recent survey \cite{gr}. Notice that the cases $\gamma<0$ and $\gamma=0$, $N=3$ always require $\lambda$ to be sufficiently away from zero.

\medskip \noindent  By Pohozaev identity \cite{Poh} equation \eqref{1422} has no solution when $\lambda \leq 0$ on domains which are strictly starshaped w.r.t. $0$. Since  solutions of
\eqref{1422} can't have a given sign when $\lambda \geq \lambda_1$, to attack existence issues for general $\lambda$'s one needs to search for sign-changing 
solutions. We can summarize the available results in literature \cite{CaHa,CaPe,CaYa,Che,cz,fg} as:
\begin{itemize}
\item if $ 0\leq  \gamma <\frac{(N-2)^2}{4}-4$ there are infinitely many sign-changing solutions for all $\lambda>0$
\item if $\max\left\{0,\frac{(N-2)^2}{4}-4\right\} \leq \gamma <\frac{(N-2)^2}{4} -\frac{(N+2)^2}{N^2}$ there exists a sign-changing solution for all $\lambda\geq \lambda_1$
\item if $\max\left\{0,  \frac{(N-2)^2}{4}-\frac{(N+2)^2}{N^2} \right\} \leq \gamma \leq \frac{(N-2)^2}{4}-1$ there exists a sign-changing  solution for all $\lambda \in \displaystyle \bigcup_{k=1}^\infty (\lambda_k,\lambda_{k+1})$
\item if $\gamma \geq 0$ and $\frac{(N-2)^2}{4}-1< \gamma <\frac{(N-2)^2}{4}$ there exist $n_k$ sign-changing solutions for all $\lambda$ in a suitable left open neighborhood of $\lambda_k$, $k \geq 2$, where $n_k$ is the multiplicity of $\lambda_k$.
\end{itemize}
Assumption $\gamma \geq 0$ allows here to use $U_\mu$, which are extremals of $S_{\gamma,0}(\mathbb{R}^N)$, as an helpful family of test functions in a variational approach.

\medskip \noindent Hereafter, we restrict our attention to the regime $\lambda=\eps$, with $\eps >0$ small:
\begin{equation} \label{1422bis}
\left\{ \begin{array}{ll}-\Delta u-\gamma \frac{u}{|x|^2}-\eps u=|u|^{\frac{4}{N-2}}u &\hbox{in }\Omega\\
u=0 & \hbox{on }\partial \Omega. \end{array}\right. 
\end{equation}
When $\gamma=0$ $S_{0,\eps}(\Omega)$ is not achieved \cite{BrNi,Dru,Esp} for $N=3$, and \eqref{1422bis} in the ball $B=B_1(0)$ admits no positive solutions for $N=3$ \cite{BrNi} and no radial sign-changing solutions for $N=3,4,5,6$ \cite{AdYa,ABP}. In the singular case, a similar situation arises depending now on $\gamma$: $S_{\gamma,\eps}(\Omega)$ is not achieved 
\cite{GhRo} when either $\gamma<0$ or $\gamma>\frac{(N-2)^2}{4}-1 $, and \eqref{1422bis} in $B$ admits no radial positive solutions \cite{cl} for $\gamma>\frac{(N-2)^2}{4}-1$. Our first main result, along with Theorem \ref{torri} below, completes the picture in a radial setting:
\begin{theorem}\label{nonexistence} 
When $\gamma\geq \frac{(N-2)^2}{4}-4$ problem \eqref{1422bis} has no radial sign-changing solutions in $B$ for $\eps>0$ small.
\end{theorem}
\noindent Theorem \ref{nonexistence} is based on a fine asymptotic analysis combined with Pohozaev identities. In this way we also recover, see the precise statement in Corollary \ref{thm1211}, the results in \cite{AdYa,ABP} and \cite{cl} concerning the regular case $\gamma=0$ and the singular case $\gamma>\frac{(N-2)^2}{4}-1$, respectively. Moreover, when $\gamma<\frac{(N-2)^2}{4}-4$ the analysis shows that radial sign-changing solutions need to develop in a very precise way a bubble of alternating towers centered at $0$ as $\eps \to 0^+$, recovering and improving the discussion in \cite{Iac} concerning the asymptotics of radial least-energy sign-changing solutions in the regular case $\gamma=0$ when $N\geq 7$. Once the radial case is well understood, we can attack by  a perturbative approach the case of a general domain $\Omega$ leading to the following result, which is optimal in the radial case.
 \begin{theorem}\label{torri}  
 Let
 \begin{equation}\label{assgamma}
 \gamma_j=  \frac{(N-2)^2}{4}\left(1-\frac{j(N-2+j)}{N-1}\right)\in(-\infty,0]\ , \quad  j \in \mathbb{N}.
\end{equation}
 Assume  that either $\Omega$ is a general domain with $\gamma\not=\gamma_j$ for all $j \in \mathbb N$ or $\Omega$ is $j-$admissible (see Definition \ref{sym} and Remark \ref{adm}) with $\gamma =\gamma_j$ for some $j \in \mathbb N$.\\
i) Let  $\gamma\leq \frac{(N-2)^2}{4}-1.$
Then there exists $\eps_1>0$ such that for all $\eps \in (0, \eps_1)$ problem \eqref{1422bis} has a positive solution $u_\eps$ developing a bubble at the origin. \\
ii)  Let $\gamma<\frac{(N-2)^2}{4}-4.$ For any integer $k \geq 2$ there exists $\eps_k>0$ such that for all $\eps \in (0, \eps_k)$ problem \eqref{1422bis} has a sign-changing solution  $u_\eps$, which looks like the superposition of  $k$  bubbles with alternating sign centered at the origin. 
\end{theorem}
\noindent   Theorem \ref{torri}-(i) provides positive solutions of \eqref{1422bis} for $\gamma<0$ which are not minimizers for $S_{\gamma,\eps}(\Omega)$, exactly as $U_\mu$ are solutions of \eqref{1635bis} which are not extremals for $S_{\gamma,0}(\mathbb{R}^N)$. More generally, our result allows to consider the case $\gamma<0$ which cannot be dealt in a variational way when $\eps >0$ is small. When $ 0 \leq  \gamma <\frac{(N-2)^2}{4}-4$ the solutions we found likely coincide with the infinitely many ones found in \cite{CaYa,cz}. 

\medskip \noindent The paper is organized as follows. In Section \ref{sec1} we discuss the asymptotic behavior for radial solutions of problem \eqref{1422bis} in $B$ with $\epsilon \to 0^+$, establishing in particular the validity of Theorem \ref{nonexistence}. In Sections \ref{sec2} and \ref{sec3} we deduce Theorem \ref{torri} by developing a very delicate perturbative approach where a crucial splitting of the remainder term is performed, see \cite{IaVa,MPV} for related results. In the Appendix  \ref{sec4} we collect several technical estimates.

\section{Asymptotic analysis in the radial case: proof of Theorem \ref{nonexistence}} \label{sec1}
\noindent 
In this section we will consider the case when $\Omega$ is the unit ball $B$. From now on, for any function $u \in L^q(A)$, $1\leq q \leq +\infty$, we let $|u|_{q, A}=\left(\int_A |u|^qdx\right)^{1/q}$ and $|u|_q=|u|_{q,\Omega}$. We will denote by $c,C$ various positive constants which can vary from lines to lines.

\medskip \noindent Let $u \in H_0^1(B)$ be a radial solution of \eqref{1422}. The function
\begin{equation} \label{1209}
v(r)=(\frac{N-2}{2\Gamma})^{\frac{N-2}{2}} r^{\frac{N-2}{2}(\frac{N-2}{2\Gamma}-1)} u(r^{\frac{N-2}{2 \Gamma}})
\end{equation}
is in $H_0^1(B)$ and is a radial solution of
\begin{equation}\label{1633}
-\Delta v=|v|^{4\over N-2}v+\eps |x|^\alpha v  \hbox{ in }B \setminus \{0\},\ v=0 \hbox{ on } \partial B,\end{equation}
where $\alpha=\frac{N-2}{\Gamma}-2$ and $\eps=(\frac{N-2}{2\Gamma})^2  \lambda$. We have the following simple description of nodal regions:
\begin{lemma} \label{lem1243} Given $\alpha>-2$, any non-trivial radial solution $v \in H_0^1(B)$ of \eqref{1633} is in $C(\overline B) \cap C^2(\overline B \setminus \{0\})$ and, if $\epsilon>0$ and $v(0)>0$, there exist an integer $k=k(v) \geq 1$ and $R_0=r_1=0<R_1<r_2<\dots<R_{k-1}<r_k<R_k=r_{k+1}=1$ so that for all $j=1,\dots,k$
$$(-1)^{j-1} v>v(R_j)=0 \hbox{ in }(R_{j-1},R_j),\quad (-1)^j v'> v'(r_j)=0 \hbox{ in }(r_j,r_{j+1}),$$
with the convention $v'(0)=0$. Moreover, there exists $\eps_0>0$ small, independent on $v$, so that for all $0<\eps \leq \eps_0$ there holds
\begin{equation} \label{2145}
\int_{A}|v|^{\frac{2N}{N-2}} \geq ( \frac{S}{2})^{\frac{N}{2}}
\end{equation}
for any nodal region $A$ of $v$, where $S=S_{0,0}(\mathbb{R}^N)$ is the Sobolev constant.
\end{lemma}
\begin{proof} Since $\alpha>-2$, we have that  
\begin{equation} \label{1640}
|x|^\alpha \in L^p(B) \hbox{ for some }p>\frac{N}{2}.
\end{equation} 
Since by the Sobolev embedding theorem $v \in L^{\frac{2N}{N-2}}(B)$, for any $\eta>0$ we can decompose $|v|^{\frac{4}{N-2}}+\eps |x|^\alpha$ as $f_1+f_2$ with $|f_1|_{\frac{N}{2}}\leq \eta$ and $f_2 \in L^\infty(B)$ in view of \eqref{1640}. We can re-write \eqref{1633} as
$$v-(-\Delta)^{-1}(f_1 v)=(-\Delta)^{-1}(f_2v).$$
By elliptic regularity theory and the Sobolev embedding $W^{2,\frac{Ns}{N+2s}}(B) \hookrightarrow L^s(B)$ we have that
\begin{equation} \label{1046}
|(-\Delta)^{-1}(f_1 v)|_s \leq C \|(-\Delta)^{-1}(f_1 v)\|_{W^{2,\frac{Ns}{N+2s}}}\leq C  |f_1 v|_{\frac{Ns}{N+2s}} \leq C  \eta |v|_s
\end{equation}
in view of the H\"older's inequality  and $|f_1|_{ \frac{N}{2}} \leq \eta$. Equivalently $H:v \in L^s(B) \to (-\Delta)^{-1}(f_1v) \in L^s(B)$ has operatorial norm $\leq C \eta$, and then the operator $\hbox{Id}-H:L^s(B)\to L^s(B)$ is invertible for all $s>1$ and $\eta$ sufficiently small. Arguing as in \eqref{1046}, we have that 
$$|v|_{\frac{Ns}{N-2s}} \leq \|(\hbox{Id}-H)^{-1} \| |(-\Delta)^{-1}(f_2 v)|_{\frac{Ns}{N-2s}} \leq  C |f_2 v|_s \leq C |f_2|_\infty |v|_s$$
when $s<\frac{N}{2}$ and for all $q>1$
$$|v|_q \leq \|(\hbox{Id}-H)^{-1} \| |(-\Delta)^{-1}(f_2 v)|_q \leq  C |f_2 v|_s \leq C |f_2|_\infty |v|_s$$
when $s\geq \frac{N}{2}$. Starting from $v \in L^{\frac{2N}{N-2}}(B)$ we iteratively prove that $v \in L^s(B)$ for all $s>1$, and then $|v|^{\frac{4}{N-2}}v+\eps |x|^\alpha v \in L^{\frac{N+2p}{4}}(B) \cap L^s_{\hbox{loc}}(\overline B \setminus \{0\})$ for all $s>1$, where $p$ is given in \eqref{1640}. Since $\frac{N+2p}{4}>\frac{N}{2}$, by elliptic regularity theory we deduce that $v \in C(\overline B) \cap C^2(\overline B \setminus \{0\})$. Moreover, we claim that
\begin{equation} \label{1107}
\lim_{r \to 0}r^{N-1} v'(r)=0.
\end{equation}
Indeed, let us write equation \eqref{1633} in radial coordinates as 
\begin{equation} \label{1105}
-\frac{1}{r^{N-1}}(r^{N-1}v')'=|v|^{\frac{4}{N-2}}v+ \eps |x|^\alpha v \qquad r \in (0,1).
\end{equation}
Since $v$ is non-trivial, then $v(0)\not=0$ and then, by continuity of $v$, the R.H.S. in \eqref{1105} has a given sign near $0$. By \eqref{1105} we deduce that the function $r^{N-1}v'(r)$ is monotone in $r$ and then has limit as $r \to 0$: $\displaystyle \lim_{r \to 0} r^{N-1}v'(r)=l$. However, $l\not=0$ would imply a discontinuity of $v$ at $0$, and then \eqref{1107} is established.

\medskip \noindent Take $\epsilon>0$ and assume w.l.o.g. $v(0)>0$. Given $R$ so that $\displaystyle \lim_{r \to R} r^{N-1} v'(r)=0$, observe that the integration of \eqref{1105} in $(R,r)$ gives 
\begin{eqnarray}
\label{1649}
v'(r)=-\frac{
1}{r^{N-1}} \int_{R}^r s^{N-1}(|v|^{\frac{4}{N-2}}v+\epsilon s^\alpha v)ds 
\end{eqnarray}
for all $r >0$. Since $v(0)>0$ and $v'<0$ near $0$ in view of \eqref{1649} with $R=0$, let us define
$$R_1=\sup \{r \in (0,1):\ v>0 \hbox{ in }(R_0,r)\},\quad r_2=\sup\{r\in (0,1):v'<0 \hbox{ in }(r_1,r) \}.$$
If $R_1=1$, then $r_2=1$ and the choice $k=1$ completes the proof. If $R_1<1$, by \eqref{1649} with $R=0$ and $v(1)=0$ we deduce that $R_1<r_2<1$, $v'(r_2)=0$ and
$$v>v(R_1)=0 \hbox{ in }(R_0,R_1),\quad v'<v'(r_1)=0 \hbox{ in }(r_1,r_2).$$
In an iterative way, for $i \geq 2$ assume to have found $R_0=r_1=0<R_1<r_2<\dots<R_{i-1}<r_i<1$ so that $v'(r_i)=0$ and for all $j=1,\dots,i-1$
$$(-1)^{j-1} v>v(R_j)=0 \hbox{ in }(R_{j-1},R_j),\quad (-1)^j v'> v'(r_j)=0 \hbox{ in }(r_j,r_{j+1}).$$
Define
$$R_i=\sup \{r\in (0,1):\ (-1)^{i-1} v>0 \hbox{ in }(R_{i-1},r)\},\quad r_{i+1}=\sup\{r \in(0,1):(-1)^i v'>0 \hbox{ in }(r_i,r) \}.$$
Since $(-1)^{i-1}v'>0$ in $(r_{i-1},r_i)$ and $R_{i-1} \in (r_{i-1},r_i)$, we have that $r_i<R_i \leq 1$, and by \eqref{1649} with $R=r_i$ it follows that $(-1)^i v'>0$ in $(r_i, R_i]$. If $R_i=1$, then $r_{i+1}=1$ and the choice $k=i$ completes the proof. If $R_i<1$, the boundary condition $v(1)=0$ implies that $R_i<r_{i+1}<1$, which in turn leads to $v'(r_{i+1})=0$ and
$$(-1)^{i-1} v>v(R_i)=0 \hbox{ in }(R_{i-1},R_i),\quad (-1)^i v'> v'(r_i)=0 \hbox{ in }(r_i,r_{i+1}).$$
Such a process needs to stop after $k$ steps. Otherwise, we would find an increasing sequence $R_i$, $i \in \mathbb{N}$, so that $v(R_i)=v(R_{i+1})=0$. Letting $R=\displaystyle \lim_{i\to +\infty} R_i \in (0,1]$, we would have that $\displaystyle \lim_{i\to +\infty} r_i=R$ in view of $R_{i-1}<r_i<R_i$. Since $v \in C^2(\overline B \setminus \{0\})$, we would deduce that $v(R)=v'(R)=0$, and then by the uniqueness for the ODE $v=0$, a contradiction.

\medskip \noindent Finally, let us integrate \eqref{1633} against $v$ on a nodal region $A$ to get
\begin{eqnarray*}
S \left(\int_{A} |v|^{\frac{2N}{N-2}}\right)^{\frac{N-2}{N}} &\leq&  \int_{A} |\nabla v|^2=\int_{A} |v|^{\frac{2N}{N-2}}+ \epsilon \int_{A} |x|^\alpha v^2\\
&\leq& \int_{A} |v|^{\frac{2N}{N-2}}+ \epsilon | |x|^\alpha|_{\frac{N}{2}} \left(\int_{A} |v|^{\frac{2N}{N-2}}\right)^{\frac{N-2}{N}}
\end{eqnarray*}
thanks to the H\"older's inequality and to the embedding $\mathcal{D}^{1,2}(\mathbb{R}^N) \subset L^{\frac{2N}{N-2}}(\mathbb{R}^N)$  with Sobolev constant $S$. 
Setting $\eps_0=\frac{S}{2 | |x|^\alpha |_{\frac{N}{2}}}$, the validity of \eqref{2145} easily follows for all $0<\eps \leq \eps_0$.
\end{proof}

\noindent Let $v_n\in H_0^1(B)$ be a sequence of non-trivial radial solutions to \eqref{1633} with $\alpha>-2$. Up to a subsequence, we can assume that there exist $k\geq 1$ and sequences $R_0^n=r_1^n=0<R_1^n<r_2^n<\dots<R_{k-1}^n<r^n_k<R_k^n=r_{k+1}^{n+1}\leq 1$ so that for all $j=1,\dots,k$
\begin{equation}
(-1)^{j-1} v_n>v_n(R_j^n)=0 \quad \hbox{ in }(R_{j-1}^n,R_j^n),\quad (-1)^j v_n'> v_n'(r_j^n)=0 \hbox{ in }(r_j^n,r_{j+1}^n). \label{1148}
\end{equation}
Notice that such an assumption simply means that all the $v_n$'s have at least $k$ nodal regions. The case of positive solutions $v_n$ corresponds to take $k=1$ and $R_1^n=1$, whereas for sign-changing solutions we can always choose a subsequence with at least $k\geq 2$ nodal regions. Set $\delta_j^n=|v_n(r_j^n)|^{-\frac{2}{N-2}}$, where
\begin{equation} \label{1717}
|v_n|(r_j^n)=\max_{[R_{j-1}^n,R_j^n]} |v_n|.
\end{equation}
Blow-up phenomena for \eqref{1633} are described in terms of the limiting problem
\begin{equation} \label{1242}
-\Delta V=V^{\frac{N+2}{N-2}}\hbox{ in }\mathbb{R}^N,
\end{equation}
whose bounded solutions are completely classified \cite{CGS,GNN}. In particular, every radial positive and bounded solution of \eqref{1242} is given by
\begin{equation} \label{1243}
V_\delta(x)=\delta^{-\frac{N-2}{2}} V(\frac{x}{\delta})= \left(\frac{\delta}{\delta^2+a_N |x|^2}\right)^{\frac{N-2}{2}}
\end{equation}
for some $\delta>0$, where $a_N=\frac{1}{N(N-2)}$ and
\begin{equation} \label{Vdelta}
V(x)=\left(\frac{1}{1+a_N |x|^2}\right)^{\frac{N-2}{2}}.
\end{equation}
The asymptotic behavior of $v_n$ is described in the following main result:
\begin{theorem} \label{thm1210} As $n \to +\infty$ there hold
\begin{equation} \label{1139}
\frac{r_j^n}{\delta_j^n} \to 0, \quad \frac{R_j^n}{\delta_j^n} \to +\infty, \quad V_j^n(x)=(-1)^{j-1}   (\delta_j^n)^{\frac{N-2}{2}} v_n(\delta_j^n  x ) \to V \hbox{ in }C^1_{\hbox{loc}}(\mathbb{R}^N \setminus \{0\})
\end{equation}
for all $j=1,\dots,k$. Moreover, $\alpha \leq N-4$ if $k = 1$ and $\alpha<\frac{N-6}{2}$ if $k\geq 2$. If in addition
\begin{equation} \label{shrink}
R_{k-1}^n \to 0 \hbox{ and }R_k^n \to R_k>0
\end{equation}
as $n \to +\infty$, there hold $R_k^n=1$ and for all $j=1,\dots k-1$ 
\begin{eqnarray} 
&& \hspace{-0.6 cm}
 R_j^n \sim \left[ \frac{\int_{\mathbb{R}^N}  V^{\frac{N+2}{N-2}}}{(N-2)\omega_{N-1}}\right]^{\frac{1}{N-2}} \sqrt{\delta_j^n \delta_{j+1}^n}  \label{1316} \\
&& \hspace{-0.6 cm}\delta_j^n \sim \left[\frac{(\alpha+2)   \int_{\mathbb{R}^N} |x|^{\alpha}V^2 }{(N-2) \int_{\mathbb{R}^N} V^{\frac{N+2}{N-2}}}  \eps_n\right]^{\frac{(N-2)(\frac{N-2}{N-6-2\alpha})^{k-j}-(N-4-\alpha)}{(2+\alpha)(N-4-\alpha)}}  \left[\frac{ (N-2) \omega_{N-1} }{\int_{\mathbb{R}^N} V^{\frac{N+2}{N-2}} }\right]^{\frac{1}{N-4-\alpha}(\frac{N-2}{N-6-2\alpha})^{k-j}}  \nonumber \\
&& \hspace{-0.6 cm} \delta_k^n \sim \left[  \frac{(\alpha+2) \omega_{N-1} \int_{\mathbb{R}^N}  |x|^{\alpha} V^2}{(\int_{\mathbb{R}^N} V^{\frac{N+2}{N-2}})^2}   \eps_n \right]^{\frac{1}{N-4-\alpha}}\nonumber
\end{eqnarray}
as $n \to +\infty$ provided $\alpha<N-4$.
\end{theorem}
\noindent Asymptotics for radial least-energy sign-changing solutions of \eqref{1633} with $\alpha=0$ and $N\geq 7$ has been already considered in \cite{Iac} and corresponds to the case $k=2$. Here we develop the asymptotic analysis in a completely general way by refining the results in \cite{Iac} for $k=2$, by covering the situation $\alpha \not=0$ and including the case $k \geq 3$. Several new difficulties arise:
\begin{itemize}
\item in each nodal region $v_n$ might develop multiple bubbles, but the Pohozaev identity will show crucial to prevent the interaction between bubbles of same sign;
\item the limiting problem admits positive radial solutions also on annuli or complements of balls, but none of them can be limit of $V_j^n$, as we will prove by a matching condition on $v_n'(R_j^n)$ as computed from the left and the right;
\item the precise law of $\delta_j^n$ is prescribed by the Pohozaev identity in terms of $\eps_n$ and $R_j^n$, but the asymptotic behavior of $R_j^n$ has to be determined according to a tricky compatibility condition between $v_n'(R_j^n)$ and $v_n(r_j^n)$.
\end{itemize}

\medskip \noindent Given $\Gamma$ in \eqref{betabis}, let 
\begin{equation}\label{sigmaj}
\sigma_j=\frac12 \frac\Gamma{\Gamma-1}\left(\frac{\Gamma}{\Gamma-2}\right)^{j-1}-\frac 12.
\end{equation} 
For $\mu=[\sqrt{N(N-2)} \delta]^{\frac{N-2}{2 \Gamma}}$, notice that the solution $U_\mu$ of \eqref{1635bis} given by \eqref{bubblebis} corresponds through \eqref{1209} to the solution 
$V_\delta$ of \eqref{1242} given by \eqref{1243}. Setting $M_{k-j+1}^n=(R_j^n)^{\frac{2\Gamma}{N-2}}$ and $\mu_j^n=[\sqrt{N(N-2)} \delta_{k-j+1}^n]^{\frac{N-2}{2 \Gamma}}$, by Theorem \ref{thm1210} with $\alpha=\frac{N-2}{\Gamma}-2$ we deduce the following:
\begin{corollary} \label{thm1211}
Let $u_n$ be a sequence of radial solutions for \eqref{1422bis} in $B$ with $\eps_n \to 0^+$ as $n \to +\infty$. 
\begin{itemize}
\item[(i)] If $u_n$ are positive functions, then $\gamma \leq \frac{(N-2)^2}{4}-1$ and
\begin{eqnarray*}
\mu_1^n =d_1 \eps_n^{\sigma_1}(1+o(1)),\quad U_1^n(x)=(\mu_1^n)^{\frac{N-2}{2}} u_n(\mu_1^n  x ) \to U \hbox{ in }C^1_{\hbox{loc}}(\mathbb{R}^N \setminus \{0\})
\end{eqnarray*}
as $n \to +\infty$ when $\gamma <\frac{(N-2)^2}{4}-1$.
\item[(ii)]  If $u_n$ are sign-changing solutions, then $\gamma < \frac{(N-2)^2}{4}-4$.
\item[(iii)] If $u_n$ have precisely $k-1$ shrinking nodal regions with nodes  
$$0=M_{k+1}^n<M_k^n<\dots<M_2^n \to 0,\qquad M_1^n \to M_1 \in (0,1]$$
as $n \to +\infty$, then there exist $\mu_j^n>0$, $j=1,\dots,k$, so that as $n \to +\infty$:
\begin{eqnarray*}
\mu_j^n =d_j \eps_n^{\sigma_j}(1+o(1)),\quad  
U_j^n(x)=(\mu_j^n)^{\frac{N-2}{2}} u_n(\mu_j^n  x ) \to U \hbox{ in }C^1_{\hbox{loc}}(\mathbb{R}^N \setminus \{0\}) 
\end{eqnarray*}
for all $j=1,\dots,k$ and
\begin{eqnarray*}
M_1^n=1,\quad M_j^n =A (\mu_{j-1}^n \mu_j^n )^{\frac{2\Gamma}{(N-2)^2}}(1+o(1))
\end{eqnarray*}
 for all $j=2,\dots, k$.
\end{itemize}
Here $U$ is given in \eqref{bubble0bis} and $A,d_j>0$ are explicit constants.
\end{corollary}
\noindent Let us discuss first the behavior of $v_n$ in $(0,R_1^n)$. Notice that the function $V_1^n=(\delta_1^n)^{\frac{N-2}{2}} v_n(\delta_1^n  x )$ solves
$$\left\{ \begin{array}{ll} 
-\Delta V_1^n=(V_1^n)^{\frac{N+2}{N-2}}+\eps_n (\delta_1^n)^{2+\alpha}|x|^\alpha V_1^n &\hbox{in }B_{\frac{R_1^n}{\delta_1^n}}(0)\\
0<V_1^n \leq V_1^n(0)=1 &\hbox{in }B_{\frac{R_1^n}{\delta_1^n}}(0) \end{array} \right.$$
in view of 
\begin{equation} \label{max}
0<(-1)^{j-1} v_n  \leq (-1)^{j-1} v_n(r_j^n)=\frac{1}{(\delta_j^n)^{\frac{N-2}{2}}}\qquad \hbox{in }
(R_{j-1}^n,R_j^n),
\end{equation}
a simple re-writing of \eqref{1717} through \eqref{1148}. By elliptic estimates we deduce that $V_1^n$ is uniformly bounded in $C^{0,\gamma}_{\hbox{loc}}(\mathbb{R}^N) \cap C^{1,\gamma}_{\hbox{loc}}(\mathbb{R}^N \setminus \{0\}) $, $\gamma \in (0,1)$, in view of \eqref{1640}. By the Ascoli-Arzel\'a's Theorem and a diagonal process, we have that, up to a subsequence, $V_1^n \to V$ in $C_{\hbox{loc}}(\mathbb{R}^N) \cap C^1_{\hbox{loc}}(\mathbb{R}^N \setminus \{0\}) $, where $V$ solves
$$-\Delta V=V^{\frac{N+2}{N-2}} \hbox{ in }\mathbb{R}^N ,\qquad 0<V \leq V(0)=1 \hbox{ in }\mathbb{R}^n $$
and has the form \eqref{Vdelta} \cite{CGS,GNN}.
We have used that 
\begin{equation} \label{1108}
\frac{R_1^n}{\delta_1^n} \to +\infty
\end{equation} 
as $n \to +\infty$. Indeed, if $\frac{\delta_1^n}{R_1^n}$ were bounded away from zero, then $\tilde V_1^n(x)=(R_1^n)^{\frac{N-2}{2}} v_n(R_1^n x)$ would be uniformly bounded in $B$ in view of \eqref{max}. Since $\tilde V_1^n>0$ solves
$$-\Delta \tilde V_1^n=(\tilde V_1^n)^{\frac{N+2}{N-2}}+\eps_n (R_1^n)^{2+\alpha}|x|^\alpha \tilde V_1^n \hbox{ in }B,\qquad \tilde V_1^n=0 \hbox{ on }\partial B,$$
by elliptic estimates, as before, we would deduce that, up to a subsequence, $\tilde V_1^n \to \tilde V_1$ in $C(\overline B) \cap C^1_{\hbox{loc}}(\overline B \setminus \{0\}) $, where $\tilde V_1 \geq 0$ is a bounded solution of
$$-\Delta \tilde V_1=(\tilde V_1)^{\frac{N+2}{N-2}} \hbox{ in }B \setminus \{0\},\quad \tilde V_1=0 \hbox{ on }\partial B.$$ 
Let us recall the Pohozaev identity \cite{Poh} in a radial form: given a solution $v$ of \eqref{1633} and a radial domain $A \subset B$, multiply \eqref{1633} by $\langle x,\nabla v \rangle= |x| v'$ and integrate in $A$ to get
\begin{equation} \label{17577}
(\alpha+2) \eps  \int_A |x|^{\alpha} v^2=\int_{\partial A} \left[ (v')^2
+\frac{N-2}{|x|} v v'+\frac{N-2}{N} |v|^{\frac{2N}{N-2}}+ \eps |x|^{\alpha} v^2 \right] \langle x,\nu \rangle.
\end{equation}
Since $0$ is a removable singularity in view of $\tilde V_1 \in L^\infty(\{0\})$, by \eqref{17577} with $\epsilon=0$ on $A=B$ we would get that $\tilde V_1=0$ and then
$$\int_ B |\tilde V_1^n|^{\frac{2N}{N-2}} \to 0$$
as $n \to +\infty$, in contradiction with \eqref{2145} in view of $\eps_n (R_1^n)^{2+\alpha}\to 0$ as $n \to +\infty$. 

\medskip \noindent We aim to show that there is no superposition of bubbles of same sign in $[0,R_1^n]$.  Interaction between bubbles of same sign can be ruled out by the Pohozaev identity \eqref{17577}. Letting
\begin{equation} \label{J}
J=\{ j=1,\dots,k: \ \eqref{1139} \hbox{ holds} \},
\end{equation}
notice that $1 \in J$ according to \eqref{1108}. We have the following general result:
\begin{proposition} \label{prop1807}
There exists $C>0$ so that
\begin{equation} \label{13044}
|v_n| \leq C V_{\delta_j^n} \qquad \hbox{in }[R_{j-1}^n,R_j^n]
\end{equation}
for all $j \in J$, where $V_\delta$ is given by \eqref{1243}.
\end{proposition}
\begin{proof}
The presence of other bubbles in $[R_{j-1}^n,R_j^n]$ can be detected by the behavior of $r^{\frac{N-2}{2}}v_n(r)$. Notice that the function
$r^{\frac{N-2}{2}}V(r)=(\frac{r}{1+a_N r^2})^{\frac{N-2}{2}}$ satisfies
\begin{equation} \label{1119b}
r^{\frac{N-2}{2}}V \Big|_{r=a_N^{-\frac{1}{2}}}=(\frac{N(N-2)}{4})^{\frac{N-2}{4}}, \quad \lim_{r \to +\infty} r^{\frac{N-2}{2}}V(r)=0
\end{equation}
and\begin{equation}\label{1120}
[r^{\frac{N-2}{2}} V(r)]'=\frac{N-2}{2} \frac{r^{\frac{N-4}{2}}(1-a_N r^2)}{(1+a_N r^2)^{\frac{N}{2}}}<0 \hbox{ in } (a_N^{-\frac{1}{2}},+\infty).
\end{equation}
Thanks to \eqref{1119b} let us fix $M>a_N^{-\frac{1}{2}}$ so that
\begin{equation} \label{1033}
M^{\frac{N-2}{2}} V(M)= \min \{  [\frac{N(N-2)}{16}]^{\frac{N-2}{4}},  [\frac{(N-2)^2(N+1)}{2(N+2)^2}]^{\frac{N-2}{4}}\}.
\end{equation}	
We claim that for $n$ large
\begin{equation} \label{1708}
(-1)^{j-1} [r^{\frac{N-2}{2}}v_n]'<0 \hbox{ in }[M \delta_j^n, R_j^n].
\end{equation} 
Indeed, if \eqref{1708} were not true, we could find $M_n \in [M\delta_j^n,R_j^n]$ so that
\begin{equation} \label{0943}
(-1)^{j-1}[r^{\frac{N-2}{2}}v_n]' <[r^{\frac{N-2}{2}}v_n]'(M_n)=0 \hbox{ in } [M\delta_j^n,M_n),\qquad \frac{M_n}{\delta_j^n}\to 0 \hbox{ as }n \to +\infty,
\end{equation} 
as it follows by \eqref{1120} and
\begin{eqnarray*}
(-1)^{j-1} \delta_j^n [r^{\frac{N-2}{2}}v_n]' (r \delta_j^n)= [r^{\frac{N-2}{2}} V_j^n]' \to [r^{\frac{N-2}{2}} V]'
\end{eqnarray*}
locally uniformly in $(0,+\infty)$ as $n \to +\infty$ in view of \eqref{1139}. By \eqref{17577} applied to $v_n$ on $A=B_{M_n}(0)$ we get that
\begin{equation} \label{1035}
[\frac{M_n v'_n(M_n)}{v_n(M_n)}]^2
+(N-2)  \frac{M_n v_n'(M_n)}{v_n(M_n)}+\frac{N-2}{N} M_n^2 |v_n(M_n)|^{\frac{4}{N-2}}+\eps_n M_n^{2+\alpha} >0
\end{equation}
in view of $\alpha>-2$. Since by \eqref{0943}
$$M_nv_n'(M_n)=-\frac{N-2}{2}v_n(M_n),$$
we deduce that
$$-\frac{(N-2)^2}{4}+\frac{N-2}{N} M_n^2 |v_n(M_n)|^{\frac{4}{N-2}}+\eps_n M_n^{2+\alpha}>0.$$
Since
$$(-1)^{j-1} M_n^{\frac{N-2}{2}} v_n(M_n) \leq (-1)^{j-1}(M \delta_j^n)^{\frac{N-2}{2}} v_n(M \delta_j^n) =
M^{\frac{N-2}{2}} V_j^n(M)\to M^{\frac{N-2}{2}} V(M)$$
as $n \to +\infty$ in view of \eqref{1139} and \eqref{0943}, by \eqref{1033} we deduce that
\begin{eqnarray*}
-\frac{(N-2)^2}{4}+\frac{N-2}{N} M_n^2 |v_n(M_n)|^{\frac{4}{N-2}}+\eps_n M_n^{2+\alpha}\leq -\frac{(N-2)^2}{8}+\eps_n<0 
\end{eqnarray*}
for $n$ large, a contradiction with \eqref{1035}. The claim \eqref{1708} is established.

\medskip \noindent Once \eqref{1708} is established, we can prove the validity of \eqref{13044}. First, since $(-1)^{j-1}v_n$ is a positive solution of $L_n v_n=0$ in $[R_{j-1}^n,R_j^n]$, the operator $L_n=-\Delta -|v_n|^{\frac{4}{N-2}}-\eps_n |x|^\alpha$ satisfies the minimum principle in $[R_{j-1}^n,R_j^n]$, and we can compare $(-1)^{j-1} v_n$ with $\varphi_n=\frac{M^{\frac{(N-2)(N+1)}{N+2}}(\delta_j^n)^{\frac{N(N-2)}{2(N+2)}}}{r^{\frac{(N-2)(N+1)}{N+2}}}$ in $[M\delta_j^n, R_j^n]$. Since
\begin{eqnarray*}
L_n \varphi_n &=& M^{\frac{(N-2)(N+1)}{N+2}} (\delta_j^n)^{\frac{N(N-2)}{2(N+2)}} r^{-\frac{N^2+N+2}{N+2}}\left[
\frac{(N-2)^2(N+1)}{(N+2)^2}- r^2 |v_n(r)|^{\frac{4}{N-2}}-\eps_n r^{2+\alpha} \right]\\
&\geq &M^{\frac{(N-2)(N+1)}{N+2}} (\delta_j^n)^{\frac{N(N-2)}{2(N+2)}} r^{-\frac{N^2+N+2}{N+2}}\left[
\frac{(N-2)^2(N+1)}{(N+2)^2}- M^2 |V_j^n(M)|^{\frac{4}{N-2}}-\eps_n  \right]
\end{eqnarray*}
in $[M\delta_j^n,R_j^n]$ in view of \eqref{1708}, we have that $L_n \varphi_n>0$ in $[M\delta_j^n,R_j^n]$ for $n$ large in view of \eqref{1139} and \eqref{1033}. Since
$$(-1)^{j-1} v_n(M\delta_j^n) \leq \frac{1}{(\delta_j^n)^{\frac{N-2}{2}}} =\varphi_n(M\delta_j^n),\quad (-1)^{j-1} v_n(R_j^n)=0<\varphi_n(R_j^n)$$
in view of \eqref{max}, we have that 
$$|v_n| (r)=(-1)^{j-1}v_n(r) \leq   \frac{M^{\frac{(N-2)(N+1)}{N+2}}(\delta_j^n)^{\frac{N(N-2)}{2(N+2)}}}{r^{\frac{(N-2)(N+1)}{N+2}}} \quad \hbox{ in }[M\delta_j^n, R_j^n],$$
or equivalently
\begin{equation} \label{17077}
V_j^n (r)\leq   \frac{M^{\frac{(N-2)(N+1)}{N+2}}}{r^{\frac{(N-2)(N+1)}{N+2}}}
\quad \hbox{ in }[M, \frac{R_j^n}{\delta_j^n}].
\end{equation}
By \eqref{1649} with $R=r_j^n$ we get that in $[R_{j-1}^n,R_j^n]$
\begin{eqnarray}
(-1)^j v_n'(r) &=&
\frac{1}{r^{N-1}} \int_{r_j^n}^r s^{N-1}(|v_n|^{\frac{N+2}{N-2}}+\epsilon_n s^\alpha |v_n|)ds \nonumber \\
&=&
\frac{(\delta_j^n)^{\frac{N-2}{2}}}{r^{N-1}} \int_{\frac{r_j^n}{\delta_j^n}}^{\frac{r}{\delta_j^n}} s^{N-1}
(V_j^n)^{\frac{N+2}{N-2}}
+\frac{\epsilon_n}{r^{N-1}} \int_{r_j^n}^r s^{N-1+\alpha} |v_n| ds.  \label{20222}
\end{eqnarray}
Inserting \eqref{17077} into \eqref{20222} we deduce that
\begin{eqnarray*}
|v_n'(r)| &\leq&  \frac{(\delta_j^n)^{\frac{N-2}{2}}}{r^{N-1}} \left[
\frac{M^N}{N}+M^{N+1} \int_M^\infty \frac{1}{s^2}\right]\\
&&+\frac{ \epsilon_n}{r^{N-1}} \left[
\frac{M^{N+\alpha}}{N+\alpha} (\delta_j^n)^{\alpha+\frac{N+2}{2}}+
\frac{1}{\alpha+2} \sup_{[M\delta_j^n,R_j^n]}  r^{N-2}|v_n|(r) \right] \\
&\leq & \frac{C}{r^{N-1}} [ (\delta_j^n)^{\frac{N-2}{2}}+  \epsilon_n 
\sup_{[M\delta_j^n,R_j^n]}  r^{N-2}|v_n|(r)]
\end{eqnarray*}
for $M \delta_j^n\leq r \leq R_j^n$ in view of \eqref{max} and $\alpha+\frac{N+2}{2}>\frac{N-2}{2}$. Integrating in $[r,R_j^n]$ we get that
\begin{eqnarray*}
r^{N-2} |v_n(r)| \leq r^{N-2} \int_r^{R_j^n} |v_n'| \leq C (\delta_j^n)^{\frac{N-2}{2}}
\end{eqnarray*}
in $[M \delta_j^n,R_j^n]$, and then
\begin{eqnarray} \label{21033}
|v_n| (r) \leq  C \frac{(\delta_j^n)^{\frac{N-2}{2}}}{r^{N-2}} \leq C V_{\delta_j^n} \qquad \hbox{in }[M \delta_j^n,R_j^n]
\end{eqnarray}
for $n$ large. By \eqref{max} there holds that
$$|v_n|\leq \frac{1}{(\delta_j^n)^{\frac{N-2}{2}}}\leq C V_{\delta_j^n}\quad \hbox{in }[R_{j-1}^n,M \delta_j^n] $$
which, combined with \eqref{21033}, completes the proof.
\end{proof}

\noindent Thanks to Proposition \ref{prop1807} we are now in position to establish Theorem \ref{thm1210}.

\medskip \noindent $Proof \: (of \: Theorem \:\ref{thm1210}).$ Let $j\in J$, $J$ given in \eqref{J}, so that Proposition \ref{prop1807} applies. By \eqref{max} and \eqref{21033} we deduce that
\begin{equation} \label{1119}
\eps_n \int_{r_j^n}^{R_j^n} s^{N-1+\alpha} |v_n|ds=O(\eps_n (\delta_j^n)^{\alpha+\frac{N+2}{2}}+\eps_n 
(\delta_j^n)^{\frac{N-2}{2}})=o((\delta_j^n)^{\frac{N-2}{2}})
\end{equation}
as $n \to +\infty$ in view of $\alpha+\frac{N+2}{2}>\frac{N-2}{2}$, and \eqref{13044} can be re-written as
\begin{equation} \label{1808}
|V_j^n| \leq C V \quad \hbox{in }[\frac{R_{j-1}^n}{\delta_j^n},\frac{R_j^n}{\delta_j^n}].
\end{equation}
Inserting \eqref{1119} into \eqref{20222}, by the Lebesgue's Theorem we have that
\begin{eqnarray}  \label{1811}
(-1)^j (\delta_j^n)^{-\frac{N-2}{2}}  (R_j^n)^{N-1} v_n'(R_j^n) \to \int_0^\infty s^{N-1} V^{\frac{N+2}{N-2}} 
\end{eqnarray} 
for all $j \in J$ as $n \to +\infty$, in view of $V_j^n \to V$ in $C_{\hbox{loc}}(\mathbb{R}^N \setminus \{0\})$ and \eqref{1808}. 

\medskip \noindent Since $1 \in J$, let us apply \eqref{17577} to $v_n$ on $B_{R_1^n}(0)$ if $j=1$ or on $B_{R_j^n}(0) \setminus B_{R_{j-1}^n}(0)$ if $j \geq 2$ with $j-1,\ j \in J$. As $n \to +\infty$ we get that
\begin{eqnarray} 
&& (\alpha+2) \eps_n \int_{R_{j-1}^n}^{R_j^n}  r^{N-1+\alpha} v_n^2 =(R_j^n)^N (v_n'(R_j^n))^2  -(R_{j-1}^n)^N (v_n'(R_{j-1}^n))^2 \nonumber \\
&&= \left(\int_0^\infty r^{N-1} V^{\frac{N+2}{N-2}} \right)^2 [(\frac{\delta_j^n}{R_j^n})^{N-2} (1+o(1)) -(\frac{\delta_{j-1}^n}{R_{j-1}^n})^{N-2} (1+o(1))] \label{1150}
\end{eqnarray}
in view of \eqref{1811}, with the convention $\frac{\delta_0^n}{R_0^n}=0$. The LHS above has the following asymptotic behavior: if $\alpha > N-4$ there holds
\begin{eqnarray} \label{1215}
\int_{R_{j-1}^n}^{R_j^n}  r^{N-1+\alpha} v_n^2 \leq  C^2 [N(N-2)]^{N-2}
(\delta_j^n)^{N-2} \int_{R_{j-1}^n}^{R_j^n}  r^{3+\alpha-N}=O((\delta_j^n)^{N-2})
\end{eqnarray}
in view of  \eqref{13044}; if $-2< \alpha\leq N-4$ there holds
\begin{eqnarray}
\int_{R_{j-1}^n}^{R_j^n}  r^{N-1+\alpha} v_n^2 &=&(\delta_j^n)^{2+\alpha} \int_{\frac{R_{j-1}^n}{\delta_j^n}}^{\frac{R_j^n}{\delta_j^n}}  r^{N-1+ \alpha} (V_j^n)^2 \nonumber \\
&=& \left\{ \begin{array}{ll} (\delta_j^n)^{2+\alpha} \int_0^{+\infty}  r^{N-1+ \alpha} V^2 (1+o(1)) &\hbox{if }\alpha<N-4\\
O( (\delta_j^n)^{N-2} |\log \frac{R_j^n}{\delta_j^n}|) &\hbox{if }\alpha=N-4 \end{array} \right.  \label{1232}
\end{eqnarray}
in view of \eqref{1139}, \eqref{1808} and the Lebesgue's Theorem. 

\medskip \noindent We have some useful properties to establish. 

\medskip \noindent {\bf \underline{$1^{\hbox{st}}$ Claim}}: We have that
\begin{equation} \label{1519bb}
j-1 \in J,\ R_{j-1}^n <1 \: \Rightarrow \: \max_{[R_{j-1}^n,R_j^n]} |v_n| \to +\infty \hbox{ as }n\to +\infty.
\end{equation}
Up to a subsequence, assume that $R_{j-1}^n \to R_{j-1}$ and $R_j^n \to R_j$ as $n \to +\infty$. If $\displaystyle \max_{[R_{j-1}^n,R_j^n]} |v_n| \leq C$, by $\eps_n \to 0$ as $n \to +\infty$, \eqref{2145} and elliptic estimates we deduce that $R_{j-1}<R_j$ and, up to a subsequence, $(-1)^{j-1} v_n \to v$ in $C^2_{\hbox{loc}}(A)$, $A=B_{R_j}(0)\setminus \overline{B_{R_{j-1}}(0)}$, as $n \to +\infty$, where $v>0$ is a bounded solution of
\begin{equation} \label{1537}
-\Delta v=v^{\frac{N+2}{N-2}} \hbox{ in }A,\quad v=0 \hbox{ on }\partial A \setminus \{0\}.
\end{equation}
We have that $R_{j-1}>0$, since otherwise $v$ would be a solution of \eqref{1537} in the whole $B_{R_j}(0)$, $0$ being a removable singularity, and then would vanish by the Pohozaev identity \eqref{17577}. Up to a subsequence, by elliptic estimates $\tilde v_n(r)=(-1)^{j-1} (R_{j-1}^n)^{\frac{N-2}{2}} v_n( r R_{j-1}^n) \to \tilde v$ in $C^2_{\hbox{loc}}(A)$, $A=B_{\frac{R_j}{R_{j-1}}}(0)\setminus B$, as $n \to +\infty$, where $\tilde v>0$ is a bounded solution of
$$-\Delta \tilde v=\tilde v^{\frac{N+2}{N-2}} \hbox{ in }A,\quad \tilde v=0 \hbox{ on }\partial A .$$
In particular, $ \tilde v_n'(1)=(-1)^{j-1} (R_{j-1}^n)^{\frac{N}{2}} v_n'(R_{j-1}^n) \to \tilde v'(1)>0$, in contradiction with \eqref{1811} when $j-1 \in J$ and $R_{j-1}^n \to R_{j-1}>0$ as $n \to +\infty$. Then \eqref{1519bb} is established and the Claim is proved. \qed

\medskip \noindent {\bf \underline{$2^{\hbox{nd}}$ Claim}}: We have that
\begin{equation} \label{1519}
j-1 \in J,\ R_{j-1}^n <1 \: \Rightarrow \: \sup \frac{r_j^n}{\delta_j^n} <+\infty.
\end{equation}
If  $\frac{r_j^n}{\delta_j^n} \to +\infty$ as $n \to +\infty$, then $j\geq 2$ and the function $\tilde V_j^n(r)= (-1)^{j-1} (\delta_j^n)^{\frac{N-2}{2}} v_n(r_j^n+\delta_j^n r)$ solves
$$ \left\{ \begin{array}{ll}
-(\tilde V_j^n)''-(N-1) \frac{\delta_j^n}{r_j^n+\delta_j^n r}(\tilde V_j^n)'=(\tilde V_j^n)^{\frac{N+2}{N-2}}+\eps_n (\delta_j^n)^2 (r_j^n+\delta_j^n r)^\alpha \tilde V_j^n &\hbox{in }I_n=\left(-\frac{r_j^n-R_{j-1}^n}{\delta_j^n}, \frac{R_j^n-r_j^n}{\delta_j^n} \right)\\
0<\tilde V_j^n \leq \tilde V_j^n(0)=1 &\hbox{in }I_n\\
\tilde V_j^n=0 &\hbox{on }\partial I_n \end{array} \right.$$
in view of \eqref{max}. Up to a subsequence, assume that
$$\frac{r_j^n-R_{j-1}^n}{\delta_j^n} \to L_1 \in [0,+\infty],\quad \frac{R_j^n-r_j^n}{\delta_j^n} \to L_2 \in [0,+\infty]$$
as $n \to +\infty$. As we will justify later, we have that 
\begin{equation} \label{LL}
L_1,L_2>0.
\end{equation} 
Notice that
\begin{equation} \label{star1716}
(\delta_j^n)^2 (r_j^n+\delta_j^n r)^\alpha= (\frac{\delta_j^n}{r_j^n+\delta_j^n r})^2 (r_j^n+\delta_j^n r)^{2+\alpha} \leq (\frac{\delta_j^n}{r_j^n+\delta_j^n r})^2  \to 0
\end{equation}
as $n \to +\infty$ in $C_{\hbox{loc}}(-L_1,L_2)$, in view of $\frac{r_j^n}{\delta_j^n}\to +\infty$ as $n \to +\infty$. Up to a subsequence, by elliptic estimates we have that $\tilde V_j^n \to \tilde V_j$ in $C^1_{\hbox{loc}}(-L_1,L_2) $, where $\tilde V_j$ is a solution of
$$\left\{ \begin{array}{ll}
-(\tilde V_j)''=(\tilde V_j)^{\frac{N+2}{N-2}} &\hbox{in }(-L_1,L_2)\\
0<\tilde V_j \leq \tilde V_j(0)=1 &\hbox{in }(-L_1,L_2). \end{array} \right.$$ 
Since by energy conservation there holds
$$\frac{N}{N-2}(\tilde V_j')^2+(\tilde V_j)^{\frac{2N}{N-2}}=1,$$
the property $\tilde V_j>0$ implies that $L_1,L_2<+\infty$. By \eqref{1649} with $R=r_j^n$ and $r=R_{j-1}^n$ we get 
\begin{eqnarray}
(-1)^{j-1} (\delta_j^n)^{\frac{N}{2}}  v_n'(R_{j-1}^n) &=& \frac{(\delta_j^n)^{\frac{N}{2}}}{(R_{j-1}^n)^{N-1}}
\int_{R_{j-1}^n}^{r_j^n} s^{N-1}(|v_n|^{\frac{N+2}{N-2}}+\epsilon_n s^\alpha |v_n|)ds \nonumber \\ 
&=& \int_{-\frac{r_j^n-R_{j-1}^n}{\delta_j^n}}^{0} (\frac{r_j^n+\delta_j^n s}{R_{j-1}^n})^{N-1}[(\tilde V_j^n)^{\frac{N+2}{N-2}}+\epsilon_n  (\delta_j^n)^2(r_j^n+\delta_j^n s)^\alpha \tilde V_j^n]ds \nonumber \\
&\to& \int_{-L_1}^{0} (\tilde V_j)^{\frac{N+2}{N-2}} ds \label{1904}
\end{eqnarray}
in view of $\tilde V_j^n \leq 1$, \eqref{star1716} and 
\begin{equation} \label{1726}
\frac{r_j^n}{\delta_j^n} \to +\infty,\: \frac{r_j^n-R_{j-1}^n}{\delta_j^n} \to L_1 \in [0,+\infty) \: \Rightarrow\:
\frac{r_j^n}{R_{j-1}^n}=1+\frac{\frac{r_j^n-R_{j-1}^n}{\delta_j^n}}{\frac{r_j^n}{\delta_j^n}-\frac{r_j^n-R_{j-1}^n}{\delta_j^n}} \to 1
\end{equation}
as $n \to +\infty$. When $j-1 \in J$, \eqref{1904} is in contradiction with \eqref{1811} since
\begin{eqnarray*} 
\frac{(\delta_{j-1}^n)^{\frac{N-2}{2}}}{(R_{j-1}^n)^{N-1}}=(\frac{\delta_{j-1}^n}{R_{j-1}^n})^{\frac{N-2}{2}}
(\frac{r_j^n}{R_{j-1}^n})^{\frac{N}{2}} (\frac{\delta_j^n}{r_j^n})^{\frac{N}{2}}
\frac{1}{(\delta_j^n)^{\frac{N}{2}}}
=o\left( \frac{1}{(\delta_j^n)^{\frac{N}{2}}} \right)
\end{eqnarray*} 
as $n \to +\infty$, as it follows by \eqref{1726}, $j-1 \in J$ and $\frac{r_j^n}{\delta_j^n} \to +\infty$ as $n \to +\infty$. Then \eqref{1519} is established.

\medskip \noindent To complete the proof of the Claim, we need to establish \eqref{LL}. Apply \eqref{1649} with $R=r_j^n$ to get by \eqref{max} that
\begin{equation} \label{11000}
|v_n'(r)| \leq (\frac{r_j^n}{r})^{N-1} (\delta_j^n)^{-\frac{N-2}{2}} \left[\frac{r_j^n-R_{j-1}^n}{(\delta_j^n)^2}+ \eps_n \frac{(r_j^n)^{\alpha+1}}{N+\alpha} \right]
\end{equation}
for all $R_{j-1}^n \leq r \leq r_j^n$ and
\begin{equation} \label{1101}
|v_n'(r)| \leq (\delta_j^n)^{-\frac{N-2}{2}}\left[ 
\frac{r-r_{j}^n}{(\delta_j^n)^2}+ \eps_n \frac{r^{\alpha+1}}{N+\alpha} \right]
\end{equation}
for all $r_j^n \leq r \leq R_j^n$. We deduce the following estimates by integrating \eqref{11000} in $[R_{j-1}^n,r_j^n]$:
\begin{equation} \label{2012}
(\delta_j^n)^{-\frac{N-2}{2}}=|\int_{R_{j-1}^n}^{r_j^n} v_n'|\leq 
 (\frac{r_j^n}{R_{j-1}^n})^{N-1} (\delta_j^n)^{-\frac{N-2}{2}} \left[(\frac{r_j^n-R_{j-1}^n}{\delta_j^n})^2+\frac{\eps_n }{N+\alpha} \right],
\end{equation}
and \eqref{1101} in $[r_j^n,R_j^n]$:
\begin{equation}
(\delta_j^n)^{-\frac{N-2}{2}}=|\int_{r_j^n}^{R_j^n} v_n'|\leq 
 (\delta_j^n)^{-\frac{N-2}{2}}\left[ 
(\frac{R_j^n-r_{j}^n}{\delta_j^n})^2+ \frac{ \eps_n}{N+\alpha} \right], \label{2013}
\end{equation}
in view of $\alpha+2>0$ and $\int_0^1 r^{\alpha+1} dr<+\infty$. Therefore we have shown that
\begin{equation} \label{1058}
\frac{R_j^n-r_{j}^n}{\delta_j^n},\ \frac{r_j^n-R_{j-1}^n}{\delta_j^n} \geq \delta>0
\end{equation}
for some $\delta>0$ in view of \eqref{1726}, and the validity of \eqref{LL} follows. \qed

\medskip \noindent When $k=1$, we can apply \eqref{1150} with $j=1$ to get $\alpha\leq N-4$. Indeed, $\alpha>N-4$ would imply, by inserting \eqref{1215} into \eqref{1150}, that $1 =O( \eps_n (R_1^n)^{N-2})$, yielding a contradiction in view of $\eps_n (R_1^n)^{N-2} \to 0$ as $n \to +\infty$. If in addition $R_1^n \to R_1>0$ as $n \to +\infty$, by \eqref{1519bb} for $j=2$ condition $R_1^n<1$ would imply $\delta_2^n \to 0$ and then $\frac{r_2^n}{\delta_2^n} \to +\infty$ as $n \to +\infty$, in contradiction with
\eqref{1519} for $j=2$. Hence $R_1^n=1$ for $n$ large and, when $\alpha<N-4$, by inserting \eqref{1232} into \eqref{1150} for $j=1$ we get that
$$\delta_1^n=\left[\frac{(\alpha+2) \omega_{N-1} \int_{\mathbb{R}^N}  |x|^\alpha V^2}{(\int_{\mathbb{R}^N} V^{\frac{N+2}{N-2}})^2}  \eps_n\right]^{\frac{1}{N-4-\alpha}} (1+o(1)),$$
completing the proof for $k=1$.

\medskip \noindent When $k \geq 2$, by \eqref{1519bb} and \eqref{1519} for $j=2$ we can assume, up to a subsequence, that $\delta_2^n \to 0$ and $\frac{r_2^n}{\delta_2^n} \to L \in [0,+\infty)$ as $n \to +\infty$.

\medskip \noindent {\bf \underline{$3^{\hbox{rd}}$ Claim}}: There holds
\begin{equation} \label{1728}
\lim_{n\to +\infty} \frac{r_2^n}{\delta_2^n}=0.
\end{equation}
Assume by contradiction that $L>0$. Since
$$\frac{r_2^n}{R_1^n}=1+\frac{\frac{r_2^n-R_1^n}{\delta_2^n}}{\frac{r_2^n}{\delta_2^n}-\frac{r_2^n-R_1^n}{\delta_2^n}} \to 1$$
if $\frac{r_2^n-R_1^n}{\delta_2^n} \to 0$ as $n \to +\infty$, by \eqref{2012}-\eqref{2013} we can still deduce the validity of \eqref{1058} for $j=2$. Up to a subsequence, we can then assume that
$$  \frac{R_1^n}{\delta_2^n} \to  L_1 \in [0,L),\qquad  \frac{R_2^n}{\delta_2^n} \to L_2 \in(L,+\infty].$$
The function $V_2^n$ does solve
$$\left\{ \begin{array}{ll} 
-\Delta V_2^n=(V_2^n)^{\frac{N+2}{N-2}}+ \eps_n (\delta_2^n)^{2+\alpha}|x|^\alpha V_2^n &\hbox{in }I_n=(\frac{R_1^n}{\delta_2^n},\frac{R_2^n}{\delta_2^n})\\
0<V_2^n \leq V_2^n(\frac{r_2^n}{\delta_2^n})=1 &\hbox{in }I_n\\
V_2^n =0 &\hbox{on }\partial I_n \end{array} \right.$$
in view of \eqref{max}. Arguing as above, by elliptic estimates we have that, up to a subsequence, $V_2^n \to V_2$ in $C^1_{\hbox{loc}}(A) $, $A=B_{L_2}(0) \setminus \overline{ B_{L_1}(0)}$, where $V_2$ solves
$$-\Delta V_2=(V_2)^{\frac{N+2}{N-2}} \hbox{ in }A,\qquad 0<V_2 \leq V_2(L)=1 \hbox{ in }A.$$
By \eqref{20222} it follows that
\begin{eqnarray} 
-(\delta_2^n)^{-\frac{N-2}{2}}  (R_1^n)^{N-1} v_n'(R_1^n) &=&
\int_{\frac{R_1^n}{\delta_2^n}}^{\frac{r_2^n}{\delta_2^n}} s^{N-1}(V_2^n)^{\frac{N+2}{N-2}}+ \eps_n
(\delta_2^n)^{2+\alpha} \int_{\frac{R_1^n}{\delta_2^n}}^{\frac{r_2^n}{\delta_2^n}} s^{N-1+\alpha} V_2^n \nonumber \\
&\to&  \int_{L_1}^{L} s^{N-1} (V_2)^{\frac{N+2}{N-2}} \label{1840} 
\end{eqnarray} 
as $n \to +\infty$ in view of $V_2^n \leq 1$. Since $1\in J$, by \eqref{1811} and \eqref{1840} we get that $\delta_1^n \sim \delta_2^n$ as $n \to +\infty$, in contradiction with
$$\frac{\delta_2^n}{\delta_1^n}\geq \frac{1}{2L} \frac{r_2^n}{\delta_1^n}\geq \frac{1}{2 L} \frac{R_1^n}{\delta_1^n} \to +\infty$$
as $n \to +\infty$ as it follows by \eqref{1108}. Then \eqref{1728} is established and the Claim is proved. \qed

\medskip \noindent Once \eqref{1728} is established, we proceed as follows. Since $0\leq \frac{r_2^n-R_1^n}{\delta_2^n}\leq \frac{r_2^n}{\delta_2^n}\to 0$ observe that 
\begin{equation} \label{2225}
\frac{R_1^n}{r_2^n} \to 0 \hbox{ as }n \to +\infty
\end{equation} 
in view of \eqref{2012}. Up to a subsequence, we can assume that $\frac{R_2^n}{\delta_2^n} \to L_2\in (0,+\infty]$ in view of \eqref{2013}, and, arguing as above, deduce by elliptic estimates that $V_2^n \to V_2$ in $C^1_{\hbox{loc}} (B_{L_2}(0) \setminus \{0\})$ as $n \to +\infty$, where $V_2$ solves
$$-\Delta V_2=(V_2)^{\frac{N+2}{N-2}} \hbox{ in }B_{L_2}(0),\qquad 0\leq V_2\leq 1 \hbox{ in }B_{L_2}(0)$$
with $V_2(L_2)=0$ if $L_2<+\infty$. Since by \eqref{1101} there holds
$$|(V_2^n)'|(r)\leq r-\frac{r_2^n}{\delta_2^n}+ \eps_n (\delta_2^n)^{\alpha+2} \frac{r^{\alpha+1}}{N+\alpha}$$
for all $\frac{r_2^n}{\delta_2^n} \leq r \leq \frac{R_2^n}{\delta_2^n}$, we have that
$$V_2^n(r)=1+\int_{\frac{r_2^n}{\delta_2^n}}^r (V_2^n)' \geq 1-
\frac{1}{2} (r-\frac{r_2^n}{\delta_2^n})^2- \eps_n (\delta_2^n)^{\alpha+2}\int_0^r \frac{s^{\alpha+1}}{N+\alpha} $$
for all $\frac{r_2^n}{\delta_2^n}\leq r \leq \frac{R_2^n}{\delta_2^n}$, and then as $n \to +\infty$ we deduce that $1\geq V_2(r) \geq 1-
\frac{1}{2} r^2$
for all $0<r<L_2$. Hence $V_2(0)=1$, $L_2=+\infty$ by Pohozaev identity \eqref{17577} and $V_2=V$, where $V$ is given by \eqref{Vdelta}. 

\medskip \noindent So far we have shown that $1 \in J \: \Rightarrow \: 2 \in J$. As already explained, the new estimate \eqref{1316}  becomes crucial here. The difficulty is that very few is known about $v_n$ in the range $[R_1^n,r_2^n]$, a problem which can be by-passed through the following trick. The key remark is that
\begin{equation} \label{1237}
\frac{1}{r^{N-1}}\int_{R_1^n}^r s^{N-1} (|v_n|^{\frac{N+2}{N-2}}+ \eps_n s^\alpha |v_n|)ds=(\delta_2^n)^{-\frac{N-2}{2}} O\left(   \frac{r_2^n}{(\delta_2^n)^2}+ \eps_n r^{\alpha+1} \right)
\end{equation}
for all $r \in [R_1^n, r_2^n]$ in view of \eqref{max}. By integrating \eqref{1105} for $v_n$ in $(R_1^n,r)$ we get that
\begin{eqnarray}
\label{1247}
v_n'(r)=\frac{(R_1^n)^{N-1} v_n'(R_1^n)}{r^{N-1}}-\frac{1}{r^{N-1}} \int_{R_1^n}^r s^{N-1}(|v_n|^{\frac{N+2}{N-2}}+\epsilon s^\alpha |v_n|)ds 
\end{eqnarray}
for all $r \in [R_1^n,r_2^n]$. Inserting \eqref{1811} with $j=1$ and \eqref{1237} into \eqref{1247} we deduce that
$$v_n'(r)= -\frac{(\delta_1^n)^{\frac{N-2}{2}}}{r^{N-1}}  \int_0^\infty s^{N-1} V^{\frac{N+2}{N-2}} [1+o(1)]
+(\delta_2^n)^{-\frac{N-2}{2}} O\left(   \frac{r_2^n}{(\delta_2^n)^2}+ \eps_n r^{\alpha+1} \right)$$
for all $r \in [R_1^n,r_2^n]$, and then
\begin{eqnarray}
(\delta_2^n)^{-\frac{N-2}{2}}&=& - \int_{R_1^n}^{r_2^n}  v_n' \label{1257} \\
&=& \frac{(\delta_1^n)^{\frac{N-2}{2}} }{N-2} \int_0^\infty s^{N-1} V^{\frac{N+2}{N-2}}[1+o(1)][\frac{1}{(R_1^n)^{N-2}}-\frac{1}{(r_2^n)^{N-2}}]  \nonumber  \\
&& +(\delta_2^n)^{-\frac{N-2}{2}} O\left(   (\frac{r_2^n}{\delta_2^n})^2+ \eps_n \int_0^1 r^{\alpha+1} \right) \nonumber
\end{eqnarray}
as $n \to +\infty$. Since $\frac{R_1^n}{r_2^n}, \ \frac{r_2^n}{\delta_2^n} \to 0$ as $n\to +\infty$ in view of \eqref{1139} with $j=2$ and \eqref{2225}, by \eqref{1257} we deduce the validity of \eqref{1316} for $R_1^n$.

\medskip \noindent We already have that $\alpha\leq N-4$. The case $\alpha=N-4$ can be excluded since \eqref{1232} into \eqref{1150} for $j=1$ would provide $1=O(\eps_n (R_1^n)^{N-2} |\log \frac{R_1^n}{\delta_1^n}|)$, a contradiction in view of $\eps_n, R_1^n \to 0$ and
$$\frac{R_1^n}{\delta_1^n}=
\frac{\delta_2^n}{R_1^n}  \frac{(R_1^n)^2}{\delta_1^n \delta_2^n}=O(
\frac{\delta_2^n}{R_1^n})=O(\frac{1}{R_1^n})$$
as $n \to +\infty$, thanks to \eqref{1316} for $R_1^n$. Hence $\alpha<N-4$ and \eqref{1232} into \eqref{1150} provides that
\begin{equation} \label{1518}
(\alpha+2) \eps_n (\delta_1^n)^{2+\alpha} \int_0^{+\infty}  r^{N-1+ \alpha}V^2= \left(\int_0^\infty r^{N-1} V^{\frac{N+2}{N-2}} \right)^2 (\frac{\delta_1^n}{R_1^n})^{N-2} (1+o(1)).
\end{equation}
In view of \eqref{1316} for $R_1^n$, \eqref{1518} gives that
$$(\delta_1^n)^{\frac{N-6-2\alpha}{2}} \sim \eps_n (\delta_2^n)^{\frac{N-2}{2}} \to 0$$
as $n \to +\infty$, which necessarily requires $\alpha<\frac{N-6}{2}.$

\medskip \noindent We can easily iterate the above procedure to show that $J=\{1,\dots,k\}$ and \eqref{1316} does hold for all $j=1,\dots,k-1$. If \eqref{shrink} does hold, condition $R_k^n<1$ would imply the existence of $R_k^n<r_{k+1}^n < R_{k+1}^n\leq 1$ so that
$v_n(R_{k+1}^n)=0$ and
$$|v_n|(r_{k+1}^n)=\max_{[R_k^n,R_{k+1}^n]} |v_n|.$$
Setting $\delta_{k+1}^n=|v_n(r_{k+1}^n)|^{-\frac{2}{N-2}}$, by  \eqref{1519bb} with $j=k$ we would deduce that
$\delta_{k+1}^n \to 0$ and then $\frac{r_{k+1}^n}{\delta_{k+1}^n} \to +\infty$ as $n \to +\infty$, in contradiction with \eqref{1519} for $j=k$. Hence $R_k^n=1$ for $n$ large.

\medskip \noindent Since by \eqref{1139}
$$\frac{\delta_j^n}{\delta_{j+1}^n} =\frac{\delta_j^n}{R_j^n} \frac{R_j^n}{\delta_{j+1}^n} <\frac{\delta_j^n}{R_j^n} \frac{r_{j+1}^n}{\delta_{j+1}^n}\to 0 \quad \hbox{as }n \to +\infty$$
for all $j=1,\dots,k-1$, by \eqref{1150} and \eqref{1232} we get that
\begin{equation} \label{15188}
(\alpha+2) \eps_n (\delta_j^n)^{2+\alpha} \int_0^{+\infty}  r^{N-1+ \alpha}V^2= \left(\int_0^\infty r^{N-1} V^{\frac{N+2}{N-2}} \right)^2 (\frac{\delta_j^n}{R_j^n})^{N-2} (1+o(1)).
\end{equation}
For $j=k$ by \eqref{15188} we have that
\begin{equation} \label{1529}
\delta_k^n =\left[  \frac{(\alpha+2) \omega_{N-1} \int_{\mathbb{R}^N}  |x|^{\alpha}V^2}{(\int_{\mathbb{R}^N} V^{\frac{N+2}{N-2}})^2}   \eps_n \right]^{\frac{1}{N-4-\alpha}}(1+o(1))
\end{equation}
as $n \to +\infty$ in view of $R_k^n=1$. For $j=1,\dots,k-1$, by inserting \eqref{1316} into \eqref{15188} we have that
$$(\alpha+2)   \int_{\mathbb{R}^N} |x|^{\alpha} V^2  \eps_n (\delta_{j+1}^n)^{\frac{N-2}{2}}= (N-2) \int_{\mathbb{R}^N} V^{\frac{N+2}{N-2}} (\delta_j^n)^{\frac{N-6-2\alpha}{2}} (1+o(1))$$
as $n \to +\infty$. We finally deduce that 
\begin{equation} \label{1530}
\delta_j^n = \left[\frac{(\alpha+2)   \int_{\mathbb{R}^N}  |x|^{\alpha} V^2 }{(N-2) \int_{\mathbb{R}^N} V^{\frac{N+2}{N-2}}}\right]^{\frac{2}{N-6-2\alpha}}
( \eps_n)^{\frac{2}{N-6-2\alpha}} (\delta_{j+1}^n)^{\frac{N-2}{N-6-2\alpha}}(1+o(1))
\end{equation}
as $n \to +\infty$ for all $j=1,\dots,k-1$, or equivalently
\begin{eqnarray*} \delta_j^n \sim  \left[\frac{(\alpha+2)   \int_{\mathbb{R}^N} |x|^{\alpha} V^2 }{(N-2) \int_{\mathbb{R}^N} V^{\frac{N+2}{N-2}}} \eps_n \right]^{\frac{(N-2)(\frac{N-2}{N-6-2\alpha})^{k-j}-(N-4-\alpha)}{(2+\alpha)(N-4-\alpha)}}  \left[\frac{ (N-2) \omega_{N-1}  }{\int_{\mathbb{R}^N} V^{\frac{N+2}{N-2}} }\right]^{\frac{1}{N-4-\alpha}(\frac{N-2}{N-6-2\alpha})^{k-j}}
\end{eqnarray*}
as it follows iteratively by \eqref{1529}-\eqref{1530}. This completes the proof. \qed

\section{A perturbative approach: setting of the problem}\label{sec2}
\noindent In this section we provide a very delicate perturbative scheme in order to prove Theorem \ref{torri}. The main ingredient in our construction are the Euclidean bubbles defined in \eqref{bubblebis} which are all the solutions to the  critical equation \eqref{1635bis} with Hardy potential  in the Euclidean space. 

\medskip \noindent It turns out to be useful to rewrite problem \eqref{1422bis} as follows.
We let $\i^*:L^{\frac{2N}{N+2}}\(\Omega\)\rightarrow H^1_0(\Omega)$ be the adjoint operator of the embedding $\i:H^1_0\(\Omega\)\hookrightarrow L^{{2N\over N-2}}\(\Omega\)$, i.e. for any $w \in L^{\frac{2N}{N+2}}(\Omega)$ the function $u=\i^*\(w\) \in H^1_0\(\Omega\)$ is the unique solution of 
\begin{equation}\label{is}
L_\gamma u=-\Delta u-\gamma {u\over|x|^2}=w \ \hbox{in}\ \Omega,\ u=0\ \hbox{on}\ \partial\Omega.
\end{equation}
By continuity of the embedding $H^1_0(\Omega)\hookrightarrow L^{2N\over N-2}\(\Omega\)$, we get
$$\left\|\i^*\(w\)\right\|\le C\left |w\right |_{\frac{2N}{N +2}}$$
for some $C>0$. We rewrite problem \eqref{1422bis} as
\begin{equation}\label{Eq1b}
u=\i^*\[|u|^{\frac{4}{N-2}}u +\eps u\],\ u\in H^1_0(\Omega).\end{equation}

\subsection{The projection of the bubble}
To get a good approximation of our solution, it is necessary to project the bubble $U_\mu$ onto the space $H^1_0(\Omega)$. 
More precisely, letting $PU_\mu=\i^*\(U_\mu^{\frac{N+2}{N-2}}\)$, according to \eqref{is} $PU_\mu$ solves
\begin{equation}\label{PU}
L_\gamma PU_\mu=L_\gamma U_\mu=U_\mu^{\frac{N+2}{N-2}} \ \hbox{in}\ \Omega,\ PU_\mu=0\ \hbox{on}\ \partial\Omega
\end{equation}
in view of \eqref{1635bis} for $U_\mu$. Since $U_\mu^{\frac{N+2}{N-2}}\geq 0$ in $\Omega$ and $PU_\mu \in H_0^1(\Omega)$, by the weak maximum principle we have that $PU_\mu \geq 0$ in $\Omega$. To get the expansion of $PU_\mu$ with respect to $\mu$, we make use of some tools introduced by Ghossoub and Robert \cite{gr,GhRo}. First, let us recall the existence of a positive singular solution $G_\gamma \in C^2(\bar \Omega\setminus\left\{0\right\})$ to 
\begin{equation}\label{pbgreenreg}
\left\{
\begin{array}{ll}
L_\gamma G_{\gamma}=0 &\mbox{in } \Omega\setminus\{0\}\\
G_{\gamma }=0 &\mbox{on }\partial\Omega
\end{array}
\right.
\end{equation}
having  near the origin the following expansion:
\begin{equation}\label{expgreen}
G_{\gamma} (x)=\frac{c_1}{|x|^{\bp}}-\frac{c_2}{|x|^{\bm}}+o\(\frac{1}{|x|^{\bm}}\)\qquad \mbox{as}\,\, x\to 0,
\end{equation}
where $c_1,\: c_2>0$ and $\beta_\pm$ are given in \eqref{betabis}. The function $H_\gamma=\frac{c_1}{|x|^{\bp}}-G_\gamma$ in turn satisfies
\begin{equation} 
\left\{
\begin{aligned}
&L_\gamma H_{\gamma}=0 &\qquad & \mbox{in}\,\, \Omega\setminus\{0\}\\
&H_{\gamma }=\frac{c_1}{|x|^{\bp}}&\qquad & \mbox{on}\,\, \partial\Omega
\end{aligned}
\right.
\end{equation}
with
\begin{equation} \label{1922}
H_\gamma (x) \sim \frac{c_2}{|x|^{\bm}}\qquad \hbox{as } x \to 0.
\end{equation}
By Theorem 9 in
\cite{GhRo} observe that $H_\gamma \in H_0^1(\Omega)$, whereas $G_\gamma \notin H_0^1(\Omega)$.
The quantity $m=m_{\gamma,0}=\frac{c_2}{c_1}>0$ is referred to as the {\it Hardy interior mass} of $\Omega$ associated to $L_\gamma$ and w.l.o.g. we can assume $c_1=1$.

\medskip \noindent We have the following estimates.
\begin{lemma}\label{prop-pro} There hold
\begin{itemize} 
\item[(i)] $0\le PU_\mu\le U_\mu\ \hbox{in}\ \Omega$
\item[(ii)] $PU_\mu=U_\mu- \alpha_N \mu^{\Gamma} H_{\gamma}+O\(\frac{\mu^{\frac{N+2}{N-2}\Gamma}}{|x|^{\beta_-}}\)$
uniformly in $\Omega$ as $\mu \to 0$
\item[(iii)] $PU_\mu=U_\mu+O\({\mu^{\Gamma}\over |x|^\bm}\)$ uniformly in $\Omega$ as $\mu \to 0$.
\end{itemize}
\end{lemma}
\begin{proof} (i) The function $\varphi_\mu=U_\mu- PU_\mu$ solves 
$$\left\{
\begin{aligned}
&L_\gamma \varphi_\mu=0 &\qquad & \mbox{in}\,\, \Omega\setminus\left\{0\right\}\\
&\varphi_\mu=U_\mu&\qquad & \mbox{on}\,\, \partial\Omega.
\end{aligned}
\right. $$
Since $U_\mu \geq 0$ and $\varphi_\mu \in H^1(\Omega)$, by the weak maximum principle it follows that $\varphi_\mu \ge 0$ and (i) holds.

\medskip \noindent (ii) Let $W_\mu= U_\mu-PU_\mu-\alpha_N \mu^{\Gamma}H_\gamma$. Then $W_\mu$ satisfies the following problem
$$\left\{
\begin{array}{ll}
L_\gamma W_\mu=0 & \mbox{in } \Omega \setminus \{0 \} \\
W_\mu=  \frac{\alpha_N \mu^{\Gamma}}{ |x|^{\beta_-} (\mu^{\frac{4\Gamma}{N-2}} +
|x|^{\frac{4\Gamma}{N-2}} )^{\frac{N-2}{2}}}-  \frac{\alpha_N \mu^\Gamma}{|x|^{\bp}}=O \(\mu^{\frac{N+2}{N-2} \Gamma} \) & \mbox{on}\,\, \partial\Omega.
\end{array}
\right. $$
Since $W_\mu \in H^1(\Omega)$,  by weak comparison principle it follows that 
$$ W_\mu = O\( \mu^{\frac{N+2}{N-2}\Gamma} H_\gamma \)=
O \( \frac{\mu^{\frac{N+2}{N-2}\Gamma}}{|x|^{\beta_-}}\) \qquad \mbox{in}\,\,\Omega\setminus\left\{0\right\}$$  
in view of \eqref{1922}, and (ii) follows.

\medskip \noindent (iii) It follows immediately by (ii) and \eqref{1922}.
\end{proof}

\subsection{The linearized operator}
It is important to linearize the problem \eqref{1635bis} around the 
 solution $U$   defined in \eqref{bubble0bis}. More precisely, let us consider the linear problem
\begin{equation}\label{lin}
\left\{\begin{aligned}&-\Delta Z-\gamma{Z\over|x|^2}=\frac{N+2}{N-2} U^{\frac{4}{N-2}} Z\ \hbox{ in }\ \mathbb R^N\\ &Z\in D^{1, 2}(\mathbb R^N).\end{aligned}\right.
\end{equation}
Dancer, Gladiali and Grossi in   \cite{dgg}  classified all  the solutions to \eqref{lin}:
\begin{lemma}[Lemma 1.3, \cite{dgg}]\label{34}
Let $\gamma<\frac{(N-2)^2}{4}$ so that $\gamma\neq \gamma_j$ for all $j\in\mathbb N$, where $\gamma_j$ is given by \eqref{assgamma}. Then the space of solutions to \eqref{lin} has dimension 1 and is spanned by 
$$Z^\gamma(x)={1-|x|^{{4 \Gamma \over N-2}}\over |x|^{\bm} \(1+|x|^{{4 \Gamma \over N-2}}\)^{N\over2}},\ x\in\mathbb R^N. $$
If   $\gamma=\gamma_j$ for some $j\in\mathbb N,$ then the space of solutions to \eqref{lin} has dimension $1+\frac{(N+2j-2)(N+j-3)!}{(N-2)! j!}$ and is spanned by $$Z^{\gamma}(x)\quad \hbox{and}\quad Z^{\gamma}_ i(x)=\frac{|x|^{\frac{N \Gamma}{N-2}-\frac{N-2}{2}}P_{j, i}(x)}{\(1+|x|^{\frac{4 \Gamma}{N-2}}\)^{\frac N 2}},\  i=1, \ldots, \frac{(N+2j-2)(N+j-3)!}{(N-2)! j!},$$ where $\{P_{j, i}\}$  is a basis for the space $\mathbb P_j(\mathbb R^N)$ of $j-$homogeneous harmonic polynomials in $\mathbb R^N$.
\end{lemma}

\noindent Given $\mathcal G$  a closed subgroup in the space of linear isometries $\mathcal O(N)$ of $\mathbb R^N,$ we say that a domain $\Omega\subset\mathbb R^N$ is  $\mathcal G-$invariant if $\mathcal G  x\subset \Omega$ for any $x\in \Omega $ and a function  $u:\Omega\to\mathbb R$ is $\mathcal G-$invariant if $u(g x)=u(x) $ for any $x\in \Omega  $ and $g\in\mathcal G .$

 \begin{definition}\label{sym}
If $\gamma=\gamma_j$ for some $j \in \mathbb N$ (see \eqref{assgamma}), $\Omega$ is said to be a $j-$admissible domain if  $\Omega$ is $\mathcal G_j-$invariant  for some closed subgroup  $\mathcal G_j \subset \mathcal O(N)$  so that
 $\int\limits_{\mathbb R^N}Z^{\gamma}_{ i}(x) \phi(x)dx=0$ for any $i$ and any $\mathcal G_j-$invariant function $\phi\in D^{1,2}(\mathbb R^N).$ 
\end{definition}

\begin{remark}\label{adm}
 A ball is $j-$admissible   for  all $j \in \mathbb N$ by taking $\mathcal G_j=O(N)$. Any even domain $\Omega$ (i.e. $x\in\Omega$ iff $-x\in\Omega$) is $j-$admissible for all $j \in \mathbb N$ odd by taking $\mathcal G_j=\{Id,-Id\}$, since any homogeneous harmonic polynomials of odd degree is odd.\end{remark}

\begin{remark}\label{1d}
In the following we will work in a setting where  the space of solutions to \eqref{lin} is simply generated by $Z^\gamma.$ In a general domain, we will require either   $\gamma>0$ or  $\gamma\le 0$ with $\gamma\neq \gamma_j$ for all $j \in \mathbb N$. If $\gamma=\gamma_j$ for some $j \in \mathbb N$, we will assume that $\Omega$ is a $j-$admissible domain and we will work in the space of $\mathcal G_j-$invariant functions. Indeed, by  Lemma \ref{34} we immediately deduce that  the space of $\mathcal G_j-$invariant   solutions to \eqref{lin} is spanned by $Z^\gamma.$ 
\end{remark}

\medskip \noindent From now on we let $Z=Z^\gamma$ and we omit the dependence on $\gamma$. 
It is clear that the function 
$$Z_\mu(x)=\mu ^{-\frac{N-2}{2}}Z\left(\frac{x}{\mu}\right)=
\frac{\mu^\Gamma (\mu^{\frac{4\Gamma}{N-2}}-|x|^{\frac{4\Gamma}{N-2}})}{|x|^{\beta_-}(\mu^{\frac{4\Gamma}{N-2}}+|x|^{\frac{4\Gamma}{N-2}})^{\frac{N}{2}}},\quad \ x\in \mathbb R^N,$$
solves the linear problem
$$-\Delta Z_\mu-\gamma{Z_\mu\over|x|^2}=\frac{N+2}{N-2} U_\mu^{\frac{4}{N-2}} Z_\mu\ \hbox{in}\ \mathbb R^N.$$
We need to project the function $Z_\mu$ to fit Dirichlet boundary condition, i.e. we consider the function $PZ_\mu=\i^*\(\frac{N+2}{N-2} U_\mu^{\frac{4}{N-2}}  Z_\mu\)$ according to \eqref{is}. We need an expansion of $PZ_\mu$ with respect to $\mu$.
\begin{lemma}\label{prop-proZ} As $\mu \to 0$ there hold uniformly in $\Omega$  
\begin{itemize}
\item[(i)] $PZ_\mu=Z_\mu+\mu^\Gamma H_\gamma+O\(\frac{\mu^{\frac{N+2}{N-2}\Gamma}}{|x|^{\beta_-}}\)$
\item[(ii)] $PZ_\mu= Z_\mu+O\({\mu^{\Gamma}\over |x|^\bm}\)$.
\end{itemize}
\end{lemma}
\begin{proof}We argue as in the proof of  Lemma \ref{prop-pro}.
\end{proof}

\subsection{The tower}
Let $k\ge1$ be a fixed integer. We look for  solutions to \eqref{1422bis}, or equivalently to \eqref{Eq1b}, of the form 
\begin{equation}\label{sol}
u= \sum_{j=1}^k (-1)^j P U_{\mu_j} +\Phi,
\end{equation}
where 
\begin{equation}\label{mujbis}
\mu_1=e^{-\frac{d_1}{\eps}} 
\end{equation} 
when $\Gamma=1$ and 
\begin{equation}\label{muj}
\mu_{j }= d_{j } \eps^{\sigma_j}, \ j=1,\ldots,k,
\end{equation} 
when $\Gamma>1$, with $d_1,\dots,d_k\in(0,+\infty)$ and $\sigma_j$ given by \eqref{sigmaj}. The choice \eqref{mujbis}-\eqref{muj} of the concentration rates is motivated by the validity of the following crucial relations: for $\Gamma=1$
\begin{equation}\label{muj2bis}
\mu_1^2 \sim \eps\mu_1^2 \log \frac{1}{\mu_1} 
\end{equation} 
and for $\Gamma>1$
\begin{equation}\label{muj2}\mu_1^{2\Gamma}\sim\eps\mu_1^2\quad \hbox{and}\quad \({\mu_j\over\mu_{j-1}}\)^{\Gamma}\sim \eps\mu_j^2, \ j=2,\ldots,k .
\end{equation} 
To build solutions of given sign with a simple blow-up point at the origin, we need to assume $\Gamma \geq 1$ and consider the case $k=1$. The assumption $\Gamma>2$ is necessary when constructing  sign-changing solutions, i.e. $k\ge2$, to guarantee $\sigma_1,\ldots,\sigma_k>0$.

\medskip \noindent  The remainder term $\Phi$ shall be splitted into the sum of $k$ terms of different order:
\begin{equation}\label{resto}
\Phi=\sum_{\ell=1}^k \phi_{\ell},
\end{equation}
where  each remainder term $\phi_{\ell}$ only depends on  $\mu_{1},\dots,\mu_\ell $ and belongs to the space $\mathcal K^\bot_{\ell}$ defined as follows. For any $\ell=1,\dots,k$  we define the subspace $\mathcal K_{\ell}={\rm Span}\left\{PZ_{\mu_1 },\dots,PZ_{\mu_\ell}\right\} $ and either
$$ \mathcal K^\bot_{\ell}=\left\{\phi\in H^1_0(\Omega)\ :\ \langle \phi, PZ_{\mu_i}\rangle  =0,\     i=1,\dots,\ell\right\}$$
when $\Omega$ is a general domain and $\gamma\not=\gamma_j$ for all $j \in \mathbb N$ or
$$ \mathcal K^\bot_{\ell}=\left\{\phi\in H^1_0(\Omega)\ :\ \hbox{$\phi$ is $\mathcal G_j-$invariant,}\ \langle \phi, PZ_{\mu_i}\rangle  =0,\     i=1,\dots,\ell\right\}$$
when $\Omega$ is $j-$admissible and $\gamma=\gamma_j$ for some $j \in \mathbb N$ (see Remark \ref{1d}). We also define $\Pi_{\ell}$ and $\Pi^\bot_{\ell}$ as the projections of the Sobolev space $H^1_0(\Omega)$ onto the respective subspaces $\mathcal K_{\ell}$ and $\mathcal K^\bot_{\ell}$. 

\medskip \noindent In order to solve \eqref{Eq1b}, we shall solve the system
\begin{eqnarray} \label{bif}
&&\Pi^\bot_{k}\left\{u-\i^*\left[|u|^{\frac{4}{N-2}}u+\eps u \right]\right\}=0\\
&&\Pi_{k}\left\{u-\i^*\left[|u|^{\frac{4}{N-2}}u+\eps u \right]\right\}=0\nonumber
\end{eqnarray}
for $u$ given as in \eqref{sol}. For sake of simplicity, for any $j=1,\dots,k$ we set $U_j=U_{\mu_j}$ and $Z_j=Z_{\mu_j}$.

\section{The Ljapunov-Schmidt procedure}\label{sec3}
\noindent In this section we give an outline for the proof of Theorem \ref{torri}. To make the presentation more clear, all the results are stated without proofs, which are postponed into the Appendix.

\subsection{The remainder term: solving equation \eqref{bif}}
In order to find the remainder term $\Phi,$ we shall find functions $\phi_{\ell}$, $\ell=1, \ldots, k,$ which solve the following system:
\begin{equation}\label{sistema}
\left\{
\begin{aligned}
&\mathcal E_1+\mathcal L_1(\phi_{1})+\mathcal N_1(\phi_{1})=0\\
&\mathcal E_2+\mathcal L_2(\phi_{2})+\mathcal N_2(\phi_{1}, \phi_{2})=0\\
&\ldots\\
&\ldots\\
&\mathcal E_k+\mathcal L_k(\phi_{k})+\mathcal N_k(\phi_{1}, \ldots, \phi_{k})=0.\\
\end{aligned}
\right.
\end{equation}
Setting $f(u)=|u|^{\frac{4}{N-2}}u $, the error terms $\mathcal E_\ell$ are defined by
$$\begin{aligned}
\mathcal E_\ell&=\Pi^\bot_{\ell}\left\{(-1)^\ell PU_\ell-\i^*\left[f\left(\sum_{j=1}^\ell (-1)^jPU_j\right)-f\left(\sum_{j=1}^{\ell-1}(-1)^jPU_j\right)+\eps (-1)^\ell PU_\ell\right]\right\} \end{aligned} $$
and the linear operators $\mathcal L_\ell$ are given by
$$\mathcal L_\ell (\phi)=\Pi^\bot_{\ell}\left\{\phi-\i^*\left[f'\left(\sum_{j=1}^\ell (-1)^jPU_j\right)\phi+\eps \phi \right]\right\},$$
with the convention that a sum over an empty set of indices is zero. The nonlinear terms $\mathcal N_\ell$ have the form
\begin{equation}\label{Nj}
\begin{aligned}
&\mathcal N_\ell(\phi_1, \ldots, \phi_\ell) =\\ &\Pi^\bot_{\ell}\left\{-\i^*\left[f\left(\sum_{j=1}^\ell \((-1)^jPU_j+\phi_{j}\)\right)-f\left(\sum_{j=1}^\ell (-1)^jPU_j\right)-f'\left(\sum_{j=1}^\ell (-1)^jPU_j\right)\phi_{\ell}\right.\right.\\
&\left.\left.-f\left(\sum_{j=1}^{\ell-1}\((-1)^jPU_j+\phi_{j}\)\right)+f\left(\sum_{j=1}^{\ell-1} (-1)^jPU_j\right)\right]\right\}.
\end{aligned}
\end{equation}

\medskip \noindent In order to solve system \eqref{sistema}, first we need to evaluate the $H^1_0(\Omega)-$ norm of the error terms $\mathcal E_\ell $.
\begin{lemma}\label{errorej}
For any $\ell=1, \ldots, k$ and any compact subset $A_\ell\subset (0,+\infty)^{\ell} $ there exist $C, \ \eps_0>0$ such that for any $\eps\in(0,\eps_0)$  and for any $(d_1,\dots,d_\ell)\in A_\ell$ there holds 
\begin{equation} \label{115}
\|\mathcal E_1\|=\left\{ \begin{array}{ll} O\(\eps \mu_1^\Gamma \) &\hbox{if } 1\leq \Gamma <2\\
O\(\eps \mu_1^2 \log^{\frac{N+2}{2N}} \frac{1}{\mu_1} \) &\hbox{if } \Gamma=2\\
O\(\eps \mu_1^2\) &\hbox{if } \Gamma>2 \end{array} \right.+\left\{\begin{array}{ll} O\(\mu_1^{2\Gamma}\) &\mbox{if }3\le N \leq 5\\
O\(\mu_1^{2\Gamma}\log^{\frac 23}\frac{1}{\mu_1}\) &\mbox{if } N= 6\\
O\(\mu_1^{{N+2\over N-2}\Gamma}\) &\mbox{if } N\ge 7 \end{array}\right.
\end{equation}
and 
\begin{equation} \label{116}
\|\mathcal E_\ell\|=O(\eps \mu_\ell^2)+\left\{\begin{array}{ll}
O\(({\mu_\ell\over\mu_{\ell-1}} )^{\Gamma}\) &\mbox{if } 3 \le N \leq 5\\
O\( ({\mu_\ell \over \mu_{\ell-1}})^{\frac{N+2}{N-2}\frac{\Gamma}{2}} \log^{\frac 23} \frac{1}{\mu_\ell} \) &\mbox{if } N\ge 6
\end{array}\right.\end{equation}
for any $l=2,\ldots, k$, when $k\geq 2$ and $\Gamma>2$.
\end{lemma}
\noindent Next, we need to understand the invertibility of the linear operators $\mathcal L_\ell  $.
This is done in the following lemma whose proof can be carried out as in \cite{mp}. 
 \begin{lemma}\label{lineare}  
For any $\ell=1, \ldots, k$ and any compact subset $A_\ell\subset (0,+\infty)^{\ell} $ there exist $C,\eps_0>0$ such that for any $\eps\in(0,\eps_0)$  and for any $(d_1,\dots,d_\ell )\in A_\ell$  there holds
\begin{equation}
\|\mathcal L_\ell (\phi_{\ell})\|\geq C  \|\phi_{\ell}\|\ \hbox{for any $\phi_\ell \in \mathcal K_{\ell}^\bot$}.\end{equation}
In particular $\mathcal L_\ell^{-1}:\mathcal K_\ell^\bot \to K_\ell^\bot$ is well defined for $\eps \in (0,\eps_0)$ and $(d_1,\dots,d_\ell )\in A_\ell$ and has uniformly bounded operatorial norm.
\end{lemma} 
\noindent Finally, we are able to solve system \eqref{sistema}. This is done in the following proposition, whose proof in the Appendix relies on a sophisticated contraction mapping argument.
\begin{proposition}\label{phij}  Given $A \subset  (0,+\infty)^{k} $ compact, there exists $\eps_0>0$ such that for any $\eps\in(0,\eps_0)$  there exist
$C^1-$maps $(d_1,\dots,d_k )\in A\to \phi_{\ell, \eps}=\phi_{\ell, \eps}(d_1,\dots,d_\ell)  \in \mathcal K^\bot_{\ell}$, $\ell=1,\ldots,k$,  which solve \eqref{sistema} and satisfy uniform estimates:
\begin{eqnarray} 
\|\phi_{1, \eps} \|=O(\|\mathcal E_1\|), \quad  \|\phi_{\ell, \eps}\| = O\( ({\mu_\ell\over\mu_{\ell-1}} )^{\Gamma} +({\mu_\ell\over\mu_{\ell-1}} )^{\frac{\Gamma}{2}+1}+ ({\mu_\ell \over \mu_{\ell-1}})^{\frac{N+2}{2(N-2)}\Gamma} \log^{\frac 23} \frac{1}{\mu_\ell}\) \label{stimaphi1} 
\end{eqnarray}
for $l\geq 2$ and
\begin{eqnarray} 
\|\nabla_{(d_1,\dots,d_\ell)}\phi_{\ell, \eps}\|=o(1) \quad \ell=1,\ldots,k. \label{stimaphi2}
\end{eqnarray}
Moreover, there exists $\rho>0$ so that 
\begin{equation}\label{stimalinfty}
 |\phi_{\ell, \eps}(x)|=O\( 1\over \mu_\ell^{\Gamma} |x|^\bm\)\ \hbox{if}\ x\in B_{\rho \mu_\ell}(0).
\end{equation}  
\end{proposition}

\subsection{The reduced problem: proof of Theorem \ref{torri}}
Let us recall the expression for the energy functional $J_\eps:H_0^1(\Omega)\to \mathbb R$:
$$J_\epsilon(u)={1\over2}\int\limits_\Omega\(|\nabla u|^2-\gamma{u^2\over|x|^2} -\epsilon u^2\)dx-{N-2 \over 2N}\int\limits_\Omega |u|^{\frac{2N}{N-2}}dx,$$
whose critical points are solutions to the problem \eqref{1422bis}. Let us introduce the reduced energy as
$$J_\eps (\mu_1, \ldots, \mu_k )=J_\eps\(\sum_{j=1}^k (-1)^jPU_j  \).$$
Given $\Phi_{\eps}$ according to \eqref{resto} and Proposition \ref{phij}, the following result is the main core of the finite dimensional reduction of our problem.
\begin{proposition}\label{prob-rido}
Given \eqref{mujbis}-\eqref{muj}, we have that
\begin{equation}\label{21556}
J_\eps(\mu_1) =  A_1+ \left\{\begin{array}{ll} A_2 m \mu_1^2- A_3 \eps \mu_1^2 \log \frac{1}{\mu_1}&\hbox{if }\Gamma=1\\
A_2 m \mu_1^{2\Gamma}- A_3 \eps \mu_1^2 & \hbox{if }\Gamma>1 \end{array} \right. +\Upsilon_1(\mu_1)
\end{equation}
and when $\Gamma>2$
\begin{equation}\label{21557}
\begin{aligned}
J_\eps(\mu_1, \ldots, \mu_k ) =&  kA_1+ A_2 m \mu_1^{2\Gamma}-A_3 \eps \mu_1^2+\sum_{\ell=2}^k \[A_4 (\frac{\mu_\ell}{\mu_{\ell-1}})^\Gamma -A_3 \eps  \mu_\ell^2 \]\\
&+\sum_{\ell=1}^k \Upsilon_\ell (\mu_1,\ldots,\mu_\ell),
\end{aligned}
\end{equation}
where $ |\Upsilon_1|=o(\mu_1^{2\Gamma})$ and $|\Upsilon_\ell|=o\left( (\frac{\mu_\ell}{\mu_{\ell-1}})^\Gamma \right)$, $\ell=2,\ldots,k$, do hold as $\eps \to 0$ locally uniformly for $(d_1, \ldots, d_k ) $ in $(0, +\infty)^k$. Here $A_1,\ldots,A_4>0$ and $ m>0$ is the {\it Hardy interior mass} of $\Omega$ associated to $L_\gamma$.  
Moreover, critical points of 
$$\widetilde{J}_\eps(\mu_1,\ldots,\mu_k)=J_\eps\(\sum_{j=1}^k (-1)^jPU_j  +\Phi_\eps \)=J_\eps(\mu_1,\ldots,\mu_k)+\sum_{\ell=1}^k \widetilde{\Upsilon}_\ell (\mu_1,\ldots,\mu_\ell)$$
give rise to solutions $\displaystyle \sum_{j=1}^k (-1)^jPU_j  +\Phi_{ \eps}$ of \eqref{1422bis}, where
$\widetilde  {\Upsilon}_\ell$ satisfies the same estimate as $\Upsilon_\ell$.
\end{proposition} 

\begin{proof} [\bf {Proof of Theorem \ref{torri}}]
By \eqref{mujbis}-\eqref{muj} and Proposition \ref{prob-rido} it is sufficient to find a critical point of 
$$F_\eps(d_1)= e^{-\frac{2d_1}{\eps}} \(A_2 m -A_3 d_1+o_\ell(1)\) $$
when $\Gamma=1$ and
$$F_\eps(d_1,\dots,d_k)=\sum\limits_{\ell=1}^k\eps^{2\sigma_\ell+1} \(G_\ell(d_1,\dots,d_\ell)+o_\ell(1)\)$$
when $\Gamma>1$, where
$$G_1(d_1)=A_2 m d_1^{2\Gamma}-A_3 d_1^2, \qquad G_\ell(d_1,\dots,d_\ell)=A_4 (\frac{d_\ell}{d_{\ell-1}} )^\Gamma -A_3 d_\ell^2, \  \ell=2,\dots,k.$$
Here $o_\ell(1)$ only depends on $d_1,\dots,d_\ell$ and $o_\ell(1)\to0$ as $\e\to0$  locally uniformly for $(d_1, \dots, d_\ell) $ in $(0, +\infty)^\ell.$ For $k=1$ it is easily found an interval 
$$I=\left\{ \begin{array}{ll}  (\frac{A_2}{A_3}m+\frac{\eps}{4},\frac{A_2}{A_3}m+\eps)& \hbox{if }\Gamma=1\\
\(\frac{1}{2}(\frac{A_3}{A_2m \Gamma})^{\frac{1}{2(\Gamma-1)}},2(\frac{A_3}{A_2m \Gamma})^{\frac{1}{2(\Gamma-1)}}\)& \hbox{if }\Gamma>1 \end{array} \right. \subset (0,+\infty)$$ 
so that
$$\inf_I F_\eps< \inf_{\partial I}F_\eps$$
for $\eps$ small, which guarantees the existence of a minimum point $d_\eps \in I$ of $F_\eps$. For $k\ge2$ it is still possible to show that $F_\eps$ has a minimum point but the proof is more involved. Since it can be carried out as in \cite{MPV}, we omit the details.
\end{proof}

\section{Appendix} \label{sec4}
\noindent All the technical proofs can be carried out as in \cite{MPV}. Since they are quite involved, we rewrite some of them here by re-adapting the arguments to the present situation.

\subsection{The rate of the error: proof of Lemma \ref{errorej} }
By the property of $\i^*,$ we get
\begin{equation}\label{e1}
\begin{aligned}
\|\mathcal E_1\|=O\(|(U_1)^{\frac{N+2}{N-2}}-(PU_1)^{\frac{N+2}{N-2}}|_{2N\over N+2}\)+O\(\eps|PU_1|_{2N\over N+2}\).
\end{aligned}
\end{equation}
By Lemma \ref{prop-pro} and scaling $x=\mu_1 y$ we have that
\begin{equation} \label{114}
\begin{aligned}
|PU_1|_{2N\over N+2} \le |U_1|_{2N\over N+2}=\mu^2_1|U|_{\frac{2N}{N+2}, \frac{\Omega}{\mu_1}}=\left\{ \begin{array}{ll} O\(\mu_1^{\Gamma} \) &\hbox{if } 1\leq \Gamma <2\\
O\(\mu_1^2 \log^{\frac{N+2}{2N}} \frac{1}{\mu_1}\) &\hbox{if } \Gamma=2\\
O\(\mu_1^2\) &\hbox{if } \Gamma>2 \end{array} \right.
\end{aligned}\end{equation}
in view of $\frac{2 \bm}{N+2}<1$ and $\frac{2 \bp}{N+2} =\frac{N-2+2\Gamma}{N+2}.$ Since $|a+b|^{\frac{N+2}{N-2}}-|a|^{\frac{N+2}{N-2}}=O( |a|^{\frac{4}{N-2}}|b|+|b|^{\frac{N+2}{N-2}})$ for all $a, b\in\mathbb R$, we deduce that
\begin{equation} 
\Big|(U_1)^{\frac{N+2}{N-2}}-(PU_1)^{\frac{N+2}{N-2}}\Big|_{2N\over N+2}=O \( \Big| U_1^{\frac{4}{N-2}}(PU_1-U_1)\Big|_{2N\over N+2}+\Big|PU_1-U_1 \Big|^{\frac{N+2}{N-2}}_{2N \over N-2}\). \label{113}
\end{equation}
By Lemma \ref{prop-pro} and scaling $x=\mu_1 y$ we have that 
\begin{equation}
\begin{aligned}\Big|U_1^{\frac{4}{N-2}}(PU_1-U_1)\Big|_{2N\over N+2}&\le c
\mu_1^\Gamma \Big| \frac{U_1^{\frac{4}{N-2}}}{|x|^{\beta_-}}\Big|_{2N\over N+2}
=   c (\mu_1)^{2\Gamma}
\Big| {U^{4\over N-2} \over |y|^\bm} \Big|_{{2N\over N+2}, \frac{\Omega}{\mu_1}}  \\
&=\left\{\begin{array}{ll} O\(\mu_1^{2\Gamma}\) & \hbox{if } 3\leq N \leq 5\\
O\( \mu_1^{2\Gamma} \log^{\frac{2}{3}} \frac{1}{\mu_1}\) & \hbox{if } N=6\\
O\(\mu_1^{{N+2\over N-2}\Gamma}\) & \hbox{if } N\geq 7 \end{array}\right.
\end{aligned} \label{111}
\end{equation}
and
\begin{equation}\Big|PU_1-U_1 \Big|^{N+2\over N-2}_{2N\over N-2}=O\( | \frac{\mu_1^{\Gamma}}{|x|^{\bm}}|^{N+2\over N-2}_{2N\over N-2}\)
=O\(\mu_1^{{N+2\over N-2}\Gamma}\),
\label{112}
\end{equation}
in view of $\frac{2 \beta_-}{N-2} <1$ and 
\begin{equation} \label{13311}
\frac{2N}{N+2}(\beta_-+\frac{4\beta_+}{N-2})=N-\frac{2N(N-6)}{N^2-4}\Gamma.
\end{equation}
Inserting \eqref{111}-\eqref{112} into \eqref{113}, by \eqref{e1}-\eqref{114} we deduce the validity of \eqref{115}.

\medskip \noindent Let us now consider the case $k \geq 2$ and assume $\Gamma>2$. For $\ell\geq 2$ we have that
\begin{equation*}
\begin{aligned} \|\mathcal E_\ell\| 
=& \underbrace{O(| |\sum_{j=1}^\ell  (-1)^j PU_j|^{\frac{4}{N-2}}\sum_{j=1}^\ell  (-1)^j PU_j- |\sum_{j=1}^{\ell-1} (-1)^j PU_j|^{\frac{4}{N-2}}\sum_{j=1}^{\ell-1} (-1)^j PU_j -(-1)^l ( PU_\ell)^{\frac{N+2}{N-2}}|_{\frac{2N}{N+2}})}_{(I)}\\
& +\underbrace{O\( \left|(U_\ell)^{\frac{N+2}{N-2}}-( PU_\ell )^{\frac{N+2}{N-2}}\right|_{\frac{2N}{N+2}}+ \eps \left|  PU_\ell\right|_{\frac{2N}{N+2}}\)}_{(II)}. 
\end{aligned} \end{equation*}
$(II)$ is estimated as in \eqref{115} with $\mu_1$ replaced by $\mu_l$. As for $(I)$, let us introduce disjoint annuli $\mathcal A_h$ as  
\begin{equation}\label{anelli}
\mathcal A_0=\Omega\setminus B_r(0),\quad \mathcal A_h = B_{\sqrt{\mu_{h-1}\mu_h}}(0)\setminus B_{\sqrt{\mu_h \mu_{h+1}}}(0),\   h=1,\ldots, \ell, 
\end{equation}
where $\mu_0$ satisfies $\mu_0 \mu_1=r^2$ with $r=\frac{1}{2}\hbox{dist}(0,\partial \Omega)$ and $\mu_{\ell+1}=0$. Moreover define $\mu_{-1}$ so that $\mu_{-1} \mu_0=(\hbox{diam}\ \Omega)^2$, in order to get $\mathcal A_0 \subset B_{\sqrt{\mu_{-1}\mu_0}}(0)\setminus B_{\sqrt{\mu_0 \mu_1}}(0)$. Since 
\begin{equation} \label{18355}
|a+b|^{\frac{4}{N-2}}(a+b)-|a|^{\frac{4}{N-2}}a-\frac{N+2}{N-2} |a|^{\frac{4}{N-2}} b= O(|b|^{\frac{N+2}{N-2}})+\underbrace{O(|a|^{\frac{6-N}{N-2}}b^2)}_{\hbox{if } 3\leq N\leq 5} \end{equation}
for all $a, b\in\mathbb R$, we have that 
\begin{eqnarray} \nonumber
&&  | |\sum_{j=1}^\ell (-1)^j PU_j|^{\frac{4}{N-2}}\sum_{j=1}^\ell (-1)^j PU_j - |\sum_{j=1}^{\ell-1} (-1)^j PU_j |^{\frac{4}{N-2}}\sum_{j=1}^{\ell-1} (-1)^j PU_j-(-1)^l ( PU_\ell )^{\frac{N+2}{N-2}} |_ {\frac{2N}{N+2}, \mathcal A_h} \nonumber \\
&&= \left\{ \begin{array}{ll}
O \(\displaystyle \sum_{j=1}^{\ell-1}| (PU_j)^{\frac{4}{N-2}} PU_\ell |_{\frac{2N}{N+2}, \mathcal A_h}+| PU_\ell |^{\frac{N+2}{N-2}}_{\frac{2N}{N-2}, \mathcal A_h}\) & \hbox{if }h=0,\ldots,l-1\\
O \(\displaystyle \sum_{j=1}^{\ell-1} |  (PU_\ell )^{\frac{4}{N-2}} PU_j |_{\frac{2N}{N+2}, \mathcal A_\ell}+\sum_{j=1}^{\ell-1} |PU_j |^{\frac{N+2}{N-2}}_{\frac{2N}{N-2}, \mathcal A_\ell}\)& \hbox{if }h=l.  \end{array}\right. \label{15066}
\end{eqnarray}
Hereafter we will repeatedly use that $\mu_1>>\ldots>>\mu_k$. Since $\frac{2 \beta_-}{N-2}<1<\frac{2 \beta_+}{N-2}$, by Lemma \ref{prop-pro} and scaling $x=\mu_i y$ we have that 
\begin{equation}\label{ok1}\begin{aligned}
\left| PU_ j \right| _{\frac{2N}{N-2}, \mathcal A_h}&\le \left| U_j \right| _{\frac{2N}{N-2}, \mathcal A_h}
=\left| U\right| _{\frac{2N}{N-2}, \frac{\mathcal A_h}{\mu_j}}
=\left\{ \begin{array}{ll} 
O ((\frac{\mu_\ell }{\sqrt{\mu_h \mu_{h+1}}})^{\Gamma})  &\hbox{if }j=\ell \\
O ((\frac{\sqrt{\mu_{l-1} \mu_l }}{\mu_j})^{\Gamma})  &\hbox{if }h=\ell \end{array} \right.
=O (( {\mu_\ell\over \mu_{\ell-1}})^{\frac{\Gamma}{2}})
\end{aligned}\end{equation}
for any $j=1, \ldots, \ell$ and $h=0, \ldots, \ell$ with $\max\{j,h\}=\ell$, $j \not=h$. Since $|x| >>\mu_l$ in $\mathcal A_h$, for any $j=1, \ldots, \ell-1$ and $h=0, \ldots, \ell-1$ by Lemma \ref{prop-pro} we have 
\begin{equation}\label{ok2}\begin{aligned}
\left| (PU_j)^{\frac{4}{N-2}} PU_\ell \right|_{\frac{2N}{N+2}, \mathcal A_h} & \le
\left| U_j^{\frac{4}{N-2}} U_\ell \right|_{\frac{2N}{N+2}, \mathcal A_h}
\leq c \mu_l^{\Gamma} \left| \frac{U_j^{\frac{4}{N-2}}}{|x|^{\beta_+} } \right|_{\frac{2N}{N+2}, \mathcal A_h}\\
&= c (\frac{\mu_l}{\mu_j})^{\Gamma} \left| \frac{U^{\frac{4}{N-2}}}{|y|^{\beta_+} } \right|_{\frac{2N}{N+2}, \frac{\mathcal A_h}{\mu_j}}=O\( (\frac{\mu_l}{\mu_{l-1}})^{\Gamma}\)
\end{aligned}
\end{equation}
when $3\leq N\leq 5$ and
\begin{equation}\label{ok22}
\begin{aligned} & \left| (PU_j)^{\frac{4}{N-2}} PU_\ell \right|_{\frac{2N}{N+2}, \mathcal A_h} \le
\left| U_j^{\frac{4}{N-2}} U_\ell \right|_{\frac{2N}{N+2}, \mathcal A_h}
\leq c \mu_j^{-\frac{4\Gamma}{N-2}} \left| \frac{U_l}{|x|^{\frac{4\beta_-}{N-2}} } \right|_{\frac{2N}{N+2}, \mathcal A_h}\\
&= c (\frac{\mu_l}{\mu_j})^{\frac{4\Gamma}{N-2}} \left| \frac{U}{|y|^{\frac{4\beta_-}{N-2}} } \right|_{\frac{2N}{N+2}, \frac{\mathcal A_h}{\mu_l}} \le c (\frac{\mu_l}{\mu_j})^{\frac{4\Gamma}{N-2}} \left\{ \begin{array}{ll} 
\log^{\frac{2}{3}} \frac{\sqrt{\mu_{h-1} \mu_h}}{\mu_l} & {if }N=6\\
(\frac{\mu_l}{\sqrt{\mu_h \mu_{h+1}}})^{\frac{N-6}{N-2}\Gamma} & \hbox{if }N\geq 7 \end{array} \right.\\
&=O\( (\frac{\mu_l}{\mu_{l-1}})^{\frac{N+2}{N-2}\frac{\Gamma}{2}} \log^{\frac{2}{3}} \frac{1}{\mu_l} \)
\end{aligned}\end{equation}
when $N\geq 6$, in view of $\frac{2 \beta_-}{N-2}<1<\frac{2 \beta_+}{N-2}$ and 
$$\frac{2N}{N+2}(\frac{4\beta_-}{N-2}+\beta_+)=N+\frac{2N(N-6)}{N^2-4}  \Gamma.$$
Similarly, for $j=1,\ldots,l-1$ we have that
\begin{equation}\label{ok222}\begin{aligned}
\left| (PU_l)^{\frac{4}{N-2}} PU_j \right|_{\frac{2N}{N+2}, \mathcal A_l} & \le
\left| U_l^{\frac{4}{N-2}} U_j \right|_{\frac{2N}{N+2}, \mathcal A_l}
\leq c \mu_j^{-\Gamma} \left| \frac{U_l^{\frac{4}{N-2}}}{|x|^{\beta_-} } \right|_{\frac{2N}{N+2}, \mathcal A_l}\\
&= c (\frac{\mu_l}{\mu_j})^{\Gamma} \left| \frac{U^{\frac{4}{N-2}}}{|y|^{\beta_-} } \right|_{\frac{2N}{N+2}, \frac{\mathcal A_l}{\mu_l}}=O\( (\frac{\mu_l}{\mu_{l-1}})^{\Gamma}\)
\end{aligned}
\end{equation}
when $3\leq N\leq 5$ and
\begin{equation}\label{ok2222}
\begin{aligned} & \left| (PU_l)^{\frac{4}{N-2}} PU_j \right|_{\frac{2N}{N+2}, \mathcal A_l} \le
\left| U_l^{\frac{4}{N-2}} U_j \right|_{\frac{2N}{N+2}, \mathcal A_l}
\leq c \mu_l^{\frac{4\Gamma}{N-2}} \left| \frac{U_j}{|x|^{\frac{4\beta_+}{N-2}} } \right|_{\frac{2N}{N+2}, \mathcal A_l}\\
&= c (\frac{\mu_l}{\mu_j})^{\frac{4\Gamma}{N-2}} \left| \frac{U}{|y|^{\frac{4\beta_+}{N-2}} } \right|_{\frac{2N}{N+2}, \frac{\mathcal A_l}{\mu_j}} \le c (\frac{\mu_l}{\mu_i})^{\frac{4\Gamma}{N-2}} \left\{ \begin{array}{ll} 
\log^{\frac{2}{3}} \frac{\mu_j}{\sqrt{\mu_{l-1} \mu_l}} & {if }N=6\\
(\frac{\sqrt{\mu_{l-1} \mu_l }}{\mu_j})^{\frac{N-6}{N-2}\Gamma} & \hbox{if }N\geq 7 \end{array} \right.\\
&=O\( (\frac{\mu_l}{\mu_{l-1}})^{\frac{N+2}{N-2}\frac{\Gamma}{2}} \log^{\frac{2}{3}} \frac{1}{\mu_l} \)
\end{aligned}\end{equation}
when $N \geq 6$ in view of $\frac{2 \beta_-}{N-2}<1$ and \eqref{13311}. By inserting \eqref{ok1}-\eqref{ok2222} into \eqref{15066} we deduce an estimate of $(I)$ which, along with the estimate on $(II)$ in terms of $\mu_\ell$, leads to the validity of \eqref{116}. 

\subsection{The reduced energy: proof of \eqref{21556}-\eqref{21557}} To get an expansion of $J_\e (\mu_1,\ldots,\mu_k)$, let us first write that
\begin{eqnarray*}
&& J_\e (\sum_{\ell=1}^k (-1)^\ell PU_\ell) = \sum_{\ell=1}^k J_\e( PU_\ell)+\sum_{i < \ell} (-1)^{i+\ell}\int_\Omega [U_\ell^{\frac{N+2}{N-2}}- \e PU_\ell -(PU_\ell)^{\frac{N+2}{N-2}}] PU_i \, dx \\
& &-\frac{N-2}{2N} \int_\Omega [|\sum_{\ell=1}^k  (-1)^\ell PU_\ell|^{\frac{2N}{N-2}}-\sum_{\ell=1}^k (PU_\ell)^{\frac{2N}{N-2}}-\frac{2N}{N-2} \sum_{i<\ell}(-1)^{i+\ell} (PU_\ell)^{\frac{N+2}{N-2}}  PU_i ]\, dx 
\end{eqnarray*}
in view of $PU_\ell=\i^*\(U_\ell^{\frac{N+2}{N-2}}\)$. Introducing the quantities 
\begin{equation*}\begin{aligned} a_\ell=&  J_\e( PU_\ell)+\sum_{ i=1 }^{\ell-1} (-1)^{i+\ell}  \int_\Omega [U_\ell^{\frac{N+2}{N-2}}-\e PU_\ell -(PU_\ell)^{\frac{N+2}{N-2}}] PU_i \, dx \\ &
- \frac{N-2}{2N}  \int_{\Omega}  [ |\sum_{i=1}^\ell   (-1)^i PU_i |^{\frac{2N}{N-2}}-|\sum_{i=1}^{\ell-1}  (-1)^i PU_i|^{\frac{2N}{N-2}}-(PU_\ell)^{\frac{2N}{N-2}}-\frac{2N}{N-2} \sum_{i=1}^{\ell-1}(-1)^{i+\ell} ( PU_\ell)^{\frac{N+2}{N-2}} PU_i ] \, dx \end{aligned}
\end{equation*}
for any $\ell=1,\dots,k$, let us notice that each $a_\ell$ only depends on $d_1,\dots,d_\ell$ and the following decomposition does hold:
\begin{equation} \label{05377}
J_\e (\sum_{\ell=1}^k (-1)^\ell PU_\ell) =\sum_{\ell=1}^k a_\ell.
\end{equation}
We claim that
\begin{equation}\label{a1} a_1=A_1+A_2 m \mu_1^{2\Gamma}(1+o(1))-A_3 \eps (1+o(1))\left\{\begin{array}{ll}
\mu_1^2 \log \frac{1}{\mu_1}&\hbox{if }\Gamma=1\\
\mu_1^2& \hbox{if }\Gamma>1 \end{array} \right.
\end{equation}
and
\begin{equation}
\label{aelle}
a_\ell=A_1+A_4\(\frac{\mu_\ell}{\mu_{\ell-1}}\)^{\Gamma}(1+o(1)) -A_3  \eps \mu_\ell^2(1+o(1)),\ \ell=2,\dots,k,
\end{equation}
where $m>0$ is the {\it Hardy interior mass} of $\Omega$ associated to $L_\gamma$ and $A_1,\ldots,A_4>0$. Inserting \eqref{a1}-\eqref{aelle} into \eqref{05377}, we deduce the validity of \eqref{21556}-\eqref{21557}.

\medskip \noindent To compute $J_\e( PU_\ell)$, let us first write
\begin{equation}\label{11598}
\begin{aligned}
J_\e(PU_\ell)&= \frac1{N} \int_\Omega U_\ell^{\frac{2N}{N-2}}dx-\frac12\int_\Omega U_\ell^{\frac{N+2}{N-2}} (PU_\ell-U_\ell )dx-\frac{\eps}{2} \int_\Omega PU_\ell^2dx\\
&-\frac{N-2}{2N} \int_\Omega [(PU_\ell)^{\frac{2N}{N-2}}-U_\ell^{\frac{2N}{N-2}}-\frac{2N}{N-2} U_\ell^{\frac{N+2}{N-2}} (PU_\ell-U_\ell) ]dx
\end{aligned}\end{equation}
in view of $PU_\ell=\i^*\(U_\ell^{\frac{N+2}{N-2}}\)$. We have that
\begin{equation} \label{11599}
 \int_\Omega U_\ell^{\frac{2N}{N-2}} dx= \int_{\mathbb R^N} U^{\frac{2N}{N-2}} dy+O(\mu_\ell^{{2N\over N-2}\Gamma}), \end{equation}
and by Lemma \ref{prop-pro} and \eqref{1922} we deduce that
\begin{equation} \label{11600}
\begin{aligned}
\int_\Omega U_\ell^{\frac{N+2}{N-2}} (PU_\ell-U_\ell )dx&=-\alpha_N  \mu_\ell^\Gamma \int_\Omega 
U_\ell^{\frac{N+2}{N-2}}[H_\gamma(x)+O(\frac{\mu_\ell^{\frac{4\Gamma}{N-2}}}{|x|^{\beta_-}})] \, dx \\
&=-\alpha_N m \mu_\ell^{2\Gamma} \int_{\mathbb R^N} {U^{\frac{N+2}{N-2}}\over |y|^\bm} dy \ (1+o(1))
\end{aligned}
\end{equation}
and 
\begin{equation} \label{11601}
\begin{aligned}
\int_\Omega PU_\ell^2dx &=
\int_\Omega U_\ell^2dx+O(\int_\Omega U_\ell \frac{\mu_\ell^\Gamma}{|x|^{\beta_-}}\  dx)=
\mu_\ell^2 \int_{\frac{\Omega}{\mu_\ell}} U^2\ dy+O(\mu_\ell^{2\Gamma} \int_\Omega \frac{dx}{|x|^{N-2}} )\\
&=
\left\{ \begin{array}{ll}
\mu_\ell^2 \log \frac{1}{\mu_\ell}  [\alpha_N^2 \omega_{N-1}+o(1)]&\hbox{if }\Gamma=1\\
\mu_\ell^2 [\int_{\mathbb{R}^N}U^2 \ dy +o(1)]& \hbox{if }\Gamma>1
\end{array} \right. \end{aligned}
\end{equation}
in view of $\frac{2N \beta_-}{N-2}<N<\beta_-+\frac{N+2}{N-2}\beta_+$ and $2\beta_\pm=N-2 \pm 2\Gamma$. Since 
\begin{equation} \label{18377}
|a+b|^{\frac{2N}{N-2}}-|a|^{\frac{2N}{N-2}}-\frac{2N}{N-2}|a|^{\frac{4}{N-2}}ab=O( |a|^{\frac{4}{N-2}}b^2+|b|^{\frac{2N}{N-2}}) \end{equation}
for all $a,b \in \mathbb R$, by Lemma \ref{prop-pro} we finally deduce
\begin{equation} \label{11602}
\begin{aligned}&\int_\Omega [(PU_\ell)^{\frac{2N}{N-2}}-U_\ell^{\frac{2N}{N-2}}-\frac{2N}{N-2}U_\ell^{\frac{N+2}{N-2}} (PU_\ell-U_\ell) ]dx\\
&=O\(\int_\Omega|PU_\ell-U_\ell |^{\frac{2N}{N-2}}dx+ \int_\Omega U_\ell^{\frac{4}{N-2}}(PU_\ell-U_\ell)^{2}dx\)=O( \mu_\ell^{\frac{2N}{N-2} \Gamma}+\mu_\ell^{2\Gamma} \int_\Omega \frac{U_\ell^{\frac{4}{N-2}}}{|x|^{2\beta_-}})\\
&=\left\{\begin{array}{ll} 
O( \mu_\ell^{\frac{2N}{N-2} \Gamma} +\mu_\ell^{4\Gamma} \int_{B_{\frac{R}{\mu_\ell}}(0)} \frac{U^{\frac{4}{N-2}}}{|y|^{2\beta_-}} \ dy) &\hbox{if }3\leq N\leq 4\\
O(\mu_\ell^{\frac{2N}{N-2} \Gamma}+\mu_\ell^{\frac{2N}{N-2} \Gamma} \int_\Omega \frac{dx}{ |x|^{\frac{4 \beta_+}{N-2}+2\beta_-}}  )&\hbox{if }N\ge 5
\end{array} \right.=o(\mu_\ell^{2\Gamma}) 
 \end{aligned}
\end{equation}
in view of $\frac{2\beta_-}{N-2}<1$ and $\frac{4\beta_+}{N-2}+2\beta_-=N-2 \frac{N-4}{N-2} \Gamma$. Inserting \eqref{11599}-\eqref{11601} and \eqref{11602} into \eqref{11598} we get the validity of \eqref{a1} for $a_1=J_\eps(PU_1)$. 

\medskip \noindent Hereafter let us consider the case $k\geq 2$ with $\Gamma>2$. As for $\ell=1$ in \eqref{a1}, the following expansion does hold  
\begin{equation}\label{a1bis} 
J_\e (PU_\ell) =A_1+A_2 m \mu_\ell^{2\Gamma}(1+o(1))-A_3 \eps \mu_\ell^2(1+o(1)),\quad \ell=1,\ldots,k.
\end{equation}
Let $\ell \geq 2$. Since 
$$U_\ell^{\frac{4}{N-2}}=O\((\frac{\mu_\ell^\Gamma}{|x|^{\beta_+}})^{\frac{5}{2(N-2)}} (\frac{1}{\mu_\ell^\Gamma |x|^{\beta_-}})^{\frac{3}{2(N-2)}}\)=O(\frac{\mu_\ell^{\frac{\Gamma}{N-2}}}{|x|^{\frac{5}{2(N-2)} \beta_+ +\frac{3}{2(N-2)}\beta_- }}),$$ 
by Lemma \ref{prop-pro} and \eqref{18355} we have that for any $i=1,\dots,\ell-1$ 
\begin{equation*}
\begin{aligned}
&\int_\Omega [U_\ell^{\frac{N+2}{N-2}}-\e PU_\ell -(PU_\ell)^{\frac{N+2}{N-2}}] PU_i \, dx\\
&= O\(\mu_\ell^{\frac{N+2}{N-2}\Gamma} \int_\Omega \frac{U_i}{|x|^{\frac{N+2}{N-2} \beta_-}}\ dx+ \mu_\ell^\Gamma \int_\Omega \frac{ U_\ell^{\frac{4}{N-2}} U_i}{|x|^{\beta_-}} \, dx
+\e \int_\Omega U_i U_\ell \, dx \)\\
&=O\(\int_\Omega \frac{\mu_\ell^{\frac{N+2}{N-2}\Gamma} \mu_i^\Gamma }{|x|^{\frac{N+2}{N-2} \beta_-+\beta_+}} \ dx
+(\frac{\mu_\ell}{\mu_i})^\Gamma \mu_\ell^{\frac{\Gamma}{N-2}} \int_\Omega \frac{dx}{|x|^{\frac{4N-5}{2(N-2)} \beta_-+
\frac{5}{2(N-2)}\beta_+}}
+\e (\frac{\mu_\ell}{\mu_i})^\Gamma \int_\Omega \frac{dx}{|x|^{\beta_-+\beta_+}}  \)
\end{aligned}
\end{equation*}
and then
\begin{equation} \label{12422}
\int_\Omega [U_\ell^{\frac{N+2}{N-2}}-\e PU_\ell -(PU_\ell)^{\frac{N+2}{N-2}}] PU_i \, dx=o\((\frac{\mu_\ell}{\mu_{\ell-1}})^\Gamma  \)
\end{equation}
in view of \eqref{mujbis}-\eqref{muj} and
\begin{equation} \label{diesis18}
\frac{N+2}{N-2}\beta_-+\beta_+<N,\quad \frac{4N-5}{2(N-2)} \beta_-+\frac{5}{2(N-2)}\beta_+<N.\end{equation} 
In order to expand the last term in $a_\ell$, $\ell=2,\ldots,k$, let us split $\Omega$ as $\displaystyle \Omega= \bigcup_{h=0}^\ell \mathcal A_h$ (see \eqref{anelli}), and for $h=0,\ldots,\ell$ set 
\begin{equation*}  
\begin{aligned}
I_h=\int_{\mathcal A_h} [ |\sum_{i=1}^\ell   (-1)^i PU_i|^{\frac{2N}{N-2}}-|\sum_{i=1}^{\ell-1}  (-1)^i PU_i|^{\frac{2N}{N-2}}-(PU_\ell )^{\frac{2N}{N-2}}-\frac{2N}{N-2} \sum_{i=1}^{\ell-1}(-1)^{i+\ell} ( PU_\ell)^{\frac{N+2}{N-2}} PU_i ] \, dx.
\end{aligned}
\end{equation*}
By \eqref{ok1}, \eqref{ok222}-\eqref{ok2222} and \eqref{18377} we deduce that
\begin{equation} \label{22511}
\begin{aligned}
I_\ell & =O \( \sum_{i=1}^{\ell-1}  \int_{\mathcal A_\ell}  [(PU_i)^{\frac{2N}{N-2}}+(PU_i)^2  PU_\ell^{\frac{4}{N-2}}] \, dx \) \\
&=O\(\sum_{i=1}^{\ell-1} |PU_i|_{{2N\over N-2},\mathcal A_\ell}^{\frac{2N}{N-2}}+\sum_{i=1}^{\ell-1} |PU_i|_{{2N\over N-2},\mathcal A_\ell} | PU_\ell ^{\frac{4}{N-2}} PU_i  |_{{2N\over N+2},\mathcal A_\ell} \)
=o\((\frac{\mu_\ell}{\mu_{\ell-1}})^{\Gamma}\).
\end{aligned}
\end{equation}
For $h=0, \ldots, \ell-1$ by \eqref{18355} and \eqref{18377} we get that
\begin{equation*}
\begin{aligned}
&I_h =\frac{2N}{N-2} \int_{\mathcal A_h} [|\sum_{i=1}^{\ell-1} (-1)^i PU_i|^{\frac{4}{N-2}}\sum_{i=1}^{\ell-1} (-1)^i PU_i](-1)^\ell PU_\ell \ dx\\
&+O(\int_{\mathcal A_h}\sum_{i=1}^{\ell -1}[ (PU_i)^{\frac{4}{N-2}} (PU_l)^2+(PU_l)^{\frac{N+2}{N-2}}PU_i] \ dx +\int_{\mathcal A_h} (PU_l)^{\frac{2N}{N-2}}\ dx)\\
&=-\frac{2N}{N-2} \int_{\mathcal A_h}  (PU_{\ell-1})^{\frac{N+2}{N-2}} PU_\ell \ dx+O\(
 \int_{\mathcal A_h}\sum_{i=1}^{\ell-2}[(PU_{\ell-1})^{\frac{4}{N-2}} PU_i PU_\ell+(PU_i)^{\frac{N+2}{N-2}} PU_\ell ] \ dx \)\\
&+O(\int_{\mathcal A_h}\sum_{i=1}^{\ell -1}[ (PU_i)^{\frac{4}{N-2}} (PU_l)^2+(PU_l)^{\frac{N+2}{N-2}}PU_i] \ dx +\int_{\mathcal A_h} (PU_l)^{\frac{2N}{N-2}}\ dx).
\end{aligned}
\end{equation*}
Since $\beta_-+\frac{N+2}{N-2}\beta_+=N+\frac{4 \Gamma}{N-2}>N$, by Lemma \ref{prop-pro} and \eqref{diesis18} we deduce that 
\begin{eqnarray*}
&&\int_{\mathcal A_h}(PU_{\ell-1})^{\frac{4}{N-2}} PU_i PU_\ell=
O((\frac{\mu_\ell}{\mu_i})^\Gamma \int_{\mathcal A_h} \frac{U_{\ell-1}^{\frac{4}{N-2}}}{|x|^{N-2}} \ dx)
=O((\frac{\mu_\ell}{\mu_i})^\Gamma \int_{\frac{\mathcal A_h}{\mu_{\ell-1}}} \frac{U^{\frac{4}{N-2}}}{|y|^{N-2}} \ dy)=O((\frac{\mu_\ell}{\mu_i})^\Gamma)\\
&&\int_{\mathcal A_h}(PU_i)^{\frac{N+2}{N-2}} PU_\ell  \ dx=O( \mu_\ell^\Gamma 
\int_{\mathcal A_h} \frac{U_i^{\frac{N+2}{N-2}}}{|x|^{\beta_+}}\ dx)=
O((\frac{\mu_\ell}{\mu_i})^{\Gamma}
\int_{\frac{\mathcal A_h}{\mu_i}} \frac{U^{\frac{N+2}{N-2}}}{|y|^{\beta_+}}\ dy)
=O((\frac{\mu_\ell}{\mu_i})^\Gamma)\\
&& \int_{\mathcal A_h} (PU_\ell)^{\frac{N+2}{N-2}} PU_i \ dx=O( \frac{1}{\mu_i^\Gamma}\int_{\mathcal A_h} \frac{U_\ell^{\frac{N+2}{N-2}}}{|x|^{\beta_-}}\ dx )=
O((\frac{\mu_\ell}{\mu_i})^{\Gamma} \int_{\frac{\mathcal A_h}{\mu_\ell}} \frac{U^{\frac{N+2}{N-2}}}{|y|^{\beta_-}} \ dy)=O((\frac{\mu_\ell}{\mu_i})^\Gamma (\frac{\mu_\ell}{\mu_{\ell-1}})^{\frac{2\Gamma}{N-2}})
\end{eqnarray*}
for any $i=1,\ldots,\ell-1$ and $h=0,\ldots, \ell-1$, which inserted into the previous expression for $I_h$ give that
\begin{equation} \label{22522}
\begin{aligned}
I_h &=-\frac{2N}{N-2} \int_{\mathcal A_h}  (PU_{\ell-1})^{\frac{N+2}{N-2}} PU_\ell \ dx
+O(\int_{\mathcal A_{h}}\sum_{i=1}^{\ell -1} (PU_i)^{\frac{4}{N-2}} (PU_l)^2 \ dx +\int_{\mathcal A_{h}} (PU_l)^{\frac{2N}{N-2}}\ dx)\\
&+o((\frac{\mu_\ell}{\mu_{\ell -1}})^\Gamma)=-\frac{2N}{N-2} \int_{\mathcal A_{h}}  (PU_{\ell-1})^{\frac{N+2}{N-2}} PU_\ell \ dx\\
&+O(\int_{\mathcal A_{h}}\sum_{i=1}^{\ell -1} | (PU_i)^{\frac{4}{N-2}} PU_l|_{\frac{2N}{N+2},\mathcal A_{h} }
| PU_l|_{\frac{2N}{N-2},\mathcal A_{h}}
+| PU_l|^{\frac{2N}{N-2}}_{\frac{2N}{N-2},\mathcal A_{h} })+o((\frac{\mu_\ell}{\mu_{\ell -1}})^\Gamma)\\
&=-\frac{2N}{N-2} \int_{\mathcal A_{h}}  (PU_{\ell-1})^{\frac{N+2}{N-2}} PU_\ell \ dx
+o((\frac{\mu_\ell}{\mu_{\ell -1}})^\Gamma)
\end{aligned}
\end{equation}
for $h=0,\ldots,\ell-1$ in view of \eqref{ok1}-\eqref{ok22}. By \eqref{18355} and Lemma \ref{prop-pro} we have that
\begin{equation} \label{star18} 
\int_{\mathcal A_h}(PU_{\ell-1})^{\frac{N+2}{N-2}} PU_\ell  \ dx=
O((\frac{\mu_\ell}{\mu_{\ell-1}})^{\Gamma}
\int_{\frac{\mathcal A_h}{\mu_{\ell-1}}} \frac{U^{\frac{N+2}{N-2}}}{|y|^{\beta_+}}\ dy)=
o((\frac{\mu_\ell}{\mu_{\ell-1}})^\Gamma)
\end{equation}
for $h=0,\ldots,\ell-2$ and
\begin{equation} \label{bistar18}
\begin{aligned} 
&\int_{\mathcal A_{\ell-1}}  (PU_{\ell-1})^{\frac{N+2}{N-2}} PU_\ell \ dx \\
&= \int_{\mathcal A_{\ell-1}}  U_{\ell-1}^{\frac{N+2}{N-2}} U_\ell \ dx+O(\int_{\mathcal A_{\ell-1}} [\mu_{\ell-1}^\Gamma \frac{(U_{\ell-1})^{\frac{4}{N-2}} U_\ell}{|x|^{\beta_-}}+
\mu_{\ell-1}^{\frac{N+2}{N-2}\Gamma} \frac{U_\ell}{|x|^{\frac{N+2}{N-2}\beta_-}}+ \mu_\ell^\Gamma \frac{U_{\ell-1}^{\frac{N+2}{N-2}}}{|x|^{\beta_-}}] \ dx)\\
&= (\frac{\mu_{\ell-1}}{\mu_\ell})^{\frac{N-2}{2}} \int_{\frac{\mathcal A_{\ell-1}}{\mu_{\ell-1}}}  U^{\frac{N+2}{N-2}} U (\frac{\mu_{\ell-1} }{\mu_\ell} y)\ dy+
O \(\mu_{\ell-1}^\Gamma  \int_{\mathcal A_{\ell-1}} \frac{(U_{\ell-1})^{\frac{4}{N-2}} U_\ell}{|x|^{\beta_-}} \ dx\)\\
&+O\(
\mu_{\ell-1}^{\frac{N+2}{N-2}\Gamma} \mu_\ell^\Gamma \int_{\mathcal A_{\ell-1}} \frac{dx}{|x|^{\frac{N+2}{N-2}\beta_-+\beta_+}}\) + O(\mu_\ell^\Gamma) (\int_{\mathcal A_{\ell-1}} \frac{dx}{|x|^{\frac{2N \beta_-}{N-2}}})^{\frac{N-2}{2N}}\\
&=\alpha_N (\frac{\mu_\ell}{\mu_{\ell-1}})^\Gamma \int_{\mathbb R^N}  \frac{U^{\frac{N+2}{N-2}} }{|y|^{\beta_+}} \ dy+o((\frac{\mu_\ell}{\mu_{\ell-1} })^\Gamma)
\end{aligned}
\end{equation}
in view of
$\frac{\mu_{\ell-1} }{\mu_\ell} y \geq \sqrt{\frac{\mu_{\ell-1} }{\mu_\ell}} \to +\infty$ for all $y \in \frac{\mathcal A_{\ell-1}}{\mu_{\ell-1}}$, \eqref{diesis18} and
\begin{equation} \label{bidiesis}
\int_{\mathcal A_{\ell-1}}\frac{U_{\ell-1}^{\frac{4}{N-2}} U_\ell }{|x|^{\beta_-}}=
\left\{ \begin{array}{ll}
O\( |U_{\ell-1}^{\frac{4}{N-2}} U_\ell|_{\frac{2N}{N+2},\mathcal A_{\ell-1}}\)=O( (\frac{\mu_\ell}{\mu_{\ell-1}})^\Gamma)&\hbox{if }3\leq N\leq 5\\
O\( \frac{\mu_\ell^\Gamma}{\mu_{\ell-1}^{\frac{4\Gamma}{N-2}}} \int_{\mathcal A_{\ell-1}} \frac{dx}{|x|^{\frac{N+2}{N-2}\beta_-+\beta_+}}\)=O( (\frac{\mu_\ell}{\mu_{\ell-1}})^\Gamma)&\hbox{if } N\geq 6
\end{array} \right.
\end{equation}
thanks to \eqref{ok2}. Therefore, inserting \eqref{star18}-\eqref{bistar18} into \eqref{22522} we have the following expansion:
\begin{equation} \label{22533}
I_h=o((\frac{\mu_\ell}{\mu_{\ell -1}})^\Gamma),\qquad
I_{\ell-1} =-\frac{2N}{N-2} A_4 (\frac{\mu_\ell}{\mu_{\ell-1}})^\Gamma +o((\frac{\mu_\ell}{\mu_{\ell -1}})^\Gamma).
\end{equation}
Summing up \eqref{22511} and \eqref{22533} we get that the third term in $a_\ell$, $\ell=2,\ldots,k$, takes the form
$$-\frac{N-2}{2N} \sum_{h=0}^\ell I_h= A_4 (\frac{\mu_\ell}{\mu_{\ell-1}})^\Gamma +o((\frac{\mu_\ell}{\mu_{\ell -1}})^\Gamma),$$
which, along with \eqref{a1bis}-\eqref{12422}, finally establishes the validity of \eqref{aelle} for $a_\ell$, $\ell=2,\ldots,k$.

\subsection{The remainder term: proof of Proposition \ref{phij}}
We assume that either $\ell=1$ or $\ell\geq 2$ and $C^1-$maps $(d_1,\ldots,d_k) \in A \to \phi_{j,\eps}(d_1,\ldots,d_j) \in \mathcal K_j^\bot$ have already been constructed for $j=1,\ldots,\ell-1$ satisfying the properties of Proposition \ref{phij}. By Lemma \ref{lineare} we can rewrite the equation $\mathcal E_\ell +\mathcal L_\ell(\phi_{\ell})+\mathcal N_\ell (\phi_{1,\eps}, \ldots,\phi_{\ell-1,\eps}, \phi_{\ell})=0$ as 
 $$\phi_{\ell} =-\mathcal L^{-1}_\ell\left(\mathcal E_\ell +\mathcal N_\ell(\phi_{1,\eps}, \ldots, \phi_{\ell-1,\eps} , \phi_{\ell})\right)=\mathcal T_\ell(\phi_{\ell}).$$
Given $R>0$ large, we show below that $\mathcal T_\ell:\mathcal B_\ell \to\mathcal B_\ell= \{\phi \in \mathcal K_\ell^\bot:\ \|\phi\| \leq R R_\ell \}$ is a contraction mapping for $\eps$ small, where 
\begin{equation} \label{15022}
R_\ell=\left\{\begin{array}{ll}
\|\mathcal E_1\| &\hbox{if }\ell=1\\
({\mu_\ell\over\mu_{\ell-1}} )^{\Gamma} +({\mu_\ell\over\mu_{\ell-1}} )^{\frac{\Gamma}{2}+1}+ ({\mu_\ell \over \mu_{\ell-1}})^{\frac{N+2}{2(N-2)}\Gamma} \log^{\frac 23} \frac{1}{\mu_\ell} &\hbox{if }\ell=2,\ldots,k. \end{array} \right.
\end{equation}
Hence, for $\epsilon>0$ small it follows the existence of a unique fixed point $\phi_{\ell,\eps}(d_1,\ldots,d_\ell) \in \mathcal B_\ell$ for any $(d_1,\ldots,d_k) \in A$. By the Implicit Function Theorem it is possible to show that $(d_1,\ldots,d_k)\in A \to \phi_{\ell,\eps}(d_1,\ldots,d_\ell)$ is a $C^1-$map satisfying also \eqref{stimaphi2}. Since the proof can be made similarly as in \cite{MPV} we omit it. The validity of \eqref{stimalinfty} will be addressed at the end of this section.

\medskip \noindent Set $\mathcal N_\ell(\phi)=\mathcal N_\ell(\phi_{1, \eps}, \ldots ,\phi_{\ell-1,\eps},\phi)$. Since by Lemma \ref{lineare} 
\begin{eqnarray*}
\|\mathcal T_\ell(\phi )\| \leq c \left(\|\mathcal E_\ell\|+ \|\mathcal N_\ell(\phi)\|\right), \quad  \|\mathcal T_\ell(\phi_1)-\mathcal T_\ell(\phi_2)\| \leq c \|\mathcal N_\ell( \phi_1)-\mathcal N_\ell(\phi_2)\|, \end{eqnarray*} 
by Lemma \ref{errorej} and \eqref{mujbis}-\eqref{muj2} it is enough to show that
\begin{eqnarray}\label{cont1}
\|\mathcal N_\ell(\phi)\|=O(R_\ell)+o(\|\phi\|), \quad \|\mathcal N_\ell(\phi_1)-\mathcal N_\ell(\phi_2)\| = o(1) \|\phi_1- \phi_2\| \end{eqnarray}
uniformly for any $\phi,\phi_1,\phi_2 \in \mathcal B_\ell$. Let $f(u)=|u|^{\frac{4}{N-2}}u$ and set 
\begin{eqnarray*}
N_\ell&=& f\(\sum_{j=1}^\ell (-1)^j PU_j+\sum_{j=1}^{\ell-1} \phi_{j,\eps}+\phi\)-f\(\sum_{j=1}^\ell (-1)^j PU_j\)-f'\(\sum_{j=1}^\ell (-1)^j PU_j\)\phi\\
&&-f\(\sum_{j=1}^{\ell-1} [(-1)^j PU_j+\phi_{j,\eps}]\)+f\(\sum_{j=1}^{\ell-1}(-1)^j PU_j \).
\end{eqnarray*}
First, by \eqref{18355} for $\ell=1$ we have that
\begin{eqnarray*}
\|\mathcal N_1 (\phi)\| & \leq & c | N_1 |_{\frac{2N}{N+2}} = c
|f(-PU_1+\phi)-f(-PU_1)-f'(- PU_1)\phi |_{\frac{2N}{N+2}} \\
&\leq&
c ( |\phi |_{\frac{2N}{N-2}}^{\frac{N+2}{N-2}}+\underbrace{|U_1^{\frac{6-N}{N-2}} \phi^2|_{\frac{2N}{N+2}}}_{\hbox{if }3 \leq N \leq 5} ) =o(\|\phi\|)
\end{eqnarray*}
and then the first in \eqref{cont1} is established. For $\ell \geq 2$, by \eqref{18355} we have the expansion
\begin{equation} \label{12511}
\begin{aligned}
 N_\ell &= f \(\sum_{j=1}^\ell (-1)^j PU_j+\sum_{j=1}^{\ell-1} \phi_{j,\eps}\)-f\(\sum_{j=1}^\ell (-1)^j PU_j\)
-f\(\sum_{j=1}^{\ell-1} [(-1)^j PU_j+\phi_{j,\eps}] \) \\
&+f\(\sum_{j=1}^{\ell-1}(-1)^j PU_j\)+\[f' \(\sum_{j=1}^\ell (-1)^j PU_j+\sum_{j=1}^{\ell-1} \phi_{j,\eps}\)-f' \(\sum_{j=1}^\ell (-1)^j PU_j\)\]\phi\\
&+O(|\phi|^{\frac{N+2}{N-2}})+\underbrace{ O( \sum_{j=1}^\ell (PU_j)^{\frac{6-N}{N-2}} \phi^2 +\sum_{j=1}^{\ell-1} |\phi_{j,\eps}|^{\frac{6-N}{N-2}} \phi^2)}_{\hbox{if }3\leq N\leq 5}
\end{aligned}
\end{equation}
Letting $\mathcal A_h$ be as in \eqref{anelli}, we have that 
\begin{equation} \label{10222}
\|\mathcal N_\ell (\phi)\|\leq c \displaystyle  \sum_{h=0}^\ell | N_\ell |_{\frac{2N}{N+2},\mathcal A_h}.
\end{equation} 
By \eqref{18355} and
$$|a+b|^{\frac{4}{N-2}}-|a|^{\frac{4}{N-2}}=O(|b|^{\frac{4}{N-2}}+\underbrace{|a|^{\frac{6-N}{N-2}} |b|}_{\hbox{if }3\leq N\leq 5}),$$
for $h=0,\ldots,\ell-1$ we have 
\begin{equation}\label{Nl}
\begin{aligned}
| N_\ell |_{\frac{2N}{N+2},\mathcal A_h} & \leq c \Big| U_l^{\frac{N+2}{N-2}}+U_l \sum_{j=1}^{\ell-1}[U_j^{\frac{4}{N-2}}+|\phi_{j,\eps}|^{\frac{4}{N-2}}]+|\phi| \sum_{j=1}^{\ell-1} |\phi_{j,\eps}|^{\frac{4}{N-2}}+|\phi|^{\frac{N+2}{N-2}} \Big|_{\frac{2N}{N+2},\mathcal A_h}\\
&+c \underbrace{ \Big|  \sum_{j=1}^\ell U_j^{\frac{6-N}{N-2}} \phi^2 +\sum_{j=1}^{\ell-1} |\phi_{j,\eps}|^{\frac{6-N}{N-2}} \phi^2 +\sum_{i=1}^\ell \sum_{j=1}^{\ell-1}U_i^{\frac{6-N}{N-2}} |\phi|  |\phi_{j,\eps}| \Big|_{\frac{2N}{N+2},\mathcal A_h}}_{\hbox{if }3\leq N\leq 5}\\
&=O\( R_\ell+ \sum_{j=1}^{\ell-1} | |\phi_{j,\eps}|^{\frac{4}{N-2}} U_l   |_{\frac{2N}{N+2},\mathcal A_h}\)+
o(\|\phi\|)
\end{aligned}
\end{equation}
and
\begin{equation}\label{Nlbis}
\begin{aligned}
& | N_\ell |_{\frac{2N}{N+2},\mathcal A_\ell}  \leq c \Big| U_l^{\frac{4}{N-2}}\sum_{j=1}^{\ell-1} |\phi_{j,\eps}|
+\sum_{j=1}^{\ell-1} |\phi_{j,\eps}|^{\frac{N+2}{N-2}}+|\phi|^{\frac{N+2}{N-2}} \Big|_{\frac{2N}{N+2},\mathcal A_\ell}\\
&+c \underbrace{ \Big| U_l  \sum_{i,j=1}^{\ell-1} U_i^{\frac{6-N}{N-2}} |\phi_{j,\eps}|
+\sum_{i=1}^{\ell} \sum_{j=1}^{\ell-1}  U_i^{\frac{6-N}{N-2}} \phi_{j,\eps}^2+
 \sum_{j=1}^\ell U_j^{\frac{6-N}{N-2}} \phi^2 +\sum_{j=1}^{\ell-1} |\phi_{j,\eps}|^{\frac{6-N}{N-2}} \phi^2 \Big|_{\frac{2N}{N+2},\mathcal A_\ell}}_{\hbox{if }3\leq N\leq 5}\\
&=O\( \sum_{j=1}^{\ell-1} | U_\ell^{\frac{4}{N-2}} \phi_{j,\eps} |_{\frac{2N}{N+2},\mathcal A_\ell}+ \sum_{j=1}^{\ell-1} |  \phi_{j,\eps}|^{\frac{N+2}{N-2}} _{\frac{2N}{N-2},\mathcal A_\ell}
+\sum_{j=1}^{\ell-1} |  \phi_{j,\eps}|^2 _{\frac{2N}{N-2},\mathcal A_\ell} \)\\
&+\underbrace{O\( \sum_{i,j=1}^{\ell-1}  | U_\ell^{\frac{4}{N-2}} \phi_{j,\eps} |^{\frac{N-2}{4}}_{\frac{2N}{N+2},\mathcal A_\ell}
|U_i|^{\frac{6-N}{N-2}}_{\frac{2N}{N-2},\mathcal A_\ell} |  \phi_{j,\eps}|^{\frac{6-N}{4}}_{\frac{2N}{N-2},\mathcal A_\ell}\)}_{\hbox{if }3\leq N\leq 5}+o(\|\phi\|)
\end{aligned}
\end{equation}
for any $\phi \in \mathcal B_\ell$ in view of \eqref{ok1}-\eqref{ok22} and H\"older inequality, where $R_\ell$ is given in \eqref{15022}. Notice in the estimate \eqref{Nl} we couple the first/second term in the expression \eqref{12511} of $N_\ell$ with the third/fourth one, while in the estimate \eqref{Nlbis} the first two and the second two terms in \eqref{12511} are coupled.

\medskip \noindent For $j=1,\ldots, \ell-1 $  there holds $\mathcal A_\ell \subset B_{\rho \mu_j}(0)$
and by \eqref{stimalinfty} we deduce that 
\begin{equation}
| \phi_{j,\eps}|_{\frac{2N}{N-2},\mathcal A_\ell} \le \frac{c}{\mu_j^\Gamma} 
| \frac{1}{|x|^{\beta_-}}|_{\frac{2N}{N-2},\mathcal A_\ell}
\leq c (\frac{\sqrt{\mu_{\ell-1}\mu_\ell}}{\mu_j})^\Gamma 
=O\((\frac{\mu_\ell}{\mu_{\ell-1}})^{\frac{\Gamma}{2}}\)
\label{13211}
\end{equation}
and
\begin{equation} \label{15599} 
\begin{aligned}
&| U_\ell^{\frac{4}{N-2}} \phi_{j,\eps} |_{\frac{2N}{N+2},\mathcal A_\ell}
\leq \frac{c}{\mu_j^\Gamma} 
| \frac{U_\ell^{\frac{4}{N-2}}}{|x|^{\beta_-}}|_{\frac{2N}{N+2},\mathcal A_\ell}\\
& \leq \left\{\begin{array}{ll}
c (\frac{\mu_\ell}{\mu_j})^\Gamma 
| \frac{U^{\frac{4}{N-2}}}{|y|^{\beta_-}}|_{\frac{2N}{N+2},\frac{\mathcal A_\ell}{\mu_\ell}}
=O\((\frac{\mu_\ell}{\mu_{\ell-1}})^{\Gamma} \)
&\hbox{if }3\leq N\leq 5\\
c (\frac{\mu_\ell}{\mu_j})^\Gamma 
| \frac{U^{\frac{4}{N-2}}}{|y|^{\beta_-}}|_{\frac{2N}{N+2},\frac{\mathcal A_\ell}{\mu_\ell}}
=O\( (\frac{\mu_\ell}{\mu_{\ell-1}})^{\Gamma} \log^{\frac{2}{3}} \frac{1}{\mu_\ell}\)
&\hbox{if } N=6\\
c \frac{\mu_\ell ^{\frac{4 \Gamma}{N-2}}}{\mu_j^\Gamma} 
| \frac{1}{|x|^{\beta_-+\frac{4}{N-2} \beta_+}}|_{\frac{2N}{N+2},\mathcal A_\ell}
=O\( (\frac{\mu_\ell }{\mu_{\ell-1}})^{\frac{N+2}{N-2}\frac{\Gamma}{2}}\)
&\hbox{if }N\geq 7.
\end{array} \right.
\end{aligned}\end{equation}
For $h=0,\ldots,\ell-2$ by \eqref{ok1} we deduce
\begin{equation} \label{17000}
|| \phi_{j,\eps}|^{\frac{4}{N-2}} PU_\ell  |_{\frac{2N}{N+2}, \mathcal A_h} \le c \|\phi_{j,\eps}\|^{\frac{4}{N-2}}   |  PU_\ell 
|_{\frac{2N}{N-2}, \mathcal A_h}
\le   c  \left(\frac{\mu_\ell}{\sqrt{\mu_h \mu_{h+1}}}\right)^{\Gamma} = O\( (\frac{\mu_\ell}{\mu_{\ell-1}})^\Gamma \)
\end{equation}
for any $j=1,\ldots,\ell-1$. Splitting $\mathcal A_{\ell-1}$ as $\mathcal A'_{\ell-1}\cup \mathcal A''_{\ell-1}$, where $\mathcal A'_{\ell-1}=\mathcal A_{\ell-1}\cap B_{ \rho \mu_{\ell-1}}(0)$ and $\mathcal A''_{\ell-1}=\mathcal A_{\ell-1}\setminus B_{ \rho \mu_{\ell-1}}(0)$, by \eqref{stimalinfty} and \eqref{ok1} we get that
\begin{equation} \label{17002}
\begin{aligned}
& | |\phi_{j,\eps}|^{\frac{4}{N-2}} PU_\ell |_{\frac{2N}{N+2}, \mathcal A_{\ell-1}}  \le \frac{c}{\mu_j^{\frac{4 \Gamma}{N-2} }} |\frac{PU_\ell}{|x|^{\frac{4\beta_-}{N-2}}}|_{\frac{2N}{N+2}, \mathcal A'_{\ell-1}}
+c \|\phi_{j,\eps} \|^{\frac{4}{N-2}}|P U_\ell|_{\frac{2N}{N-2}, \mathcal A''_{\ell-1} }\\
& \leq  
\frac{c}{ \mu_j^{\frac{4\Gamma}{N-2}}} |PU_\ell|_{\frac{2N}{N+2}, \mathcal A'_{\ell-1}} \left\{ \begin{array}{ll}
\mu_{\ell-1}^{\frac{4\Gamma}{N-2}-2}&\hbox{if }\Gamma \geq \frac{N-2}{2}\\
(\mu_{\ell-1}\mu_\ell)^{\frac{2\Gamma}{N-2}-1}&\hbox{if }\Gamma <\frac{N-2}{2}
\end{array} \right.
+c (\frac{\mu_\ell}{\mu_{\ell-1}})^{\Gamma} \\
&=O\(  (\frac{\mu_\ell}{\mu_{\ell-1}})^{\frac{N+2}{N-2}\frac{\Gamma}{2}}+  (\frac{\mu_\ell}{\mu_{\ell-1}})^{\frac{\Gamma}{2}+1} \)
\end{aligned}
\end{equation}
in view of 
\begin{equation}\label{ok31}
\begin{aligned}
& |PU_\ell  |_{\frac{2N}{N+2}, \mathcal A'_{\ell-1}}\leq
 \mu_\ell^2 | U   |_{\frac{2N}{N+2}, \mathbb R^N \setminus B_{\sqrt{\frac{\mu_{\ell-1}}{\mu_\ell}}}(0)}=O\(
\mu_\ell^2 (\frac{\mu_\ell}{\mu_{\ell-1}})^{\frac{\Gamma}{2}-1}\)\\
&|PU_\ell  |_{\frac{2N}{N-2}, \mathcal A''_{\ell-1} }\leq
 | U   |_{\frac{2N}{N-2}, \mathbb R^N \setminus B_{ \frac{\mu_{\ell-1}}{\mu_\ell}}(0)}=O\((\frac{\mu_\ell}{\mu_{\ell-1}})^\Gamma\).
\end{aligned}
\end{equation}
Estimates \eqref{ok1}, \eqref{13211}-\eqref{15599} into \eqref{Nlbis} and \eqref{17000}-\eqref{17002} into \eqref{Nl} lead to 
$| N_\ell |_{\frac{2N}{N+2},\mathcal A_h}=O(R_\ell)+o(\|\phi\|)$ for any $h=0,\ldots,\ell$, which, inserted into \eqref{10222}, finally give the validity of the first in \eqref{cont1} for $\ell\geq 2$.

\medskip \noindent Concerning the second one in \eqref{cont1}, we have that
\begin{equation*}
\begin{aligned}
\|\mathcal N_\ell(\phi_1)-\mathcal N_\ell(\phi_2)\|  & \leq c \left| [\sum_{j=1}^{\ell-1}|\phi_{j, \eps}|^{\frac{4}{N-2}}+|\phi_1|^{\frac{4}{N-2}}+|\phi_2|^{\frac{4}{N-2}}] (\phi_1-\phi_2)\right|_{\frac{2N}{N+2}}\\
& +\underbrace{ c \left| \sum_{i=1}^\ell U_i^{\frac{6-N}{N-2}}
[\sum_{j=1}^{\ell-1}   |\phi_{j, \eps}|+|\phi_1|+|\phi_2|]  (\phi_1-\phi_2) \right|_{\frac{2N}{N+2}}}_{\hbox{if }3\leq  N \leq 5} \end{aligned}
\end{equation*}
in view of
\begin{eqnarray*}
&&|a+b+c_1|^{\frac{4}{N-2}}(a+b+c_1)- |a+b+c_2|^{\frac{4}{N-2}}(a+b+c_2)-\frac{N+2}{N-2} |a|^{\frac{4}{N-2}}(c_1-c_2)\\
&& =|c_1-c_2| O\( |b|^{\frac{4}{N-2}}+|c_1|^{\frac{4}{N-2}}+|c_2|^{\frac{4}{N-2}}
+ \underbrace{|a|^{\frac{6-N}{N-2}}(|b|+|c_1|+|c_2|)}_{\hbox{if }3\leq N \leq 5} \)
\end{eqnarray*}
for all $a,b,c_1,c_2 \in \mathbb R$. Therefore there holds
\begin{equation} \label{17557}
\begin{aligned}
\|\mathcal N_\ell(\phi_1)-\mathcal N_\ell(\phi_2)\|  & \leq c \[\sum_{j=1}^{\ell-1}\| \phi_{j, \eps}\|^{\frac{4}{N-2}}+\| \phi_1\|^{\frac{4}{N-2}}+\| \phi_2\|^{\frac{4}{N-2}} \]
\|\phi_1-\phi_2\| \\
& +\underbrace{ c  \sum_{i=1}^\ell |U_i|^{\frac{6-N}{N-2}}_{\frac{2N}{N-2}}
\[\sum_{j=1}^{\ell-1}   \|\phi_{j, \eps}\|+\| \phi_1 \|+\| \phi_2\|  \] \|\phi_1-\phi_2\|}_{\hbox{if }3\leq  N \leq 5}\\ 
&=o(1) \|\phi_1-\phi_2\|.
\end{aligned}
\end{equation}
in view of $\phi_1,\phi_2 \in \mathcal B_\ell$. The validity of \eqref{cont1} has been fully established.

\medskip \noindent To prove the validity of \eqref{stimalinfty}, assume that either $\ell=1$ or $\ell\geq 2$ and $C^1-$maps $(d_1,\ldots,d_k) \in A \to \phi_{j,\eps}(d_1,\ldots,d_j) \in \mathcal K_j^\bot$ have already been constructed for $j=1,\ldots,\ell-1$ satisfying the properties of Proposition \ref{phij}. Setting $u_j=(-1)^j PU_j +\phi_{j, \eps}$, $j=1,\ldots, \ell$, we have that $u_j$ satisfies
\begin{equation}
\label{18222}
u_j=\i^* \[f(\sum_{i=1}^j u_i)-f (\sum_{i=1}^{j-1} u_i)+\e u_j \]+\Psi_j, \quad \Psi_j \in \mathcal K_j,
\end{equation}
for any $j=1,\ldots,\ell$, and then 
\begin{equation}
\label{18222bis}
\sum_{j=1}^\ell u_j=\i^* \[f(\sum_{i=1}^\ell u_i)+\e \sum_{j=1}^\ell u_j \]+\sum_{j=1}^\ell \lambda_{j , \eps} P Z _j
\end{equation}
in view of $\displaystyle \sum_{j=1}^\ell \Psi_j \in \mathcal K_\ell$. We claim that $ \lambda_{j,\eps}=o(1)$   as $\eps\to 0$ for any $j=1,\ldots, \ell$. Indeed, let us take the inner product of \eqref{18222bis} against $PZ _i$, $i=1,\dots,\ell$, to get
\begin{equation}\label{eqMint1-2}
\begin{aligned}
\sum_{j=1}^\ell \lambda_{j, \eps}\langle P Z _j,PZ_i \rangle  
& =\int_\Omega [\sum_{j=1}^\ell (-1)^j  U_j^{\frac{N+2}{N-2}}-f(\sum_{j=1}^\ell u_j)] PZ _i\, dx  +\sum_{j=1}^{i-1} \langle  \phi_{j, \eps}, PZ_i \rangle\\
& -\e\sum_{j=1}^\ell\int_\Omega u_j  PZ _i\, dx
\end{aligned}
\end{equation}
in view of $\phi_{j, \eps}\in \mathcal K_i^\bot$ for any $j \geq i$ and $PU_j=\i^*(U_j^{\frac{N+2}{N-2}})$. By Proposition \ref{prop-proZ} and \eqref{18377} we have that
\begin{equation}\label{11025}
\begin{aligned}
\int_\Omega (PZ_i)^{\frac{2N}{N-2}} \ dx & =\int_\Omega Z_i^{\frac{2N}{N-2}} \ dx+O\( 
\mu_i^\Gamma \int_\Omega \frac{Z_i^{\frac{N+2}{N-2}}}{|x|^{\beta_-}} \ dx+
\int_\Omega \frac{\mu_i^{\frac{2N}{N-2}\Gamma}}{|x|^{\frac{2N}{N-2}\beta_-}}\ dx \)\\
&=
\int_{\frac{\Omega}{\mu_i}} Z^{\frac{2N}{N-2}} \ dy+O\( 
\mu_i^{2 \Gamma} \int_{\frac{\Omega}{\mu_i}} \frac{Z^{\frac{N+2}{N-2}}}{|y|^{\beta_-}} \ dy\)+o(1)
=\int_{\mathbb{R}^N } Z^{\frac{2N}{N-2}} \ dy+o(1) 
\end{aligned}
\end{equation}
and
\begin{equation}\label{czero}
\begin{aligned}
 \langle PZ _j,PZ _i \rangle&= \frac{N+2}{N-2} \int_\Omega U_j^{\frac{4}{N-2}} Z _jPZ _i\, dx  \\
&=  \frac{N+2}{N-2} \int_\Omega U_j^{\frac{4}{N-2}} Z _j Z _i\, dx+
O\( \mu_i^\Gamma \int_\Omega \frac{U_j^{\frac{4}{N-2}} Z _j}{|x|^{\beta_-}}  \, dx \) \\
&= \frac{N+2}{N-2} \delta_{ij} \int_{\frac{\Omega}{\mu_j}} U^{\frac{4}{N-2}} Z ^2  \, dy +
O\( \mu_i^\Gamma \mu_j^\Gamma \int_{\frac{\Omega}{\mu_j}} \frac{U^{\frac{4}{N-2}} Z }{|y|^{\beta_-}}  \, dy \)+o(1)\\
& = \frac{N+2}{N-2} \delta_{ij} \int_{\mathbb{R}^N} U^{\frac{4}{N-2}} Z ^2  \, dy +o(1)
 \end{aligned}\end{equation}
in view of  $PZ_j=\i^* (\frac{N+2}{N-2} U_j^{\frac{4}{N-2}}  Z_j)$, $|Z_i|\leq U_i$ and
\begin{equation*}
|\int_\Omega U_j^{\frac{4}{N-2}} Z _j  Z _i \ dx | \leq c \int_\Omega U_j^{\frac{N+2}{N-2}} U_i \ dx
\leq \left\{ \begin{array}{ll}
c (\frac{\mu_j}{\mu_i})^\Gamma \int_{\mathbb R^N}\frac{U^{\frac{N+2}{N-2}}}{|y|^\bm} \, dy & \hbox{if }j>i\\
c (\frac{\mu_i}{\mu_j})^\Gamma \int_{\mathbb R^N}\frac{U^{\frac{N+2}{N-2}}}{|y|^\bp} \ dy
& \hbox{if }j<i.
\end{array} \right.
\end{equation*}
By inserting \eqref{11025}-\eqref{czero} into \eqref{eqMint1-2} we get that
\begin{equation}\label{1258}
\begin{aligned}
\frac{N+2}{N-2} \( \int_{\mathbb{R}^N} U^{\frac{4}{N-2}} Z ^2  \, dy\)  \lambda_{i, \eps}
&=
\int_\Omega [\sum_{j=1}^\ell (-1)^j  (PU_j)^{\frac{N+2}{N-2}}-f(\sum_{j=1}^\ell (-1)^j PU_j)] PZ _i\, dx \\ &+o(\sum_{j=1}^\ell |\lambda_{j, \eps}|)+
o(1)
\end{aligned}
\end{equation}
in view of \eqref{stimaphi1}, \eqref{113}-\eqref{112}, \eqref{18355} and $\| PU_j \|=O(1)$. We have that
$$\begin{aligned}
&\int_\Omega [\sum_{j=1}^\ell (-1)^j  (PU_j)^{\frac{N+2}{N-2}}-f(\sum_{j=1}^\ell (-1)^j PU_j)] PZ _i\, dx\\ &=-\sum_{j=1}^\ell \int_\Omega \[ f\(\sum_{i=1}^j (-1)^i PU_i\) -f\(\sum_{i=1}^{j-1} (-1)^i PU_i\)-(-1)^j  (PU_j)^{\frac{N+2}{N-2}}\] PZ _i\, dx \\
&=
O \( \sum_{h=0}^\ell \Big| f\(\sum_{i=1}^j (-1)^i PU_i\) -f\(\sum_{i=1}^{j-1} (-1)^i PU_i\)-(-1)^j  (PU_j)^{\frac{N+2}{N-2}}\Big|_{\frac{2N}{N+2},\mathcal A_h}\)
\end{aligned} $$
in view of \eqref{11025}, with $\mathcal A_h$ given as in \eqref{anelli}. By \eqref{15066}-\eqref{ok2222} we deduce that
$$\int_\Omega [\sum_{j=1}^\ell (-1)^j  (PU_j)^{\frac{N+2}{N-2}}-f(\sum_{j=1}^\ell (-1)^j PU_j)] PZ _i\, dx=o(1),$$
and then \eqref{1258} reduces to
$$\frac{N+2}{N-2} \( \int_{\mathbb{R}^N} U^{\frac{4}{N-2}} Z ^2  \, dy\)  \lambda_{i, \eps}
=o(\sum_{j=1}^\ell |\lambda_{j, \eps}|)+o(1).$$
This in turn implies that $\displaystyle \sum_{j=1}^\ell |\lambda_{j, \eps}|=o(1)$, and the claim is established.
 
\medskip \noindent The function $\mathcal U_\ell(y)=\mu_\ell^{\frac{N-2}{2}}  (\displaystyle \sum_{j=1}^\ell u_j)(\mu_\ell y )$ solves
\begin{equation}\label{eqM2ell}
\begin{aligned}
&-\Delta \mathcal U_\ell- \frac{\gamma}{|y|^2} \mathcal U_\ell -\eps\mu_\ell^2 \mathcal U_\ell- \mathcal U_\ell^{\frac{N+2}{N-2}}=h \quad \hbox{in } \frac{\Omega}{\mu_\ell}
\end{aligned}
\end{equation}
in view of \eqref{18222bis}, where 
$$h=O\( \sum_{j=1}^{\ell} |\lambda_{j,\eps}| (\frac{\mu_\ell}{\mu_j})^{\frac{N+2}{2}} 
\ U^{\frac{N+2}{N-2}} (\frac{\mu_\ell}{\mu_j}y) \).$$ 
We have that
$$|y|^\tau |h(y)|=O(\sum_{j=1}^{\ell} |\lambda_{j,\eps}| (\frac{\mu_\ell}{\mu_j})^{\frac{N+2}{N-2}\Gamma}) =O(1)$$
with $\tau=\frac{N+2}{N-2}\beta_-<\beta_-+2$ and
$$\(\int_{B_r(0)}  |\mathcal U_\ell |^{\frac{2N}{N-2}} \ dy\)^{\frac{N-2}{2N}} \leq \ell \(\int_{B_r(0)} U^{\frac{2N}{N-2}} \ dy\)^{\frac{N-2}{2N}}+\sum_{j=1}^\ell \|\phi_{j,\epsilon}\|^{\frac{2N}{N-2}} \leq \epsilon$$
in view of $B_{r \mu_\ell} \subset B_{r \mu_j}$ for any $j=1,\ldots,\ell-1$ and \eqref{stimaphi1}, for some $r=r(\epsilon)$. We are in position to apply Proposition \ref{astratto} below to get the existence of $\rho,K>0$ such that
$$|y|^\bm |\mathcal U_\ell(x)|\le K$$
for all $x \in B_\rho(0)$, or equivalently
$$|x|^\bm  |\sum_{j=1}^\ell u_j (x)| \le \frac{K}{\mu_\ell^{\Gamma}} \quad \mbox{ in } B_{\rho \mu_\ell}(0).$$
Since by assumption for any $j=1,\ldots,\ell-1$
$$|u_j|\leq PU_j+|\phi_{j,\e}|\leq \frac{C}{\mu_j^\Gamma |x|^{\beta_-}}\leq \frac{C}{\mu_\ell^\Gamma |x|^{\beta_-}}$$
in $B_{\rho \mu_j}(0)$ with $B_{\rho \mu_\ell}(0)\subset B_{\rho \mu_j}(0)$, we deduce that
$|u_\ell| \leq \frac{C}{\mu_\ell^\Gamma |x|^{\beta_-}}$ and then $|\phi_{\ell,\e}| \leq \frac{C}{\mu_\ell^\Gamma |x|^{\beta_-}}$ in $B_{\rho \mu_\ell}(0)$, and \eqref{stimalinfty} is established.

\medskip \noindent The following result is established using the same scheme as in \cite{GhRo} and for convenience we reproduce it here. 
\begin{proposition}\label{astratto}
Let $M>0$ and $\tau< \beta_-+2$. There exist $\varepsilon,\rho,K>0$ so that 
\begin{equation}\label{tesiastratto}
\sup_{x \in B_\rho(0)}|x|^\bm |u(x)|\le K
\end{equation}
does hold for any solution $u$ of
\begin{equation}\label{operatore}
-\Delta u -\frac{\gamma}{|x|^2} u=a u+|u|^{\frac{4}{N-2}}u+ h \mbox{ in }B_1(0),\quad u \in H^1(B_1(0)),
\end{equation}
with 
\begin{eqnarray}\label{ipotesiastratto}
&& |u|_{\frac{2N}{N-2},B_1(0)} \leq \varepsilon\\
\label{iph}
&& |a|_{\infty,B_1(0)} +\sup_{x \in B_1(0)}|x|^\tau |h(x)|\leq M.
\end{eqnarray}
\end{proposition} 
\begin{proof} We need some preliminary facts.

\medskip \noindent \underline{{\bf $1^{\hbox{st}}$ Claim}}: Let $M>0$ and $q>2 $ with $\frac{4(q-1)}{q^2}>\frac{4\gamma}{(N-2)^2}$. There exist $\epsilon,K>0$ so that for any $0<\rho_2<\rho_1\leq 1$ there holds
\begin{equation}\label{11255}
|u|_{B_{\rho_2}(0), \frac{Nq}{N-2}} \leq K \[  |u|_{B_{\rho_1}(0),q}+
|h(x)|_{B_{\rho_1}(0),\frac{Nq}{N-2+2 q}}\]
\end{equation}
for any solution $u \in L^q (B_{\rho_1}(0))$ of \eqref{operatore} so that \eqref{ipotesiastratto}, $h^{\frac{N}{N-2+2 q}} \in L^q (B_{\rho_1}(0))$ and $|a|_{\infty,B_{\rho_1}(0)} \leq M$ do hold.

\medskip \noindent Indeed, given $L>0$ define 
$$G_L(t)=\left\{
\begin{array}{ll}
|t|^{q-2} t &\mbox{if } |t|\le L\\ 
(q-1) L^{q-2}t-(q-2) L^{q-1}\ \hbox{sign} \ t &\mbox{if } |t| > L\\
\end{array}\right. 
$$
and
$$H_L(t)=\left\{
\begin{array}{ll}
|t|^{\frac{q}{2}} &\mbox{if } |t|\le L\\ 
\frac{q}{2} L^{\frac{q-2}{2}}|t|-\frac{q-2}{2} L^{\frac{q}{2}} &\mbox{if } |t| > L
\end{array} \right. $$
in such a way that $H_L,G_L \in C^1(\mathbb R)$ satisfy
\begin{eqnarray}\label{1247bis}
G_L'(t)=\frac{4(q-1)}{q^2}[H_L'(t)]^2, \quad t \in \mathbb R.
\end{eqnarray}
Observe that for all $t\in\mathbb R$ there hold
\begin{eqnarray}\label{propGLHL}
0\leq t G_L(t)\leq H_L^2(t), \qquad |G_L(t)|  \leq H_L^{\frac{2 (q-1)}{q}}  (t).
\end{eqnarray}
Given $0<\rho_2<\rho_1\leq 1,$ let $\eta\in C^\infty_c(\mathbb R^N)$ be so that $\eta=1$ in $B_{\rho_2}(0)$ and $\eta=0$ in $ \mathbb R^N \setminus B_{\rho_1}(0)$. Test \eqref{operatore} against $\eta^2 G_L(u)$ to get
\begin{equation}\label{mul}
\begin{aligned}
& \int_{B_1(0)} \langle \nabla u, \nabla (\eta^2 G_L(u)) \rangle dx-\int_{B_1(0)}\frac{\gamma}{|x|^2}\eta^2 uG_L(u)\, dx \\
&=\lambda\int_{B_1(0)}\eta^2 u G_L(u)\, dx +\int_{B_1(0)}\eta^2 |u|^{\frac{4}{N-2}}u G_L(u)\, dx + \int_{B_1(0)}\eta^2  h(x) G_L(u)\, dx.
\end{aligned}
\end{equation}
By \eqref{1247bis} an integration by parts leads to
\begin{equation}\label{mul1}
\begin{aligned}
& \int_{B_1(0)} \langle \nabla u, \nabla (\eta^2 G_L(u)) \rangle dx =\frac{4(q-1)}{q^2}\int_{B_1(0)} |\nabla (\eta H_L(u))|^2 \\
& +\frac{4(q-1)}{q^2}\int_{B_1(0)} \eta  \Delta \eta H_L^2(u) \, dx  -\int_{B_1(0)} \Delta(\eta^2) J_L(u)\, dx
\end{aligned}
\end{equation}
where $J_L(t)=\int_0^t G_L(\tau)\, d\tau.$ Inserting \eqref{mul1} into \eqref{mul} we get
\begin{equation}\label{mul2}
\begin{aligned}
&\frac{4\alpha}{(\alpha+1)^2}\int_{B_1(0)}|\nabla (\eta H_L(u))|^2\, dx -\int_{B_1(0)}\frac{\gamma}{|x|^2}\eta^2 uG_L(u)\, dx\\
&\le K   \int_{B_{\rho_1}(0)}[ H_L^2(u)+ J_L(u)]\ dx
+K   \int_{B_1(0)}\left\{  |u|^{\frac{4}{N-2}} [\eta H_L(u)]^2+\eta^2  |h(x)| |G_L(u)| \right\} \ dx
\end{aligned}
\end{equation}
in view of \eqref{propGLHL}, where $K$ denotes a generic constant just depending on $q$, $M$, $\gamma$, $N$ and $\rho_1,\rho_2$. By H\"older and Sobolev inequalities we have that
\begin{equation}\label{mul3}
\begin{aligned}
\int_{B_1(0)}|u|^{\frac{4}{N-2}} [\eta H_L(u)]^2\, dx&\le \(\int_{B_1(0)}|u|^{\frac{2N}{N-2}}  dx\)^{\frac{2}{N}}\(\int_{B_1(0)}|\eta H_L(u)|^{\frac{2N}{N-2}}\, dx\)^{\frac{N-2}{N}}\\
&\le K  \varepsilon^{\frac{4}{N-2}} \int_{B_1(0)}|\nabla (\eta H_L(u))|^2\, dx
\end{aligned}
\end{equation}
in view of \eqref{ipotesiastratto} and
\begin{equation}\label{mul4}
\begin{aligned}
& \int_{B_1(0)}\eta^2  |h(x)|| G_L(u)|\, dx \le  \int_{B_1(0)}|h(x)| [\eta H_L(u)]^{\frac{2(q-1)}{q}}  \, dx \\
&\leq K \( \int_{B_1(0)}|h(x)|^{\frac{Nq}{N-2+2q}} \ dx \)^{\frac{N-2+2q}{Nq}}
 \( \int_{B_1(0)}|\nabla (\eta H_L(u))|^2\, dx \)^{\frac{q-1}{q}}
\end{aligned}
\end{equation}
in view of \eqref{propGLHL}. Plugging \eqref{mul3}-\eqref{mul4} into \eqref{mul2} by \eqref{propGLHL} we get
\begin{eqnarray*}
&& \[\frac{4(q-1)}{q^2}-K\epsilon^{\frac{4}{N-2}}\] \int_{B_1(0)}|\nabla (\eta H_L(u))|^2\, dx -\gamma^+\int_{B_1(0)}\frac{1}{|x|^2}[\eta H_L(u)]^2\, dx\\
&&\le  K \int_{B_{\rho_1}(0)}[H_L^2(u)+J_L(u)]\, dx+
K \( \int_{B_1(0)}|h(x)|^{\frac{Nq}{N-2+2q}} \ dx \)^{\frac{N-2+2q}{Nq}}
 \( \int_{B_1(0)}|\nabla (\eta H_L(u))|^2\, dx \)^{\frac{q-1}{q}}
\end{eqnarray*}
where $\gamma^+=\max\{\gamma, 0\}$. By the Hardy inequality we finally deduce that
\begin{equation}\label{1124}
\begin{aligned}
& \[\frac{4(q-1)}{q^2}-K \varepsilon^{\frac{4}{N-2}} -\frac{4 \gamma^+}{(N-2)^2} \] \int_{B_1(0)}|\nabla (\eta H_L(u))|^2\, dx \leq K \int_{B_{\rho_1}(0)}[H_L^2(u)+J_L(u)]\, dx\\
&+
K \( \int_{B_1(0)}|h(x)|^{\frac{Nq}{N-2+2q}} \ dx \)^{\frac{N-2+2q}{Nq}}
 \( \int_{B_1(0)}|\nabla (\eta H_L(u))|^2\, dx \)^{\frac{q-1}{q}}
\end{aligned}. 
\end{equation}
Since $\frac{4(q-1)}{q^2}>\frac{4 \gamma}{(N-2)^2}$, for $\epsilon $ small we can assume that $\frac{4(q-1)}{q^2}-K \varepsilon^{\frac{4}{N-2}} -\frac{4 \gamma+}{(N-2)^2}>0$. By \eqref{1124} we deduce that
\begin{equation}\label{1125}
\begin{aligned}
&\(\int_{B_1(0)}|\eta H_L(u)|^{\frac{2N}{N-2}}\, dx\)^{\frac{N-2}{N}} \leq K \int_{B_1(0)}|\nabla (\eta H_L(u))|^2\, dx \\
& \leq K \int_{B_{\rho_1}(0)}[H_L^2(u)+J_L(u)]\, dx+
K \( \int_{B_1(0)}|h(x)|^{\frac{Nq}{N-2+2q}} \ dx \)^{\frac{N-2+2q}{N}}
\end{aligned}\end{equation}
in view of $\frac{q-1}{q}<1$ and the Sobolev inequality. Since $0\leq J_L(t)\leq t G_L(t)\leq H_L^2(t) \leq |t|^q$ does hold for all $t \in \mathbb R$ in view of \eqref{propGLHL}, by \eqref{1125} we get that
$$\(\int_{B_{\rho_2}(0)}|H_L(u)|^{\frac{2N}{N-2}}\, dx\)^{\frac{N-2}{N}} \leq K \int_{B_{\rho_1}(0)}|u|^q
\, dx+K \( \int_{B_{\rho_1}(0)}|h(x)|^{\frac{Nq}{N-2+2q}} \ dx \)^{\frac{N-2+2q}{N}}.$$
Taking the power $\frac{1}{q}$ and letting $L\to+\infty$ by the Fatou's Lemma we obtain the validity of \eqref{11255}.

\medskip \noindent \underline{{\bf $2^{\hbox{nd}}$ Claim}}: Let $1\leq q <Q$, $M>0$ and $\tau< \beta_-+2$, where
$$Q=\left\{\begin{array}{ll} +\infty& \hbox{if }\gamma\leq 0\\ \frac{N}{\beta_-}& \hbox{if }\gamma>0. \end{array} \right. $$ 
There exist $\epsilon,K>0$ so that 
\begin{equation}\label{112557}
|u|_{q,B_{\frac{1}{2}}(0)} \leq K \[ |u|_{\frac{2N}{N-2},B_1(0)}+1\]
\end{equation}
does hold for any solution $u$ of \eqref{operatore} so that \eqref{ipotesiastratto}-\eqref{iph} are valid.

\medskip \noindent Indeed, notice that for $\gamma>0$ the property $\frac{4(q-1)}{q^2}>\frac{4 \gamma}{(N-2)^2}$, $q>2$, is equivalent to $2<q<\frac{N-2}{\bm}=\frac{N-2}{N}Q$. Since
$$\sup_{q \in [1,\frac{N-2}{N}Q)} \frac{Nq}{N-2+2q}=\left\{ \begin{array}{ll} \frac{N}{2}& \hbox{if }\gamma\leq 0\\
\frac{N}{\beta_-+2}& \hbox{if }\gamma>0 \end{array} \right.< \frac{N}{\tau}$$
if $\tau<\beta_-+2$, we have that
\begin{equation} \label{1227}
|h|_{\frac{Nq}{N-2+2q},B_1(0)} \leq K(M,\tau)
\end{equation}
for any $q \in [1,\frac{N-2}{N}Q)$, $\tau<\tau_0$ and $h$ satisfying \eqref{iph}. Let $q_j=(\frac{N-2}{N})^j q$, $j \in \mathbb N$, and $r_j$ be any decreasing sequence so that $r_0=1$ and $r_k=\frac{1}{2}$. Since $q_j \to 0$ as $j \to +\infty$, we can find a smallest index $k \in \mathbb N$ so that $q_k\leq \frac{2N}{N-2}$. Notice that $q_j \leq q_1<\frac{N-2}{N} Q$ for all $j \geq 1$ and $q_k>2$ in view of $q_{k-1}>\frac{2N}{N-2}$. We can apply the $1^{\hbox{st}}$ Claim with $q_j$ between $r_{j+1}$ and $r_j$ for $j=1,\ldots,k-1$ and obtain by \eqref{1227} that for $\epsilon>0$ small
\begin{equation}\label{2144}
|u|_{q,B_{\frac{1}{2}}(0)} \leq K \[ |u|_{q_k,B_1(0)}+1\]
\end{equation}
does hold for some $K>0$. We can conclude in view of $q_k \leq \frac{2N}{N-2}$ and
$$|u|_{q_k,B_1(0)}\leq \omega_N^{\frac{2N-(N-2)q_k}{2Nq_k}} |u|_{\frac{2N}{N-2},B_1(0)}.$$

\medskip \noindent 
\underline{{\bf $3^{\hbox{rd}}$ Claim}}: Let $\frac{2N}{N-2}< q <Q$, $M>0$ and $\tau< \beta_-+2$. There exist $\epsilon,K>0$ so that 
\begin{equation}\label{0843}
\sup_{x \in B_{\frac{1}{4}}(0)}|x|^{\frac{N}{q}}|u(x)| \leq K
\end{equation}
does hold for any solution $u$ of \eqref{operatore} so that \eqref{ipotesiastratto}-\eqref{iph} are valid.

\medskip \noindent Given $\frac{2N}{N-2}< q <Q$, $M>0$ and $\tau< \beta_-+2$, choose $\epsilon>0$ small so that the $2^{\hbox{nd}}$ Claim applies. The function $U(y)=|x|^{\frac{N}{q}} u(|x|y)$ satisfies
$$-\Delta U-\frac{\gamma}{|y|^2}U=|x|^2 a(|x|y) U+|x|^{2-\frac{4N}{q(N-2)}}  |U|^{\frac{4}{N-2}} U
+|x|^{\frac{N}{q}+2} h(|x|y)\quad \hbox{in }B_2(0)\setminus B_{\frac{1}{2}}(0),$$
where  
\begin{equation} \label{0901}
\big| |x|^2 a(|x|y) U+|x|^{2-\frac{4N}{q(N-2)}}  |U|^{\frac{4}{N-2}} U \big|\leq
\frac{M}{16} |U|+4^{\frac{4N}{q(N-2)}-2} |U|^{\frac{N+2}{N-2}}
\end{equation}
and
\begin{equation} \label{0902}
|x|^{\frac{N}{q}+2} |h(|x|y)| \leq |x|^{\frac{N}{q}+2-\tau} \frac{M}{|y|^\tau}
\leq 4^{2 \tau-\frac{N}{q}-2} M
\end{equation}
for any $|x|\leq \frac{1}{4}$  and $\frac{1}{2}\leq |y|\leq 2$, in view of $\frac{N}{q}+2-\tau>\frac{N}{Q}+2-\tau \geq \beta_-+2-\tau>0$. Since
$$|U|_{q,B_2(0)\setminus B_1(0)}\leq |u|_{q,B_{\frac{1}{2}}(0)},$$
by \eqref{0901}-\eqref{0902} standard elliptic estimates apply for any $\tilde q\geq q>\frac{2N}{N-2}$ and through a bootstrap argument yield the validity of \eqref{0843} for some universal constant $K>0$. 

\medskip \noindent  To conclude the proof, let us rewrite \eqref{operatore} as 
\begin{equation} \label{operatore1}
-\Delta u -\frac{\gamma+\tilde a(x)}{|x|^2} u=h(x), \quad \tilde a(x)=|x|^2 a(x)+|x|^2 |u(x)|^{\frac{4}{N-2}}.
\end{equation}
Since $\frac{4N}{Q(N-2)}<2$, by $3^{\hbox{rd}}$ Claim and \eqref{iph} it follows that there exists $C_0,\theta>0$ such that 
\begin{equation} \label{0957}
|\tilde a(x)|\leq C_0 |x|^\theta
\end{equation}
for any $|x|\leq \frac{1}{4}$. Since $\tau<\beta_-+2$, we can fix $\alpha$ so that $\bm-\theta<\alpha<\bm$ and $\alpha>\tau-2$. Then we can find $\rho>0$ small so that $\Phi(x)= |x|^{-\bm}-|x|^{-\alpha}\geq 
\frac{1}{2}|x|^{-\bm}$ in $B_\rho(0)$ and satisfies
$$-\Delta \Phi-\frac{\gamma+\tilde a}{|x|^2} \Phi =\frac{\alpha^2-\alpha(N-2)+\gamma}{|x|^{\alpha+2}} - \frac{\tilde a}{|x|^2}\Phi \geq
\frac{\alpha^2-\alpha(N-2)+\gamma}{|x|^{2+\alpha}}- \frac{C_0}{|x|^{\bm+2-\theta}} \geq \frac{M}{|x|^\tau}$$
in $B_\rho(0)$ in view of $\alpha^2 -\alpha(N-2)+\gamma>0$. 
Since $|u(x)|\leq K \Phi(x)$ for some $K \geq 1$ and any $x \in \partial B_\rho(0)$ in view of \eqref{0843}, by \eqref{iph} we can use $K \Phi$ as a supersolution of \eqref{operatore1} with $h$ and $-h$ to get by the maximum principle $|u(x)|\leq K \Phi(x)\leq K |x|^{-\beta_-}$ for any $x \in B_\rho(0)$, as desired.
\end{proof}

\subsection{The reduced energy: end of the proof for Proposition \ref{prob-rido}}
Let us first show that $\widetilde{J}_\eps$ has the same expansion as $J_\eps.$ Setting $u_\ell=(-1)^\ell PU_\ell +\phi_{\ell, \eps}$, $\ell=1,\ldots, k$, we have that
\begin{equation}
\label{funz}
\begin{aligned} 
\widetilde{J}_\e ( \mu_1,\ldots,\mu_k) & =J_\e (\mu_1,\ldots,\mu_k)
+\sum_{\ell,i=1}^k  \[\langle u_\ell , \phi_{i, \eps} \rangle - \eps \int_\Omega u_\ell  \phi_{i, \eps} \, dx \] -\frac 12 \| \sum_{\ell=1}^k \phi_{\ell, \eps}\|^2\\
& +\frac \eps2 |\sum_{\ell=1}^k \phi_{\ell, \eps}|_2  -\frac{N-2}{2N}\int_\Omega \[ |\sum_{\ell=1}^k u_\ell|^{\frac{2N}{N-2}}
-|\sum_{\ell=1}^k (-1)^\ell PU_\ell |^{\frac{2N}{N-2}}  \]\, dx
\end{aligned}
\end{equation}
in view $\langle u+v,u+v \rangle =\langle u, u \rangle -\langle v,v \rangle+2 \langle u+v,v\rangle$ for any bi-linear form $\langle\cdot,\cdot \rangle$. By multiplying  \eqref{18222} against $\phi_{i, \eps}\in \mathcal K_\ell^\bot$, $i \geq \ell$, we get that
\begin{equation*}
\begin{aligned}
 \langle u_\ell , \phi_{i, \eps} \rangle - \eps \int_\Omega u_\ell  \phi_{i, \eps} \, dx = \int_\Omega \[ f (\sum_{j=1}^\ell u_j)-f (\sum_{j=1}^{\ell-1} u_j )\]\phi_{i, \eps}\, dx
\end{aligned}
\end{equation*}
for any $i \geq \ell$. Therefore,  \eqref{funz} reads as
\begin{equation} \label{12377}
\begin{aligned}
\widetilde{J}_\e (\mu_1,\ldots,\mu_k)=& J_\e(\mu_1,\ldots,\mu_k)+\sum_{i<\ell}(-1)^\ell \[\langle PU_\ell,\phi_{i, \eps}\rangle-\eps \int_\Omega  PU_\ell \phi_{i, \eps} \ dx\]\\
& -\frac 12 \sum_{\ell=1}^k\|\phi_{\ell, \eps}\|^2 +\frac \e 2 \sum_{\ell=1}^k|\phi_{\ell, \eps}|_2^2+\sum_{i \geq \ell}\int_\Omega \[f (\sum_{j=1}^\ell u_j) -f (\sum_{j=1}^{\ell-1} u_j)\]
\phi_{i, \eps}\, dx \\
& -\frac{N-2}{2N} \int_\Omega \[|\sum_{\ell=1}^k u_\ell|^{\frac{2N}{N-2}} -|\sum_{\ell=1}^k (-1)^\ell PU_\ell|^{\frac{2N}{N-2}} \]\, dx.
\end{aligned}
\end{equation}
Setting
\begin{equation*}
\begin{aligned}
\widetilde{\Upsilon}_\ell  &=
(-1)^\ell \( \langle PU_\ell, \sum_{i=1}^{\ell-1}  \phi_{i, \eps}\rangle-\eps \int_\Omega  PU_\ell (\sum_{i=1}^{\ell-1} \phi_{i, \eps})\) -\frac{1}{2}\|\phi_{\ell,\eps}\|^2+\frac{\eps}{2}|\phi_{\ell,\eps}|_2^2+\int_\Omega f (\sum_{j=1}^\ell u_j)  \phi_{\ell, \eps}\, dx\\
&- \frac{N-2}{2N}  \int_\Omega \[ |\sum_{j=1}^{\ell}u_j|^{\frac{2N}{N-2}}
-|\sum_{j=1}^{\ell-1}u_j|^{\frac{2N}{N-2}}-|\sum_{j=1}^{\ell}(-1)^j PU_j|^{\frac{2N}{N-2}}
+|\sum_{j=1}^{\ell-1}(-1)^j PU_j|^{\frac{2N}{N-2}}\],
\end{aligned}
\end{equation*}
by \eqref{12377} we have that
$$\widetilde{J}_\e (\mu_1,\ldots,\mu_k)=J_\e(\mu_1,\ldots,\mu_k)+\sum_{\ell=1}^k \widetilde{\Upsilon}_\ell$$
in view of
$$\sum_{i \geq \ell}\int_\Omega \[f (\sum_{j=1}^\ell u_j) -f (\sum_{j=1}^{\ell-1} u_j)\]
\phi_{i, \eps}\, dx =
\sum_{i=1}^k \int_\Omega f (\sum_{j=1}^i u_j)  \phi_{i, \eps}\, dx.$$
Since for $\ell\geq 2$
\begin{equation*}
\begin{aligned}
&(-1)^\ell \( \langle PU_\ell, \sum_{i=1}^{\ell-1}  \phi_{i, \eps}\rangle-\eps \int_\Omega  PU_\ell (\sum_{i=1}^{\ell-1} \phi_{i, \eps})\) \\
&= (-1)^\ell  \int_\Omega PU_\ell^{\frac{N+2}{N-2}} ( \sum_{i=1}^{\ell-1}  \phi_{i, \eps}) \, dx +O\(  | U_\ell^{\frac{N+2}{N-2}}-(PU_\ell)^{\frac{N+2}{N-2}}|_{\frac{2N}{N+2}}+\e |P U_\ell |_{\frac{2N}{N+2}}\)  \ \sum_{i=1}^{\ell-1} \|\phi_{i, \e}\| \\
&=(-1)^\ell \int_\Omega PU_\ell^{\frac{N+2}{N-2}} (\sum_{i=1}^{\ell-1}  \phi_{i, \eps}) \, dx+o\( (\frac{\mu_{\ell}}{\mu_{\ell-1}})^\Gamma +\eps \mu_\ell^2 \)
\end{aligned}
\end{equation*}
in view of \eqref{stimaphi1} and \eqref{114}-\eqref{112} with $\mu_1$ replaced by $\mu_\ell$, we have that
\begin{equation} \label{114pp}
\begin{aligned}
\widetilde{\Upsilon}_1 =- \frac{N-2}{2N}  \int_\Omega \widetilde{\upsilon}_1 \ dx+O(\|\mathcal E_1\|^2), \quad
\widetilde{\Upsilon}_\ell  =- \frac{N-2}{2N}  \int_\Omega \widetilde{\upsilon}_\ell \ dx+o\( (\frac{\mu_{\ell}}{\mu_{\ell-1}})^\Gamma +\eps \mu_\ell^2 \)
\end{aligned}
\end{equation}
for any $\ell=2,\ldots,k$, where
\begin{eqnarray*}
\widetilde{\upsilon}_\ell &=& |\sum_{j=1}^{\ell}u_j|^{\frac{2N}{N-2}}
-|\sum_{j=1}^{\ell-1}u_j|^{\frac{2N}{N-2}}
-|\sum_{j=1}^{\ell}(-1)^j PU_j|^{\frac{2N}{N-2}}
+|\sum_{j=1}^{\ell-1}(-1)^j PU_j|^{\frac{2N}{N-2}}\\
&&-\frac{2N}{N-2} \[f(\sum_{j=1}^\ell u_j)\phi_{\ell,\eps}+(-1)^\ell (PU_\ell)^{\frac{N+2}{N-2}} (\sum_{i=1}^{\ell-1} \phi_{i,\eps}) \].
\end{eqnarray*}
By \eqref{18355} and \eqref{18377} we have the expansion
\begin{equation} \label{13000}
\begin{aligned}
\widetilde{\upsilon}_\ell &= |\sum_{j=1}^\ell (-1)^j PU_j+\sum_{j=1}^{\ell-1} \phi_{j,\eps}|^{\frac{2N}{N-2}}
-|\sum_{j=1}^{\ell-1} u_j|^{\frac{2N}{N-2}}
-|\sum_{j=1}^{\ell} (-1)^j PU_j|^{\frac{2N}{N-2}} \\
&+|\sum_{j=1}^{\ell-1}(-1)^j PU_j|^{\frac{2N}{N-2}}
-\frac{2N}{N-2} (-1)^\ell (PU_\ell)^{\frac{N+2}{N-2}} (\sum_{i=1}^{\ell-1} \phi_{i,\eps})\\
&+ O\(|\phi_{\ell,\eps} |^{\frac{2N}{N-2}}+\sum_{j=1}^\ell (PU_j)^{\frac{4}{N-2}} \phi_{\ell,\eps}^2 +\sum_{j=1}^{\ell-1} |\phi_{j,\eps}|^{\frac{4}{N-2}} \phi_{\ell,\eps}^2\).
\end{aligned}
\end{equation}
We have that
\begin{equation} \label{1829}
 \widetilde{\Upsilon}_1=O(\|\mathcal E_1\|^2)=o(\mu_1^{2\Gamma})
\end{equation}
in view of \eqref{mujbis}-\eqref{muj2} and \eqref{114pp}.

\medskip \noindent Let us now discuss the case $\ell\geq 2$. Given $\mathcal A_h$ as in \eqref{anelli}, by \eqref{18355} and \eqref{18377}  for $h=0,\ldots,\ell-1$ we have 
\begin{equation}\label{14032}
\begin{aligned}
| \widetilde{\upsilon}_\ell |_{1,\mathcal A_h} & \leq c \Big| U_l^{\frac{2N}{N-2}}+U_l^2 \sum_{j=1}^{\ell-1}[U_j^{\frac{4}{N-2}}+|\phi_{j,\eps}|^{\frac{4}{N-2}}]
+U_\ell^{\frac{N+2}{N-2}}  \sum_{j=1}^{\ell-1}  |\phi_{j,\eps}|  \Big|_{1,\mathcal A_h}\\
&+c \Big| [f (\sum_{j=1}^{\ell-1} u_j)-f(\sum_{j=1}^{\ell-1} (-1)^j PU_j) ] U_\ell \Big|_{1,\mathcal A_h}
+O(\|\phi_{\ell,\eps}\|^2)\\
&\le 
c  \sum_{j=1}^{\ell-1} \Big|U_\ell |\phi_{j,\eps}|^{\frac{N+2}{N-2}} 
\Big|_{1,\mathcal A_h}
+ \sum_{i,j=1}^{\ell-1} \Big| U_\ell U_i^{\frac{4}{N-2}} |\phi_{j,\eps}| \Big|_{1,\mathcal A_h}+o\( (\frac{\mu_{\ell}}{\mu_{\ell-1}})^\Gamma +\eps \mu_\ell^2 \)
\end{aligned}
\end{equation}
and
\begin{equation}\label{14033}
\begin{aligned}
| \widetilde{\upsilon}_\ell |_{1,\mathcal A_\ell}  & \leq c  \Big|   \sum_{i=1}^{\ell} \sum_{j=1}^{\ell-1} U_i^{\frac{4}{N-2}}  \phi_{j,\eps}^2
+\sum_{j=1}^{\ell-1} |\phi_{j,\eps}|^{\frac{2N}{N-2}} 
+ \sum_{i,j=1}^{\ell-1} U_i^{\frac{N+2}{N-2}}  |\phi_{j,\eps}|
\Big|_{1,\mathcal A_\ell}\\
&+c \Big| [f (\sum_{j=1}^{\ell} (-1)^j PU_j )-f((-1)^\ell PU_\ell) ]  
\sum_{j=1}^{\ell-1} \phi_{j,\eps} \Big|_{1,\mathcal A_\ell}
+O(\|\phi_{\ell,\eps}\|^2)\\
&\le c  \Big| U_\ell^{\frac{4}{N-2}}  \sum_{j=1}^{\ell-1}   \phi_{j,\eps}^2
+\sum_{j=1}^{\ell-1} |\phi_{j,\eps}|^{\frac{2N}{N-2}} 
\Big|_{1,\mathcal A_\ell}+o\( (\frac{\mu_{\ell}}{\mu_{\ell-1}})^\Gamma +\eps \mu_\ell^2 \)
\end{aligned}
\end{equation}
in view of \eqref{stimaphi1}, \eqref{ok1}-\eqref{ok2222} and for any $i,j=1,\ldots,\ell-1$
\begin{eqnarray*}
&& \hspace{-0.6 cm}
| U_l^2 U_j^{\frac{4}{N-2}}  |_{1,\mathcal A_h}=O( |  U_j^{\frac{4}{N-2}} U_\ell |_{\frac{2N}{N+2},\mathcal A_h} |U_\ell |_{\frac{2N}{N-2},\mathcal A_h}),\   U_\ell^2 |\phi_{j,\eps}|^{\frac{4}{N-2}}+U_\ell^{\frac{N+2}{N-2}}|\phi_{j,\eps}|=O(U_\ell |\phi_{j,\eps}|^{\frac{N+2}{N-2}}+U_\ell^{\frac{2N}{N-2}}),\\
&& \hspace{-0.6 cm} U_i^{\frac{4}{N-2}} \phi_{j,\eps}^2+U_i^{\frac{N+2}{N-2}}  |\phi_{j,\eps}|=O(|\phi_{j,\eps}|^{\frac{2N}{N-2}}+U_i^{\frac{2N}{N-2}}), \
U_\ell^{\frac{4}{N-2}}  U_i \phi_{j,\eps} =O(
[U_\ell^{\frac{4}{N-2}}  U_i]^{\frac{2N}{N+2}}  +|\phi_{j,\eps} |^{\frac{2N}{N-2}}).
\end{eqnarray*}
Notice in the estimate \eqref{14032} we couple the first two and the second two terms in the expression \eqref{13000} of $\widetilde{\upsilon}_\ell$, while in the estimate \eqref{14033} the first/second term  is coupled with the third/fourth one in \eqref{13000}.

\medskip \noindent For $j=1,\ldots,\ell-1 $  there holds $\mathcal A_\ell \subset B_{ \rho \mu_j}(0)$
and by \eqref{stimalinfty} we deduce that 
\begin{equation} \label{15007}
\begin{aligned}
& \Big| U_\ell^{\frac{4}{N-2}}    \phi_{j,\eps}^2+|\phi_{j,\eps}|^{\frac{2N}{N-2}}  \Big|_{1,\mathcal A_\ell} \leq  \frac{c}{\mu_j^{2\Gamma}}  \int_{\mathcal A_\ell} \frac{U_\ell^{\frac{4}{N-2}}}{|x|^{2\beta_-}} \ dx
+\frac{c}{\mu_j^{\frac{2N}{N-2}\Gamma}}  \int_{\mathcal A_\ell}
\frac{dx}{|x|^{\frac{2N}{N-2}\beta_-}} \\
& \leq c (\frac{\mu_\ell}{\mu_j})^{2\Gamma}  \int_{\frac{\mathcal A_\ell}{\mu_\ell}} \frac{U^{\frac{4}{N-2}}}{|y|^{2\beta_-}} \ dy
+c  (\frac{\mu_\ell}{\mu_{\ell-1}})^{\frac{N}{N-2}\Gamma} 
=O \( (\frac{\mu_\ell}{\mu_{\ell-1}})^{\frac{N}{N-2}\Gamma} \log \frac{1}{\mu_\ell} \).
\end{aligned}
\end{equation}
For $h=0,\ldots,\ell-2$ by \eqref{stimaphi1} and \eqref{ok1} we deduce
\begin{equation} \label{17000bis}
\begin{aligned}
|| \phi_{j,\eps}|^{\frac{N+2}{N-2}} U_\ell  |_{1, \mathcal A_h} +|U_\ell U_i^{\frac{4}{N-2}} \phi_{j,\eps}  |_{1, \mathcal A_h} &
\le c \|\phi_{j,\eps}\|^{\frac{N+2}{N-2}}   | U_\ell |_{\frac{2N}{N-2}, \mathcal A_h}
+ c \|\phi_{j,\eps}\| |U_\ell |_{\frac{2N}{N-2},\mathcal A_h} \\
& =o \left( (\frac{\mu_\ell}{\sqrt{\mu_h \mu_{h+1}}})^{\Gamma} \right)= o\( (\frac{\mu_\ell}{\mu_{\ell-1}})^\Gamma \)
\end{aligned}
\end{equation}
for any $i,j=1,\ldots,\ell-1$. Splitting $\mathcal A_{\ell-1}$ as $\mathcal A'_{\ell-1}\cup \mathcal A''_{\ell-1}$, where $\mathcal A'_{\ell-1}=\mathcal A_{\ell-1}\cap B_{ \rho \mu_{\ell-1}}(0)$ and $\mathcal A''_{\ell-1}=\mathcal A_{\ell-1}\setminus B_{ \rho \mu_{\ell-1}}(0)$, by \eqref{stimaphi1}
and \eqref{stimalinfty} we get that
\begin{equation} \label{17002bis}
\begin{aligned}
& | |\phi_{j,\eps}|^{\frac{N+2}{N-2}} U_\ell |_{1, \mathcal A_{\ell-1}}  
+|U_\ell U_i^{\frac{4}{N-2}} \phi_{j,\eps}  |_{1, \mathcal A_{\ell-1}}\\
& \le  c \|\phi_{j,\eps}\|^{\frac{2N}{(N-2)(N-1)}}   | \phi_{j,\eps}^{\frac{N+1}{N-1}} U_\ell |_{\frac{N-1}{N-2}, \mathcal A'_{\ell-1}}
+ c \|\phi_{j,\eps}\|^{\frac{2N}{(N-2)(5N-9)}} |  U_\ell U_i^{\frac{4}{N-2}} \phi_{j,\eps}^{\frac{5N^2-21N+18}{(N-2)(5N-9)}}  |_{\frac{5N-9}{5N-10},\mathcal A'_{\ell-1}}\\
&+c \[ \|\phi_{j,\eps} \|^{\frac{N+2}{N-2}}+\|\phi_{j,\eps}\| \] \left| U_\ell \right|_{\frac{2N}{N-2}, \mathcal A''_{\ell-1} }\\
&=o\[ \frac{\mu_\ell^\Gamma}{\mu_j^{\frac{N+1}{N-1}\Gamma}} (\int_{\mathcal A'_{\ell-1}} \frac{dx}{|x|^{N-\frac{2\Gamma}{N-2}}})^{\frac{N-2}{N-1}}+
\frac{\mu_\ell^\Gamma}{\mu_i^{\frac{4\Gamma}{N-2}} \mu_j^{\frac{5N^2-21N+18}{(N-2)(5N-9)}\Gamma}  }  (\int_{\mathcal A'_{\ell-1}} \frac{dx}{|x|^{N- \frac{18\Gamma}{5(N-2)} }})^{\frac{5N-10}{5N-9}}\]\\
&+o\(  (\frac{\mu_\ell}{\mu_{\ell-1}})^{\Gamma} \)=o\(  (\frac{\mu_\ell}{\mu_{\ell-1}})^{\Gamma} \)
\end{aligned}
\end{equation}
for any $i,j=1,\ldots,\ell-1$ in view of \eqref{ok31}. Inserting \eqref{17000bis}-\eqref{17002bis} into \eqref{14032} and \eqref{15007} into \eqref{14033} we deduce that $|\widetilde{\upsilon}_\ell |_{1,\mathcal A_h}=o\(  (\frac{\mu_\ell}{\mu_{\ell-1}})^{\Gamma} +\eps \mu_\ell^2 \)$ for any $h=0,\ldots, \ell$ and then
\begin{equation} \label{1830}
\widetilde{\Upsilon}_\ell =o\(  (\frac{\mu_\ell}{\mu_{\ell-1}})^{\Gamma}\)
\end{equation}
for any $\ell \geq 2$ in view of \eqref{mujbis}-\eqref{muj2} and \eqref{114pp}. Thanks to \eqref{1829} and \eqref{1830} we have established that $\widetilde{\Upsilon}_\ell$ satisfies the same estimate as $\Upsilon_\ell$, $\ell =1,\ldots,k$.

\medskip \noindent To conclude the proof of Proposition \ref{prob-rido}, let us show that, if $(d_1,\dots,d_k)$ is a critical point of $\widetilde J_\e$, then $ \displaystyle \sum_{\ell=1}^k (-1)^\ell PU_\ell +\Phi_{ \e}$ is a critical point of the functional $J_\e.$
 Assume that
$$
0=\partial_{d_h}\widetilde J_\e(d_1,\dots,d_k)=\nabla J_\e\( \sum_{\ell=1}^k  (-1)^\ell PU_\ell +\Phi_{\e} \) \[(-1)^h \partial _{d_h}  PU_h+\partial_{d_h} \Phi_\eps \]$$
for any $h=1,\dots,k$. Since
$$\nabla J_\e\( \sum_{\ell=1}^k (-1)^\ell PU_\ell +\Phi_{\e} \) =\sum\limits_{j=1}^k \lambda_{j,\e}P Z _j,$$
we get that
\begin{equation} \label{15347} 0=\sum\limits_{j=1}^k \lambda_{j,\e}\< P Z _j,(-1)^h \partial _{d_h}  PU_h+\partial_{d_h} \Phi_\e \>
\end{equation}
for any $h=1,\dots,k$. 
Since by \eqref{czero} there hold
$$\|P Z _h\|^2=c_0+o(1),\qquad \<P Z _j, PZ _h\>=o(1)\ \forall \ j\not=h,$$
we have that
$$\<P Z _j, \partial_{d_h}\Phi_\e \>=\sum_{\ell=h}^k \<P Z _j, \partial_{d_h}\phi_{\ell,\e} \>=
O\( \sum_{\ell=h}^k \|PZ _j\| \ \|\partial_{d_h}\phi_{\ell,\eps}\|\)=o\(1\).$$
in view of \eqref{stimaphi2}. Since
$$\partial _{d_h} PU_h=-\Gamma \alpha_N P Z _h \times \left\{ \begin{array}{ll}-\frac{1}{\e}&\hbox{if }\Gamma=1\\
{1\over d_h} &\hbox{if }\Gamma>1, \end{array} \right.$$
by \eqref{15347} we deduce that $\lambda_{j,\e}=0$ for any $j=1,\ldots,k$, or equivalently
$$\nabla J_\e\( \sum_{\ell=1}^k (-1)^\ell PU_\ell +\Phi_{\e} \) =0.$$
Then $ \displaystyle \sum_{\ell=1}^k (-1)^\ell PU_\ell +\Phi_{ \e}$ is a critical point of the functional $J_\e$ and the proof of Proposition \ref{prob-rido} is complete.

\end{document}